\documentclass{amsart}
\usepackage[margin=1.8cm]{geometry}

\usepackage[all]{xy}
\usepackage{graphicx}
\usepackage{indentfirst}
\usepackage{bm}
\usepackage{amsthm}
\usepackage{mathrsfs}
\usepackage{latexsym}
\usepackage{amsmath}
\usepackage{amssymb}
\usepackage{color}
\usepackage{hyperref}
\usepackage{microtype}
\usepackage{tikz}
\usepackage{enumerate}
\usepackage{comment}
\usepackage{braket}
\usepackage{manfnt}
\usepackage{hyperref}
\usepackage{bookmark}

\usepackage{color}
\usetikzlibrary{calc}
\usetikzlibrary{shapes.geometric}
\usetikzlibrary{arrows}
\usetikzlibrary{decorations.pathmorphing}
\usetikzlibrary{decorations.text}
\usetikzlibrary{decorations.shapes}
\usetikzlibrary{decorations.fractals}
\usetikzlibrary{decorations.footprints}

\begin{document}

\newtheorem{theorem}{Theorem}[section]
\newtheorem{main}{Main Theorem}
\newtheorem{proposition}[theorem]{Proposition}
\newtheorem{corollary}[theorem]{Corollary}
\newtheorem{definition}[theorem]{Definition}
\newtheorem{lemma}[theorem]{Lemma}
\newtheorem{example}[theorem]{Example}
\newtheorem{remark}[theorem]{Remark}
\newtheorem{question}[theorem]{Question}
\newtheorem{conjecture}[theorem]{Conjecture}
\newtheorem{fact}[theorem]{Fact}
\newtheorem*{ac}{Acknowledgements}

\newcommand{\FS}{\mathfrak{F}_s}
\newcommand{\fF}{\mathfrak{F}}
\newcommand{\tfF}{\widetilde{\mathfrak{F}}}
\newcommand{\bZ}{\mathbb{Z}}
\newcommand{\bN}{\mathbb{N}}
\newcommand{\bC}{\mathbb{C}}
\newcommand{\bR}{\mathbb{R}}
\newcommand{\bRp}{\mathbb{R}_{\geq 0}}
\newcommand{\bZp}{\mathbb{Z}_{\geq 0}}
\newcommand{\fA}{\mathfrak{A}} 
\newcommand{\fB}{\mathfrak{B}} 
\newcommand{\bA}{\mathcal{A}} 
\newcommand{\bB}{\mathcal{B}} 
\newcommand{\fr}{r_\mathfrak{A}}
\newcommand{\cR}{\mathcal{R}}
\newcommand{\cS}{\mathcal{S}}
\newcommand{\cN}{\mathcal{N}}
\newcommand{\cM}{\mathcal{M}}
\newcommand{\sP}{\mathscr{P}}
\newcommand{\sC}{\mathscr{C}}
\newcommand{\sD}{\mathscr{D}}
\newcommand{\tB}{\widetilde{B}}
\newcommand{\Hil}{\mathcal{H}} 
\newcommand{\conv}{*} 
\newcommand{\FPdim}{\mathrm{FPdim}}
\newcommand{\PSL}{\mathrm{PSL}}
\newcommand{\Rep}{\mathrm{Rep}}

\title[Fusion Bialgebras and Fourier Analysis]{Fusion Bialgebras and Fourier Analysis \\ \tiny Analytic obstructions for unitary categorification}

\author{Zhengwei Liu}
\address{Z. Liu, Yau Mathematical Sciences Center and Department of Mathematics, Tsinghua University, Beijing, 100084, China}
\address{Z. Liu, Beijing Institute of Mathematical Sciences and Applications, Huairou District, Beijing, 101408, China}
\email{liuzhengwei@math.tsinghua.edu.cn}

\author{Sebastien Palcoux}
\address{S. Palcoux, Beijing Institute of Mathematical Sciences and Applications, Huairou District, Beijing, 101408, China}
\email{sebastienpalcoux@gmail.com}
\urladdr{https://sites.google.com/view/sebastienpalcoux}

\author{Jinsong Wu}
\address{J. Wu, Institute for Advanced Study in Mathematics, Harbin Institute of Technology, Harbin, 150001, China}
\email{wjs@hit.edu.cn}

\begin{abstract}
We introduce fusion bialgebras and their duals and systematically study their Fourier analysis.
As an application, we discover new efficient analytic obstructions on the unitary categorification of fusion rings.
We prove the Hausdorff-Young inequality, uncertainty principles for fusion bialgebras and their duals.
We show that the Schur product property, Young's inequality and the sum-set estimate hold for fusion bialgebras, but not always on their duals.
If the fusion ring is the Grothendieck ring of a unitary fusion category, then these inequalities hold on the duals. Therefore, these inequalities are analytic obstructions of categorification.
We classify simple integral fusion rings of Frobenius type up to rank 8 and of Frobenius-Perron dimension less than 4080. We find 34 ones, 4 of which are group-like and 28 of which can be eliminated by applying the Schur product property on the dual.
In general, these inequalities are obstructions to subfactorize fusion bialgebras.
\end{abstract}


\maketitle

\vspace*{-.2cm}
\section{Introduction}


Lusztig introduced fusion rings in \cite{Lus87}. Etingof, Nikshych and Ostrik studied fusion categories \cite{ENO05} as a categorification of fusion rings, see also \cite{EtiKho95,EGNO15}. A central question is whether a fusion ring can be unitarily categorified, namely it is the Grothendieck ring of a unitary fusion category.

Jones introduced subfactor planar algebras as an axiomatization of the standard invariant of a subfactor in  \cite{Jon99}.
Planar algebras and fusion categories have close connections. There are various ways to construct one from the other.
For example, if $N \subset N\rtimes G$ is the group crossed product subfactor of a finite group $G$, then the $2$-box space $\mathscr{P}_{2,+}$ of its planar algebra captures the unitary fusion category $Vec(G)$ and its Fourier dual $\mathscr{P}_{2,-}$ captures the unitary fusion category $\mathrm{Rep}(G)$.
The Grothendieck ring of a unitary fusion category can be realized as the 2-box space of a subfactor planar algebra using the quantum double construction, such that the ring multiplication is implemented by the convolution of 2-boxes \cite{Mug03II,Liu19}.

Recently, Jiang, the first author and the third author formalized and proved numbers of quantum inequalities for subfactor planar algebras \cite{Liuex,JLW16,JLWscm,LiuWuscm} inspired by Fourier analysis.
These inequalities automatically hold for the Grothendieck rings of unitary fusion categories $\sC$ as explained in \cite{Liu19}, through the well-known quantum double construction from unitary fusion categories to subfactors, see e.g. \cite{Mug03II}.
Moreover, the Fourier dual of a subfactor is still a subfactor. So these inequalities also hold on the Fourier dual of the Grothendieck ring, which can be regarded as representations of the Grothendieck ring.

This paper is inspired by three questions:
\begin{itemize}
\item Vaughan Jones \cite{Jones15}: What are the applications of these inequalities on subfactors to other areas?
\item Zhenghan Wang \cite{Wang17}: Are these inequalities obstructions of categorification?
\item Pavel Etingof \cite{Etingof19}: Do the inequalities on Grothendieck rings hold on fusion rings?
\end{itemize}

In this paper, we prove that these quantum inequalities on subfactor planar algebras hold on fusion rings and partially, but not all, on the Fourier dual of fusion rings.
Therefore, the inequalities that fail on the dual of the fusion rings are new analytic obstructions for unitary categorification of fusion rings. For examples, the quantum Schur product theorem [\cite{Liuex}, Theorem 4.1] holds on the Fourier dual of Grothendieck rings, but not on the Fourier dual of fusion rings. It turns out to be a surprisingly efficient obstruction of unitary categorification of fusion rings. Moreover, it is easy to check the Schur product property on the dual of a commutative fusion ring in practice. In this way, we find many fusion rings which admit no unitary categorification, due to the Schur product property, and which cannot be ruled out by previous obstructions.

In \S\ref{Sec: Fusion Bialgebras}, we introduce fusion bialgebras as a generalization of fusion rings and their duals over the field $\mathbb{C}$. The definition of fusion bialgebras is inspired by the 2-box spaces $\mathscr{P}_{2,\pm}$ of subfactor planar algebras.
We show that if $\mathscr{P}_{2,+}$ is commutative, then it is a fusion bialgebra.
If a fusion bialgebra arises in this way, then we say that it is subfactorizable. We classify fusion bialgebras up to dimension three.
The classification of the two dimensional subfactorizable fusion bialgebras is equivalent to the remarkable classification of the Jones index of subfactors \cite{Jon83}. It remains challenging to classify three dimensional subfactorizable fusion bialgebras.

In \S\ref{Sec: Schur Product Property}-\S\ref{Sec: Fusion Subalgebras and Bishifts of Biprojections}, we systematically study quantum Fourier analysis on fusion bialgebras.
We show that the Hausdorff-Young inequalities, uncertainty principles hold for fusion bialgebras and their duals;
Young's inequalities and the sum-set estimate hold for fusion bialgebras, but not necessarily on their duals.
We characterize their extremizers in \S\ref{Sec: Fusion Subalgebras and Bishifts of Biprojections}.
In fact, for the dual of a fusion bialgebra, Young's inequality implies Schur product property, and Schur product property implies the sum-set estimate.
Therefore, Young's inequality is also an obstruction to unitary categorify a fusion ring or to subfactorize a fusion bialgebra, and the sum-set estimate is a potential obstruction.
It is worth mentioning that the Schur product property (or Young's inequality) holds on arbitrary $n$-box space of the Temperley-Lieb-Jones planar algebra if and only if it is a subfactor planar algebras, namely the circle parameter is the square root of the Jones index \cite{Jon83}.

In \S\ref{Sec: Application}, we reformulate Schur product property (on the dual) in terms of irreducible representations of the fusion ring/algebra, especially in terms of the character table for the commutative case. In the family of fusion algebras of rank 3 with every object self-dual, we observe that about $30\%$ of over 10000 samples do not have the Schur product property (on the dual). So they cannot be subfactorized. We consider families of rank 4 or 5 fusion rings, and we compare (visually) Schur product criterion and Ostrik's criterion \cite[Theorem 2.21]{Ost15}.

Next, we give a classification of simple integral fusion rings of Frobenius type with the following bounds of Frobenius-Perron dimensions (with  $\FPdim \neq p^aq^b, pqr$, by \cite{ENO11}).
$$\begin{array}{c|c|c|c|c|c|c|c}
\text{rank} & \le 5 & 6     & 7   & 8   & 9 & 10 &  \text{all}  \\ \hline
\FPdim <  & 1000000 & 150000 & 15000 & 4080 & 504 & 240 & 132
\end{array}$$
First, given a Frobenius-Perron dimension, we classify all possible types (the list of dimensions of the ``simple objects''). Secondly, we classify the fusion matrices for a given type.
We derive several inequalities from Fourier analysis on fusion rings which bound the fusion coefficients using the dimensions. These inequalities are efficient in the second step of the classification. For some specific types, the use of these inequalities reduced drastically the computation time (from 50 hours to 5 seconds).
We end up with 34 simple integral fusion rings in the classification (all commutative), 4 of which are group-like and 28 of which cannot be unitarily categorified by showing that the Schur product property (on the dual) does not hold. It remains 2 ones. None of these 28+2 ones can be ruled out by already known methods.
\begin{question}
Do the remaining two fusion rings admit a unitary categorification?
\end{question}
\noindent It has two motivations, first the categorification of a simple integral non group-like fusion ring would be non weakly-group-theoretical and so would provide a positive answer to Etingof-Nikshych-Ostrik \cite[Question 2]{ENO11}, next there is no known non group-like examples of irreducible finite index maximal depth $2$ subfactor \cite[Problem 4.12]{Pal18}, but its fusion category would be unitary, simple, integral (and of Frobenius type, assuming Kaplansky's 6th conjecture \cite{Kap75}).

In summary, Fourier analysis on subfactors provides efficient analytic obstructions of unitary categorification or of subfactorization.



\vspace*{-0.7cm}
\tableofcontents

\begin{ac}
Zhengwei Liu would like to thank Pavel Etingof, Vaughan Jones and Zhenghan Wang for inspiring questions and valuable discussions. We thank David Penneys for sharing the two families of rank $4$ fusion rings and the hospitality of the 2019 Oberwolfach workshop Subfactors and Applications. The computation was supported by Tsinghua National Laboratory for Information Science and Technology, and we thank Lin Jiao from this laboratory for her help.
Zhengwei Liu and Sebastien Palcoux were supported by Tsinghua University (Grant no. 100301004) and Beijing Institute of Mathematical Sciences and Applications.
Zhengwei Liu was supported by NKPs (Grant no. 2020YFA0713000).
Jinsong Wu was supported by NSFC (Grant no. 11771413 and 12031004).
Zhengwei Liu and Jinsong Wu would like to thank Harvard University for the hospitality and the support of Templeton Religion Trust (Grant no. TRT 0159).
\end{ac}

\section{Fusion Bialgebras}\label{Sec: Fusion Bialgebras}

In this section, we introduce fusion bialgebras which capture fusion algebras of fusion rings over $\mathbb{C}$ and their duals, namely representations.
The definition of fusion bialgebras is motivated by a connection between subfactor planar algebras and unitary fusion categories based on the quantum double construction. Its algebraic aspects have been discussed in \cite{Liu19}. In this paper, we investigate its analytic aspects and study Fourier analysis on fusion bialgebras.

The fusion bialgebra has a second multiplication $\diamond$ and involution $\#$ on the fusion algebra.
Several basic results on fusion rings, see for example \cite{EGNO15}, can be generalized to fusion bialgebras.
Many examples of fusion bialgebras come from subfactor theory, and we say that they can be subfactorized.
It is natural to ask whether a fusion bialgebra can be subfactorized.
The question for the two dimensional case is equivalent to the classification of the Jones index.
If a fusion ring has a unitary categorification, then the corresponding fusion bialgebra has a subfactorization.
We introduce analytic obstructions of subfactorization from Fourier analysis on subfactors, so they are also obstructions of unitary categorification. We discuss their applications in \S\ref{Sec: Application}.

\subsection{Definitions}

Let $\bN=\bZp$ be the set of all natural numbers.
Let $\bRp$ be the set of non-negative real numbers.

\begin{definition}\label{def:bB}
Let $\bB$ be a unital *-algebra over the complex field $\bC$.
We say $\bB$ has a $\bRp$-basis $B=\{x_1=1_{\bB}, x_2, \ldots, x_m\}, m\in\mathbb{Z}_{\geq 1}$, if
\begin{enumerate}[(1)]
\item $\{x_1, \ldots, x_m\}$ is a linear basis over $\bC$;
\item $\displaystyle x_jx_k=\sum_{s=1}^m N_{j,k}^s x_s$, $N_{j,k}^s\in \bRp$;
\item there exists an involution $*$ on $\{1, 2, \ldots, m\}$ such that $x_k^*:=x_{k^*}$ and $N_{j, k}^1=\delta_{j, k^*}$.
\end{enumerate}
\end{definition}

We write the identity $1_{\bB}$ as $1$ for short, if there is no confusion.
When $N_{j,k}^s \in \mathbb{N}$, $B$ gives a fusion ring, and $\bB$ is called a fusion algebra.
The *-algebra $\bB$ with a $\bRp$-basis $B$ can be considered as a fusion algebra over the field $\bC$.

\begin{definition}
For a unital *-algebra $\bB$ with a $\bRp$-basis $B$, we define a linear functional $\tau: \bB \to \bC$ by $\tau(x_j)=\delta_{j,1}$.
\end{definition}

Then $\tau(x_jx_k)=N_{j,k}^1=\delta_{j,k^*}$ and $\tau(xy)=\tau(yx)$ for any $x, y\in \bB$.
Moreover
\begin{equation}\label{Equ: rotation}
N_{j,k}^{s^*}=\tau(x_jx_k x_s)=\tau(x_sx_jx_k)=N_{s,j}^{k^*}=\tau(x_kx_sx_j)=N_{k,s}^{j^*}.
\end{equation}
Note that $x_{k^*}x_{j^*}=(x_jx_k)^*$.
We obtain Frobenius reciprocity
\begin{equation}\label{Equ: star}
N_{j,k}^{s}=N_{k^*, j^*}^{s^*}=N_{j^*,s}^{k}.
\end{equation}
\noindent Therefore $\tau$ is a faithful tracial state on the *-algebra $\bB$.
Following the Gelfand-Naimark-Segal construction,
we obtain a Hilbert space $\Hil=L^2(\bB, \tau)$ with the inner product
$$\langle x , y \rangle=\tau(y^*x),$$
and a unital *-representation $\pi$ of the $*$-algebra $\bB$ on $\Hil$.
Moreover $B$ forms an orthonormal basis of $\Hil$. On this basis, we obtain a representation $\pi_B: \bB \to M_{m}(\bC)$.
In particular,
$$\pi_B(x_j)_{k,s}=N_{j,k}^{s}.$$
We denote the matrix $\pi_B(x_j)$ by $L_j$. Then
$$L_jL_k=\sum_{s=1}^m N_{j,k}^s L_s,$$
and
$$L_j^*=L_{j^*}.$$

\begin{remark}
Under the Gelfand-Naimark-Segal construction, the $*$-algebra $\bB$ forms a $C^*$-algebra, which is also a von Neumann algebra. In this paper, we only consider the finite dimensional case, so we do not distinguish $C^*$-algebras and von Neumann algebras.
\end{remark}

\begin{definition}
For a unital *-algebra $\bB$ with a $\bRp$-basis $B$, we define a linear functional $d: \bB \to \bC$ by setting $d(x_j)$ to be the operator norm $\|L_j\|_\infty$ of $L_j$, below denoted by $\|x_j\|_{\infty, \bB}$.
\end{definition}

Recall the Perron-Frobenius theorem for matrices:
\begin{theorem}[Perron-Frobenius Theorem, \cite{ENO05} Theorem 8.1]\label{Thm:FP}
Let $A$ be a square matrix with nonnegative entries.
\begin{enumerate}[(1)]
\item $A$ has a nonnegative real eigenvalue.
The largest nonnegative real eigenvalue $\lambda(A)$ of $A$ dominates absolute values of all other eigenvalues of $A$.
\item If $A$ has strictly positive entries then $\lambda(A)$ is a simple positive eigenvalue, and the corresponding eigenvector can be normalized to have strictly positive entries.
\item If $A$ has an eigenvector $f$ with strictly positive entries, then the corresponding eigenvalue is $\lambda(A)$.
\end{enumerate}
\end{theorem}

\begin{proposition}\label{Prop: dim d}
Let $\bB$ be a unital *-algebra with a $\bRp$-basis $B$. Then
$$d(x_j)d(x_k)=\sum_{s=1}^m N_{j,k}^s d(x_s),\quad d(x_j)=d(x_{j^*})\geq 1.$$
\end{proposition}

\begin{proof}
The right multiplication of $x_j$ on the orthonormal basis $B$ defines a matrix $R_j$. Then $R=\sum_{j=1}^m R_j$ has strictly positive entries.
Let $v=\sum_{j=1}^m \lambda_j x_j$ be the simple positive eigenvector of the right action $R$.
By Theorem~\ref{Thm:FP}, we can normalize $v$, such that $\lambda_1=1$ and $\lambda_j >0$.
As $L_j v$ is also a positive eigenvector, we have that $L_j v =\|L_j\|_\infty v=d(x_j) v$ by Theorem~\ref{Thm:FP}.
Since $L_jL_kv=d(x_k)L_jv=d(x_k)d(x_j)v$, we obtain that
$$d(x_j)d(x_k)=\sum_{s=1}^m N_{j,k}^s d(x_s).$$
Note that $\sum_{j=1}^m d(x_j)x_j$ is an eigenvector for $R$ by the equation above, we see that $\lambda_k=d(x_k)$ for any $1\leq k\leq m$.

Note that $L_j^*=L_{j^*}$, we have $d(x_j)=d(x_{j^*})$, and
$d(x_j)^2=d(x_j)d(x_{j^*}) \geq 1.$
Finally, we see that $d(x_j)\geq 1$.
\end{proof}

\begin{definition}[An alternative C$^*$-algebra $\bA$]\label{def:bA}
We define an abelian C$^*$-algebra $\bA$ with the basis $B$, a multiplication $\diamond$ and an involution $\#$,
$$x_j \diamond x_k =\delta_{j,k} d(x_j)^{-1} x_j,$$
$$( x_j)^{\#}= x_j.$$
The C$^*$-norm on $\bA$ is given by $\displaystyle \|x\|_{\infty, \bA}=\max_{1\leq j\leq m} \frac{|d(x\diamond x_j)|}{d(x_j)}$ for any $x\in \bA$.
\end{definition}

\begin{proposition}
The linear functional $d$ is a faithful state on $\bA$.
\end{proposition}

\begin{proof}
Note that $\{d(x_j) x_j\}$ are orthogonal minimal projections of $\bA$.  By Proposition~\ref{Prop: dim d}, $d(x_j) \geq 1$, so $d$ is faithful.
\end{proof}

\begin{definition}
For any $1\leq t\leq \infty$, the $t$-norms on $\bA$ and $\bB$ are defined as follows:
\begin{align*}
\|x\|_{t, \bA}=d(|x|^t)^{1/t}, \ x\in \bA, \quad \|x\|_{t, \bB}=\tau(|x|^t)^{1/t},\  x\in \bB, |x|=(x^*x)^{1/2} \quad 1\leq t<\infty
\end{align*}
and
\begin{align*}
\|x\|_{\infty, \bA}=\max_{1\leq j\leq m} \frac{|d(x\diamond x_j)|}{d(x_j)}, \quad \|x\|_{\infty, \bB}=\sup_{\|y\|_{2, \bB}=1}\|xy\|_{2, \bB},
\end{align*}
\end{definition}

\begin{remark}
For any $x\in \bA$, $|x|=(x^\#\diamond x)^{1/2}$.
\end{remark}

\begin{definition}\label{Fourier transform}[Fourier transform]
Let $\bA$, $\bB$ be *-algebra with the same basis $B$, but different multiplications and involutions defined in Definition~\ref{def:bB} and~\ref{def:bA}.
The Fourier transform $\fF:\bA \to \bB$ is a linear map defined by
$$\fF(x_j)=x_j, ~\forall  j.$$
\end{definition}

Note that $\fF$ is a bijection.

\begin{proposition}[Plancherel's formula]\label{prop:planch}
The Fourier transform $\fF: \bA \to \bB$ is a unitary transformation:
$$\| \fF(x) \|_{2,\bB} = \|x\|_{2,\bA},$$
i.e.
$$\tau(\fF(x)^*\fF(x))=d(x^{\#}\diamond x).$$
\end{proposition}
\begin{proof}
We only have to check the equation for the basis $B$.
For any $1\leq j,k \leq m$, we have
\begin{align*}
d(x_j \diamond x_k)= \delta_{j, k} d(x_k)^{-1}d(x_k)=\delta_{j, k}=\tau(x_{j^*}x_k)= \tau(x_j^*x_k).
\end{align*}
Then the proposition is true.
\end{proof}

Under the Fourier transform, the multiplication on $\bB$ induces the convolution on $\bA$. We denote the convolution of $x, y \in \bA$ by $$x \conv y:= \fF^{-1}(\fF(x)\fF(y)).$$
The C$^*$-algebras $\bA$ and $\bB$ share the same vector spaces, but have different multiplications, convolutions and traces.
These traces are non-commutative analogues of measures.


We axiomatize the quintuple $(\bA, \bB, \fF, d, \tau)$ as a \emph{fusion bialgebra} in the following definition. To distinguish the multiplications and convolutions on $\bA$ and $\bB$, we keep the notations as above.

\begin{definition}[Fusion bialgebras]\label{Def: Fusion bialgebras}
Suppose $\bA$ and $\bB$ are two finite dimensional C$^*$-algebras with faithful traces $d$ and $\tau$ respectively, $\bA$ is commutative, and $\fF:\bA \to \bB$ is a unitary transformation preserving $2$-norms (i.e. $\tau(\fF(x)^*\fF(y))=d(x^\#\diamond y)$ for any $x, y\in \bA$).
We call the quintuple $(\bA, \bB, \fF, d, \tau)$ a fusion bialgebra, if the following conditions hold:
\begin{itemize}
\item[(1)]Schur Product: For operators $x,y \geq 0$ in $\bA$, $x \conv y := \fF^{-1}(\fF(x)\fF(y)) \geq 0$ in $\bA$.
\item[(2)]Modular Conjugation: The map $J(x):=\fF^{-1}(\fF(x)^*)$ is an anti-linear, *-isomorphism on $\bA$.
\item[(3)]Jones Projection: The operator $\fF^{-1}(1)$ is a positive multiple of a minimal projection in $\bA$.
\end{itemize}
Furthermore, if $\fF^{-1}(1)$ is a minimal projection and $d(\fF^{-1}(1))=1$, then we call the fusion bialgebra canonical.
\end{definition}


\begin{remark}\label{Rem: Equivalent Definition of Fusion Bialgebras}
One can reformulate the definition of fusion bialgebras using the quintuple $(\bA, \conv, J, d, \tau)$.
\end{remark}

\begin{remark}
We show that subfactors provide fruitful fusion bialgebras in \S\ref{Sec: example}.
One can compare the three conditions in Definition~\ref{Def: Fusion bialgebras} with the corresponding concepts in subfactor theory.
\end{remark}

\begin{proposition}[Gauge transformation]
Given a fusion bialgebra $(\bA, \bB, \fF, d, \tau)$, then $(\bA, \bB, \lambda_1^{\frac{1}{2}} \lambda_2^{-\frac{1}{2}}\fF, \lambda_1 d, \lambda_2 \tau)$ is also a fusion bialgebra, with $\lambda_1, \lambda_2 > 0$.
Therefore, any fusion bialgebra is equivalent to a canonical one up to a gauge transformation.
\end{proposition}

\begin{proof}
It follows from the definition of the fusion bialgebra in Definition~\ref{Def: Fusion bialgebras}.
\end{proof}

\begin{theorem}\label{Thm: extension}
If $(\bA, \bB, \fF, d, \tau)$ is a fusion bialgebra, then $\bB$ has a unique $\bRp$-basis $B=\{x_1=1, x_2, \ldots, x_m\}$, such that $\fF^{-1}(x_j)$ are multiples of minimal projections of $\bA$.
Moreover, $B$ is invariant under the gauge transformation.
Conversely, any C$^*$-algebra $\bB$ with a $\bRp$-basis $B$ can be  extended to a canonical fusion bialgebra, such that $\fF^{-1}(x_j)$ are multiples of minimal projections of $\bA$.
\end{theorem}

\begin{proof}
By the above arguments, if a C$^*$-algebra $\bB$ has a $\bRp$-basis $B=\{x_1=1, x_2, \ldots, x_m\}$,  then we obtain a canonical fusion bialgebra $(\bA, \bB, \fF, d, \tau)$.

On the other hand,
suppose $(\bA, \bB, \fF, d, \tau)$ is a fusion bialgebra.
Let $P_j$, $1\leq j \leq m$, be the minimal projections of $\bA$, and $\fF^{-1}(1)=\delta_{\bB} P_1$, for some $\delta_{\bB}>0$.
The modular conjugation $J$ is a *-isomorphism, so $J(P_j)=P_{j^*}$, for some $1\leq j^* \leq m$. Then $\fF(P_j)=\fF(P_{j^*})^*$ and $J(P_1)=P_1$.
Moreover,
\begin{align*}
d(P_j)&=d(P_j^\# \diamond P_j) =\tau( \fF(P_j)^* \fF(P_j) )\\
&=\tau( \fF(P_{j^*})^* \fF(P_{j^*})) =d(P_{j^*}) \;.
\end{align*}
By the Schur Product property,
\begin{align}\label{Equ: jks1}
P_j \conv P_k=\sum_{s=1}^m \tilde{N}_{j,k}^s P_s \;,
\end{align}
for some $\tilde{N}_{j,k}^s \in \bRp$.
Since the functional $d$ is faithful, $d(P_j)>0$.
Taking the inner product with $P_1$ on both sides of Equation \eqref{Equ: jks1}, we have that
\begin{align*}
\tilde{N}_{j,k^*}^1 &=\frac{d( P_1 \diamond (P_j \conv P_{k^*}) )}{ d(P_1)}
= \frac{\tau(  \fF(P_1)^* \fF(P_j \conv P_{k^*})  )}{ d(P_1)}\\
&= \frac{1}{d(P_1) \delta_{\bB}} \tau( \fF(P_j \conv P_{k^*}) )
= \frac{1}{d(P_1) \delta_{\bB}} \tau( \fF(P_j) \fF(P_{k^*}) )\\
&= \frac{1}{d(P_1) \delta_{\bB}} \tau( \fF(P_j)  \fF(P_{k})^*)
= \frac{1}{d(P_1) \delta_{\bB}} d( P_j \diamond P_{k}) \\
&=  \frac{d(P_j) \delta_{j,k}}{d(P_1) \delta_{\bB}} \;.
\end{align*}
In particular, $\tilde{N}_{1,1}^1=\delta_{\bB}^{-1}$.
Take
\begin{align*}
x_j&= \delta_{\bB}^{\frac{1}{2}} (\tilde{N}_{j,j^*}^1)^{-\frac{1}{2}} \fF(P_j) \;, \\
N_{j,k}^s&=\delta_{\bB}^{\frac{1}{2}}  (\tilde{N}_{j,j^*}^1)^{-\frac{1}{2}}(\tilde{N}_{k,k^*}^1)^{-\frac{1}{2}}(\tilde{N}_{s,s^*}^1)^{\frac{1}{2}}\tilde{N}_{j,k}^s \;.
\end{align*}
Then
\begin{align*}
x_j  x_k=\sum_{s=1}^m N_{j,k}^s x_s,  \quad x_j^*=x_{j^*} \;, \quad N_{j,k}^s  \geq 0  \quad N_{j,k}^1  = \delta_{j,k^*} \;.
\end{align*}
Therefore, $\{ x_j \}_{1\leq j\leq m}$ forms a $\bRp$-basis of $\bB$.
Moreover, it is the unique $\bRp$-basis of $\bB$ such that $\fF^{-1}(x_j)$ are positive multiples of minimal projections in $\bA$.

Furthermore, applying the gauge transformation, we obtain a canonical fusion bialgebra $$(\bA, \bB, \check{\fF}, \check{d}, \check{\tau})=\left(\bA, \bB, \delta_{\bB} \fF,  \frac{d}{d(P_1)}, \frac{\tau}{d(P_1)  \delta_{\bB}^2} \right).$$
In this fusion bialgebra, the minimal projections in $\bA$ are still $P_j$, $1\leq j \leq m$. Their convolution becomes
\begin{align*}
P_j \conv P_k=\delta_{\bB} \sum_{s=1}^m \tilde{N}_{j,k}^s P_s \;.
\end{align*}
The corresponding $x_j$ becomes
\begin{align*}
(\delta_{\bB}\tilde{N}_{j,j^*}^1)^{-\frac{1}{2}} \delta_{\bB}\fF(P_j) =x_j\;. \\
\end{align*}
Therefore, the $\bRp$-basis $B$ is invariant under the gauge transformation.

\end{proof}

\begin{definition}[Frobenius-Perron Dimension]
For a fusion bialgebra $(\bA, \bB, \fF, d, \tau)$, $\fF^{-1}(1)$ is a multiple of a minimal projection $P_1$ in $\bA$. We define $\mu:= \frac{d(1_\bA)}{d(P_1)}$ as the Frobenius-Perron dimension of the fusion bialgebra.
\end{definition}

\begin{remark}
Note that the Frobenius-Perron dimension $\mu$ is invariant under the gauge transformations.
When the fusion bialgebra is canonical, let $B=\{x_1=1, x_2, \ldots, x_m\}$ be the basis of $\bB$ in Theorem \ref{Thm: extension}. Denote $\fF^{-1}(x_j) \in \bA$ by $x_j$ (as in Definition \ref{Fourier transform}). Then
$ \mu=\sum_{j=1}^m d(x_j)^2$.
This coincides with the definition of the Frobenius-Perron dimension of a fusion ring.
\end{remark}

\begin{remark}
For a canonical fusion bialgebra $(\bA, \bB, \fF, d, \tau)$,
one can consider $\tau$ as a Haar measure and $\mu d\circ \fF^{-1}$ as a Dirac measure on $\bB$.
On the dual side, one can consider $\mu^{-1} d$ as a Haar measure and $\tau\circ\fF$ as a Dirac measure on $\bA$.
\end{remark}

\begin{proposition}\label{prop:rotation2}
Let $(\bA, \bB, \fF, d, \tau)$ be a fusion ring.
Then for any $ x, y, z\in \bA$, we have
$$d((x*y)\diamond z)=\overline{d((J(z)*x^\# )\diamond J(y))}$$
\end{proposition}
\begin{proof}
We have
\begin{align*}
d((x*y)\diamond z)&= \tau(\fF(x*y)\fF (z^\#)^*)
=\tau(\fF(x)\fF(y)\fF (z^\#)^*)\\
&= \tau(\fF (z^\#)^*\fF(x)\fF(y))
= \tau(\fF (J(z^\#))\fF(x)\fF(J(y))^*)\\
&= \tau(\fF (J(z^\#)* x)\fF(J(y))^*)
=d((J(z^\#)* x)\diamond J(y)^\#)\\
&=\overline{d((J(z)*x^\# )\diamond J(y))}
\end{align*}
This completes the proof of the proposition.
\end{proof}

\subsection{Examples}\label{Sec: example}
\begin{example}
When the basis $B$ forms a group under the multiplication of $\bB$,
the C$^*$-algebra $\bB$ is the group algebra, $\Hil$ is its left regular representation Hilbert space, and $\tau$ is the normalized trace.
On the other side, the C$^*$-algebra $\bA$ is $L^{\infty}(B)$ and $d$ is the unnormalized Haar measure.
\end{example}

\begin{example}
When the basis $B$ forms a fusion ring,
the C$^*$-algebra $\bB$ is the fusion algebra. The quintuple $(\bA, \bB, \fF, d, \tau)$ is a canonical fusion bialgebra.
\end{example}

\begin{theorem}\label{PA to FB}
Suppose $\cN\subset \cM$ is a finite-index subfactor and $\sP_{\bullet}$ is its planar algebra \cite{Jon99}.
If $\sP_{2,+}$ is abelian, then $(\sP_{2,+}, \sP_{2,-}, \FS, tr_{2,+}, tr_{2,-})$ is a fusion bialgebra, and $\mu$ is the Jones index. Moreover,
we obtain a canonical one $(\sP_{2,+}, \sP_{2,-}, \fF, d, \tau)$, such that $d=\mu tr_{2,+}$ is the unnormalized trace of $\sP_{2,+}$, $\tau=tr_{2,-}$ is the normalized trace of $\sP_{2,-}$, and $\fF =\mu^{1/2}\FS= \delta \FS$, where
$\FS: \sP_{2,+} \to \sP_{2,-}$ is the string Fourier transform.
\end{theorem}

\begin{proof}
Let $P_j$, $j=1,2,\ldots, m$ be the minimal projections of $\sP_{2,+}$ and $P_1$ be the Jones projection.
Let $Tr$ be the unnormalized trace of $\sP_{2,+}$, namely $Tr(P_1)=1$.
Take $x_j=\frac{1}{\sqrt{Tr(P_i)}}\fF(P_j)$ and $x_{j^*}=\frac{1}{\sqrt{Tr(P_i)}} \fF(\overline{P_j})$, where $\overline{P_j}$ is the contragredient of $P_j$.
Then
$$x_j x_k=N_{j,k}^s x_s,$$
$x_1$ is the identity, $x_k^*=x_{k^*}$, $N_{j,k}^s \geq 0$, and $N_{j, k}^1=\delta_{j, k^*}$.
\end{proof}

\begin{remark}
On the 2-box space $\sP_{2,\pm}$ of a subfactor planar algebra, the Fourier transform is a $90^{\circ}$ rotation and the contragredient is a $180^{\circ}$ rotation, see e.g. \S 2.1 in \cite{Liuex}.
\end{remark}

\begin{definition}[Subfactorization]
We call $(\sP_{2,+}, \sP_{2,-}, \FS, tr_{2,+}, tr_{2,-})$ the fusion bialgebra of the subfactor $\cN\subset \cM$.
We say a fusion bialgebra $(\bA,\bB)$ can be subfactorized, if it comes from a subfactor $\cN\subset \cM$ in this way. We call $\cN\subset \cM$ a subfactorization of the fusion bialgebra.
\end{definition}

\subsection{Classifications}
In this section, we classify fusion bialgebras up to dimension three. By the gauge transformation, it is enough to classify canonical fusion bialgebras, which reduces to classify the $\bRp$-basis of C$^*$-algebra by Theorem~\ref{Thm: extension}.
Recall in Theorem \ref{PA to FB} that $(\sP_{2,+}, \sP_{2,-}, \FS, tr_{2,+}, tr_{2,-})$ of a subfactor planar algebra is a fusion bialgebra, if $\sP_{2,+}$ is abelian.
We refer the readers to \cite{BJ00,BJ03,BJL17,LiuYB,Ren,Edg19} for known examples of three dimensional fusion bialgebras from 2-box spaces of subfactors planar algebras.
In these examples, different subfactor planar algebras produce different fusion bialgebras.

In general, without assuming $\sP_{2,+}$ to be abelian, $(\sP_{2,+}, \sP_{2,-}, \FS, tr_{2,+}, tr_{2,-})$ has been studied as {\it the structure of 2-boxes of a planar algebra}, see Definition 2.25 in \cite{Liuex}.
One may ask when a subfactor planar algebra (generated by its 2-boxes) is determined by its structure of 2-boxes, equivalently by its fusion bialgebra when $\sP_{2,+}$ is abelian.
A positive answer is given in Theorem 2.26 in \cite{Liuex} for exchange relation planar algebras: exchange relation planar algebras are classified by its structure of 2-boxes.
Classifying fusion bialgebras is a key step to classify exchange relation planar algebras.
On the other hand, it would be interesting to find different subfactors planar algebras generated by 2-boxes with the same fusion bialgebras (or structures of 2-boxes).

\begin{proposition}[Rank-Two Classification]
Two dimensional canonical fusion bialgebras are classified by the Frobenius-Perron dimension $\mu \geq 2$. Moreover, they can be subfactorized if and only if $\mu$ is a Jones index.
\end{proposition}

\begin{proof}
If $\{x_1, x_2\}$ is a $\bRp$-basis, then $x_2^*=x_2$.
By Proposition~\ref{Prop: dim d}, $d_2:=d(x_2)\geq 1$, and
\begin{align*}
x_2^2&=x_1+\frac{d_2^2-1}{d_2} x_2 \;.
\end{align*}
So $\mu \geq 2$.
Conversely, when $\mu \geq 2$, we obtain a $\bRp$-basis in this way.

Furthermore, when $\mu$ is a Jones index, the canonical fusion bialgebra can be subfactorized by the Temperley-Lieb-Jones subfactors with index $\mu$.
\end{proof}

Suppose $\{x_1=1, x_2, x_3\}$ is the $\bRp$-basis of a three-dimensional C$^*$-algebra $\bB$. Then $\bB$ is commutative.
Take $d_2=d(x_2)$ and $d_3=d(x_3)$.
There are two different cases: $x_2^*=x_2$ or $x_2^*=x_3$.

\begin{proposition}[Rank-Three Classification, Type I]\label{Prop: Classification I}
In the case $x_2^*=x_2$, three dimensional canonical fusion bialgebras are classified by three parameters $d_2,d_3,a$, such that
$d_2,d_3\geq 1$, $0 \leq a \leq 1$, $b=1-a$, $d_2^2-1-a d_3^2 \geq 0$, $d_3^2-1-b d_2^2 \geq 0$. Moreover,
\begin{align*}
x_2 x_2&=x_1+ \frac{d_2^2-1-a d_3^2}{d_2} x_2 +  a d_3 x_3  \;, \\
x_2 x_3&=ad_3 x_2 + bd_2 x_3 \;, \\
x_3 x_3&=x_1+b d_2 x_2 + \frac{d_3^2-1-b d_2^2}{d_3}x_3 \;.
\end{align*}
\end{proposition}

\begin{proof}
Take parameters $a,b$, such that
\begin{align*}
x_2 x_3&=ad_3 x_2 + bd_2 x_3 \;.
\end{align*}
Then $a,b \geq 0$. Computing $d$ on both sides, we have that $a+b=1$. So $a\leq 1$.
By Equation \eqref{Equ: rotation}, $N_{2,2}^3=N_{2,3}^2$.
Hence
\begin{align*}
x_2 x_2&=x_1+ \frac{d_2^2-1-a d_3^2}{d_2} x_2 +  a d_3 x_3  \;,
\end{align*}
by computing $d$ on both sides.
Similarly $N_{3,3}^2=N_{2,3}^3$, and
\begin{align*}
x_3 x_3&=x_1+b d_2 x_2 + \frac{d_3^2-1-b d_2^2}{d_3}x_3 \;.
\end{align*}
As the coefficients are non-negative, we have that $d_2^2-1-a d_3^2 \geq 0$ and $d_3^2-1-b d_2^2 \geq 0$.

Conversely, with the above parameters, the multiplication is associative and $(x_jx_k)^*=x_k^*x_j^*$ by a direct computation. Therefore, we obtain the classification.
\end{proof}

\begin{proposition}[Rank-Three Classification, type II]\label{Prop: Classification II}
In the case $x_2^*=x_3$, three dimensional canonical fusion bialgebras are classified by one parameter $\mu \geq 3$. Moreover,  $ d_2=d_3=\sqrt{\frac{\mu-1}{2}}$,
\begin{align*}
x_2 x_2&=\frac{d_2^2-1}{2d_2} x_2 +  \frac{d_2^2+1}{2d_2} x_3  \;,  \\
x_2 x_3&=x_1+ \frac{d_2^2-1}{2d_2} (x_2 + x_3) \;, \\
x_3 x_3&=\frac{d_2^2+1}{2d_2} x_2 + \frac{d_2^2-1}{2d_2}x_3 \;.
\end{align*}
\end{proposition}

\begin{proof}
As $x_2^*=x_3$, we have that $ d_2=d_3=\sqrt{\frac{\mu-1}{2}} \geq 1$ and $x_2x_3$ is self-adjoint. So
\begin{align*}
x_2 x_3&=x_1+ \lambda (x_2 + x_3) \;,
\end{align*}
for some $\lambda \geq 0$.
Computing $d$ on both sides, we have that $ \lambda=\frac{d_2^2-1}{2d_2}$.
By Equation \eqref{Equ: rotation}, $N_{2,2}^2=N_{2,3}^3$. So
\begin{align*}
x_2 x_2&=\frac{d_2^2-1}{2d_2} x_2 +  \frac{d_2^2+1}{2d_2} x_3  \;,
\end{align*}
by computing $d$ on both sides.
Similarly $N_{3,3}^3=N_{2,3}^2$, and
\begin{align*}
x_3 x_3&=\frac{d_2^2+1}{2d_2} x_2 + \frac{d_2^2-1}{2d_2}x_3 \;.
\end{align*}
The coefficients are non-negative.

Conversely, with the above parameters, the multiplication is associative and $(x_jx_k)^*=x_k^*x_j^*$ by a direct computation. Therefore, we obtain the classification.
\end{proof}

The one-parameter family of three dimensional canonical fusion bialgebras in the above classification can be realized as the 2-box spaces of a one-parameter family of planar algebras constructed in \cite{LiuYB}. For each $d_2\geq 1$, there are a complex-conjugate pair of planar algebras to realize the fusion bialgebra as the 2-box spaces. So such a realization may not be unique.
Moreover, these planar algebras arise from subfactors if and only if $\mu=\cot^2(\frac{\pi}{2N+2})$ for some $N\in \bZ_+$.
Inspired by this observation, we conjecture that:
\begin{conjecture}
In the case II, the one-parameter family of three dimensional fusion bialgebras can be subfactorized if and only if $\mu=\cot^2(\frac{\pi}{2N+2})$.
\end{conjecture}

%

\subsection{Duality}
\begin{definition}
For a fusion bialgebra $(\bA, \bB, \fF, d, \tau)$, we define the quintuple $(\bB, \bA, \tfF, \tau, d)$ as its Fourier dual, where $\tfF=\#\fF^{-1}*$.
\end{definition}
\begin{remark}
To be compatible with the examples from subfactor theory, this is the natural Fourier dual, not $(\bB, \bA, \fF^{-1}, \tau, d)$.
\end{remark}

\begin{definition}[Contragredient]
For any $x \in \bA$, we define its contragredient as
\begin{align*}
\overline{x}:=&\tfF \fF(x) \;.
\end{align*}
For any $y \in \bB$, we define its contragredient as
\begin{align*}
\overline{y}:=&\fF \tfF (y) \;.
\end{align*}
\end{definition}

When $\bB$ is commutative, it is natural to ask whether the dual  $(\bB, \bA, \fF^{-1}, \tau, d)$ is also a fusion bialgebra.
We need to check the three conditions in Definition~\ref{Def: Fusion bialgebras}.
The conditions (2) and (3) always hold on the dual, but condition (1) may not hold.

\begin{proposition}[Dual Modular Conjugation]
For a fusion bialgebra $(\bA, \bB, \fF, d, \tau)$,
the map $J_{\bB}(x):=\tfF^{-1}(\tfF(x)^{\#})$ is an anti-linear, *-isomorphism on $\bB$.
\end{proposition}

\begin{proof}
Note that the map $J_{\bB}$ is anti-linear and
$$J_{\bB}(x_j)=\tfF^{-1}(\tfF(x_j)^{\#})= \fF(\fF^{-1}(x_j^*)^{\#})^*=x_j,$$
so $J_{\bB}(x_jx_k)=x_jx_k$, and $J_{\bB}$ is an anti-linear *-isomorphism on $\bB$.
\end{proof}

\begin{proposition}[Dual Jones Projection]\label{Prop: Dual Jones Projection}
For a fusion bialgebra $(\bA, \bB, \fF, d, \tau)$,
$\fF(1_{\bA})$ is a positive multiple of a central, minimal projection $e_{\bB}$ in $\bB$, where $1_{\bA}$ is the identity of $\bA$.
Moreover, $ \mu=\frac{\tau(1_{\bB})}{\tau(e_{\bB})}$.
\end{proposition}

\begin{proof}
Since the gauge transformation only changes the global scaler, without loss of generality, we assume that $(\bA, \bB, \fF, d, \tau)$ is a canonical fusion bialgebra. Then
\begin{align*}
 \fF(1_{\bA})&=\sum_{j=1}^m \fF(P_j)=\sum_{j=1}^m d(x_j) x_j \;.
\end{align*}
Note that $d(x_j)=d(x_{j*}) \geq 0$ and $x_j^*=x_{j^*}$, so $\fF(1_{\bA})=\fF(1_{\bA})^*$.
By Equation \eqref{Equ: rotation} and Proposition~\ref{Prop: dim d},
\begin{align*}
\fF(1_{\bA}) x_k=&\sum_{j=1}^m d(x_j) x_j x_k
=\sum_{j,s=1}^m d(x_j) N_{j,k}^{s} x_s \\
=&\sum_{j,s=1}^m d(x_{j^*}) N_{k,s^*}^{j^*} x_s
=\sum_{s=1}^m d(x_{k}x_{s^*}) x_s \\
=&d(x_k) \sum_{s=1}^m d(x_{s}) x_s
=d(x_k) \fF(1_{\bA}) \;.
\end{align*}
So $\fF(1_{\bA})*\fF(1_{\bA})=\mu \fF(1_{\bA})$  and
$$e_{\bB}=\mu^{-1}\tfF(1_{\bA})=\mu^{-1}\fF(1_{\bA})=\mu^{-1}\sum_{j=1}^m d(x_j) x_j$$
 is a central, minimal projection.
Moreover,
\begin{align*}
\tau(e_{\bB})=\mu^{-1}\sum_{j=1}^m d(x_j) \tau(x_j)=\mu^{-1}, \quad \tau(1_{\bB})=\tau(x_1)=1 \;.
\end{align*}
We have $ \frac{\tau(1)}{\tau(e_{\bB})}=\mu.$
\end{proof}

\subsection{Self Duality}
In this subsection, we will give the definition of the self-dual fusion bialgebra and study the $S$-matrix associated to it.

\begin{definition}
Two fusion bialgebras $(\bA, \bB, \fF, d, \tau)$ and $(\bA', \bB', \fF', d', \tau')$ are called isomorphic, if there are *-isomorphisms $\Phi_{\bA}: \bA \to \bA'$ and $\Phi_{\bB}: \bB \to \bB'$, such that $\Phi_{\bB} \fF = \fF' \Phi_{\bA}$, $d=d'\Phi_{\bA}$ and $\tau=\tau' \Phi_{\bB}$.
\end{definition}

\begin{definition}
A fusion bialgebra $(\bA, \bB, \fF, d, \tau)$ is called self-dual, if its dual $(\bB, \bA, \tfF, \tau, d)$ is a fusion bialgebra and they are isomorphic.
Furthermore, it is called symmetrically self-dual, if $\Phi_{\bB}\Phi_{\bA}=1$ on $\bA$.
\end{definition}

The maps $\Phi_{\bA}, \Phi_{\bB}$ implementing the self-duality may not be unique, even for finite abelian groups.

\begin{proposition}\label{Prop: Contragredient}
Suppose $(\bA, \bB, \fF, d, \tau)$ is a self-dual canonical fusion bialgebra with a $\bRp$-basis $B=\{x_1=1_{\bB}, x_2, \ldots, x_m\}$ of $\bB$,
Then
$$\overline{\sum_j \lambda_j {x_j}}=\sum_j \lambda_j x_{j^*}.$$
Consequently, the contragredient maps on $\bA$ and $\bB$ are anti-$*$-isomorphisms.
\end{proposition}

\begin{proof}
The statements follow from the fact that the contragredient map is linear and
$$\overline{x_j}=\fF \tfF(x_j)=\fF \fF^{-1}(x_j^*)^{\#}=x_{j^*}=x_{j}^*,  ~\forall~ 1\leq j \leq m.$$
\end{proof}

\begin{definition}
Suppose $(\bA, \bB, \fF, d, \tau)$ is a self-dual canonical fusion bialgebra with a $\bRp$-basis $B=\{x_1=1_{\bB}, x_2, \ldots, x_m\}$ of $\bB$, we define the $S$-matrix $S$ as an $m \times m$ matrix with entries $S_j^k$, such that
$$\fF\Phi_{\bB}(x_j)=\sum_{k=1}^m S_{j}^k x_k.$$
\end{definition}

\begin{proposition}\label{Prop: S unitary}
For a self-dual canonical fusion bialgebra  $(\bA, \bB, \fF, d, \tau)$,
$\fF\Phi_{\bB}$ is a unitary transformation on $L^2(\bB,\tau)$, and the $S$ matrix is a unitary.
\end{proposition}

\begin{proof}
Both $\fF$ and $\Phi_{\bB}$ are unitary transformations, so the composition is a unitary on $L^2(\bB,\tau)$.
Recall that $B$ is an orthonormal basis of $L^2(\bB,\tau)$, we have $S$ is a unitary matrix.
\end{proof}

\begin{proposition}
A self-dual canonical fusion bialgebra is symmetrically self-dual if and only if $S_{j}^k=S_{k}^j$.
In this case, $$\fF\Phi_{\bB}=\Phi_{\bA}\tfF.$$
\end{proposition}

\begin{proof}
For a self-dual canonical fusion bialgebra, we have that
\begin{align*}
\Phi_{\bB} \fF \Phi_{\bB} \fF &= \tfF \Phi_{\bA} \Phi_{\bB} \fF \;.
\end{align*}
By Propositions~\ref{Prop: Contragredient} and~\ref{Prop: S unitary},
the fusion bialgebra is symmetrically self-dual if and only if $(\Phi_{\bB}\fF)^2=\tfF\fF$  if and only if
$(S^2)_{j}^k=\delta_{j,k^*}$ if and only if $S_{j}^k=S_{k}^j$.
In this case,
\begin{align*}
\Phi_{\bB}\fF\Phi_{\bB}\fF&=\Phi_{\bB}\Phi_{\bA}\tfF\fF \;.
\end{align*}
So $\fF\Phi_{\bB}=\Phi_{\bA}\tfF.$
\end{proof}

\begin{remark}
For the group case, $S$ is a bicharacter, see \cite{LMP18} for the discussion on self-duality and symmetrically self-duality.
\end{remark}

\begin{theorem}[Verlinde Formula]
For a self-dual canonical fusion bialgebra  $(\bA, \bB, \fF, d, \tau)$,
$$SL_jS^*=D_j, \quad 1\leq j\leq m,$$
where $L_jx=x_j x$ for $x\in \bB$ and $D_j x_k= \frac{S_j^k}{d(x_k)} x_k$.
\end{theorem}

\begin{proof}
Assume that $B=\{x_1=1, x_2, \ldots, x_m\}$ is a basis of $\bB$.
We have
\begin{align*}
\fF\Phi_\bB L_j(\fF\Phi_\bB)^{-1}(x_k))&=\fF\Phi_\bB(x_j(\fF\Phi_\bB)^{-1}(x_k)) \\
&=\fF(\Phi_\bB(x_j)\diamond \Phi_\bB((\fF\Phi_\bB)^{-1}(x_k)))\\
&=\fF(\fF^{-1}\fF\Phi_\bB(x_j)\diamond \fF^{-1}(x_k))\\
&=\sum_{s=1}^m S_j^s \fF(\fF^{-1}(x_s)\diamond \fF^{-1}(x_k))\\
&= \frac{S_j^k}{d(x_k)} x_k\\
\end{align*}
This completes the proof of the theorem.
\end{proof}

\section{Schur Product Property}\label{Sec: Schur Product Property}
In this section, we will study Schur product property for the dual of a fusion bialgebra.

\begin{definition}
For a fusion bialgebra $(\bA, \bB, \fF, d, \tau)$, the multiplication $\diamond$ on $\bA$ induces a convolution $\conv_{\bB}$ on $\bB$:  $~\forall  x,y \in \bB$,
\begin{align}\label{Equ: convolution b}
x \conv_{\bB} y :=& \tfF^{-1}(\tfF(x) \diamond \tfF(y))
=(\fF(\fF^{-1}(y^*) \diamond \fF^{-1}(x^*)))^* \;.
\end{align}
We say $\bB$ has the Schur product property, if
$x \conv_{\bB} y \geq 0 $, for any $x,y \geq 0$ in $\bB$.
\end{definition}

\begin{proposition}[Frobenius Reciprocity]\label{prop:rotation}
Let $(\bA, \bB, \fF, d, \tau)$ be a fusion bialgebra.
Then for any $ x ,  y,  z\in \bB$, we have
$$\tau((x*_\bB y)z)=\overline{\tau((J_\bB(x)*_\bB z^*)y^*)}$$
\end{proposition}
\begin{proof}
We have
\begin{align*}
\tau((x*_\bB y)z)&=\tau((\fF(\fF^{-1}(y^*) \diamond \fF^{-1}(x^*)))^*z)\\
&= d(\fF^{-1}(y^*)^\# \diamond \fF^{-1}(x^*)^\#\diamond \fF^{-1}(z))\\
&=\tau(y \fF( \fF^{-1}(x^*)^\#\diamond \fF^{-1}(z)))\\
&=\overline{\tau( \fF( \fF^{-1}(J_\bB(x)^*)\diamond \fF^{-1}(z^{**}))^*y^*)}\\
&= \overline{\tau((J_\bB(x)*_\bB z^*)y^*)}
\end{align*}
This completes the proof of the proposition.
\end{proof}

\begin{proposition}\label{prop:ufusionschur}
Suppose the fusion bialgebra $(\bA, \bB, \fF, d, \tau)$ can be subfactorized.
Then the Schur product property holds on the dual.
\end{proposition}

\begin{proof}
It follows from the quantum Schur product theorem on subfactors, Theorem 4.1 in \cite{Liuex}.
\end{proof}

\begin{proposition}
Suppose the fusion bialgebra $(\bA, \bB, \fF, d, \tau)$ is self-dual.
Then the Schur product property holds on the dual.
\end{proposition}
\begin{proof}
It follows from the definition of self-dual fusion bialgebras.
\end{proof}

We define a linear map $\Delta:\bB \to \bB \otimes \bB$ such that
$$\Delta(x_j)=\frac{1}{d(x_j)}x_j\otimes x_j, \quad \Delta(x^*)=\Delta(x)^*, \quad x \in\bB.$$
Then $\Delta$ is a $*$-preserving map.
We say $\Delta$ is positive if $\Delta(x)\geq 0$ for any $x\geq 0$.

\begin{proposition}\label{prop:comulti3}
Let $(\bA, \bB, \fF, d, \tau)$ be a fusion bialgebra and suppose $\Delta$ is positive.
Then the Schur product property holds for $(\bB, \bA, \tfF, \tau, d)$
\end{proposition}

\begin{proof}
We denote by $\iota$ the identity map.
Note that for any $x=\sum_{j=1}^m \lambda_j x_j, y=\sum_{j=1}^m \lambda_j' x_j\in \bB$, we have
\begin{align*}
(\iota\otimes \tau)(\Delta(x)(1\otimes \overline{y})) &=(\iota\otimes \tau)\left(\sum_{j=1}^m\frac{\lambda_j}{d(x_j)} x_j\otimes x_j\left(1\otimes \sum_{j=1}^m \lambda_j' x_{j^*}\right)\right)\\
&=(\iota\otimes \tau)\left(\sum_{j,k=1}^m\frac{\lambda_j\lambda_k'}{d(x_j)} x_j\otimes x_jx_{k^*} \right)\\
&=\sum_{j,k=1}^m\frac{\lambda_j\lambda_k'}{d(x_j)} x_j\tau( x_jx_{k^*} )\\
&=\sum_{j=1}^m\frac{\lambda_j\lambda_j'}{d(x_j)} x_j= \overline{x *_\bB y}.
\end{align*}

Let $\overline{y}=y_1y^*_1$ and $x\geq 0$.
Then $(1\otimes y_1^*)\Delta(x)(1\otimes y_1)\geq 0$.
Hence
$$(\iota\otimes \tau)\left((1\otimes y_1^*)\Delta(x)(1\otimes y_1)\right)\geq 0,$$
i.e. $x *_\bB y\geq 0$.
\end{proof}

\begin{proposition} \label{SchurProp}
For a fusion bialgebra $(\bA, \bB, \fF, d, \tau)$, the Schur product property holds on $\bB$, if and only if
\begin{align*}
d((J(x)*x)\diamond (J(y)*y) \diamond (J(z)*z)) \geq 0, ~\forall x,y,z \in \bA.
\end{align*}
\end{proposition}

\begin{proof}
By the Schur product property on $\bA$,
\begin{align*}
(J(x)*x)^*&=J(x)^\#*x^*=J(x^\#)*x^*,  ~\forall x \in \bA.
\end{align*}
Note that $\fF(J(x)*x)=|\fF(x)|^2 \geq 0$, and any positive operator in $\bB$ is of such form.
Therefore, by Proposition~\ref{prop:planch} and Equation \eqref{Equ: convolution b},
\begin{align*}
     & d((J(x)*x)\diamond (J(y)*y) \diamond (J(z)*z)) \geq 0, ~\forall x,y,z \in \bA, \\
\iff & \tau(\fF(J(x^{\#})*(x^{\#})^* \fF((J(y)*y)\diamond (J(z)*z) ) \geq 0, ~\forall x,y,z \in \bA, \\
\iff & \tau (|\fF(x^*)|^2  (|\fF(y)|^2 *_{\bB} |\fF(z)|^2)^* ) \geq 0~\forall x,y,z \in \bA, \\
\iff & |\fF(y)|^2 *_{\bB} |\fF(z)|^2  \geq 0~\forall y,z \in \bA,
\end{align*}
if and only if the Schur product property holds on $\bB$.
\end{proof}

The Schur product property may not hold on the dual, even for a 3-dimensional fusion bialgebra.
We give a counterexample.
For this reason, Young's inequality do not hold on the dual as well, see \S\ref{Sec: Young} for further discussions.
As a preparation, we first construct the minimal projections in $\bB$.

\begin{proposition}\label{Prop: Schur}
For the canonical fusion bialgebra  $(\bA, \bB, \fF, d, \tau)$ in Proposition~\ref{Prop: Classification I},
the minimal projections of $\bB$ are given by
\begin{align*}
Q_1&=\mu^{-1} (x_1 + d_2 x_2 + d_3 x_3) \;, \\
Q_2&=\nu_2^{-1} (x_1-\frac{\lambda_2}{d_2} x_2 - \frac{1-\lambda_2}{d_3} x_3 ) \;, \\
Q_3&=\nu_3^{-1}(x_1-\frac{\lambda_3}{d_2} x_2 - \frac{1-\lambda_3}{d_3} x_3 ) \;,
\end{align*}
where $\lambda_2,\lambda_3$ are the solutions of
\begin{align*}
\lambda_2+\lambda_3&=a d_3^2-b d_2^2 +1\;, \\
\lambda_2\lambda_3&=-bd_2^2 \;;
\end{align*}
and $ \nu_j=1+\frac{\lambda_j^2}{d_2^2}+\frac{(1-\lambda_j)^2}{d_3^2}$.
\end{proposition}

\begin{proof}
Note that $\mu=1+d_2^2+d_3^2$.
By Proposition~\ref{Prop: Dual Jones Projection},
\begin{align*}
Q_1&=e_{\bB}=\mu^{-1} (x_1 + d_2 x_2 + d_3 x_3) \;.
\end{align*}
For $j=2,3$, $d(Q_1Q_j)=0$, so
\begin{align*}
Q_2&=\nu_2^{-1} (x_1-\frac{\lambda_2}{d_2} x_2 - \frac{1-\lambda_2}{d_3} x_3 ) \;, \\
Q_3&=\nu_3^{-1}(x_1-\frac{\lambda_3}{d_2} x_2 - \frac{1-\lambda_3}{d_3} x_3 ) \;,
\end{align*}
for some $\nu_2, \nu_3 >0$.
As $Q_j^2=Q_j$, we have that
$ \nu_j=1+\frac{\lambda_j^2}{d_2^2}+\frac{(1-\lambda_j)^2}{d_3^2}$.
Furthermore, $Q_2Q_3=0$, so
$$x_1- \frac{\lambda_2+\lambda_3}{d_2} x_2- \frac{2-\lambda_2-\lambda_3}{d_3}  x_3 +\frac{\lambda_2\lambda_3}{d_2^2} x_2^2 + \frac{\lambda_2+\lambda_3-2\lambda_2\lambda_3}{d_2d_3} x_2x_3 +\frac{(1-\lambda_2)(1-\lambda_3)}{d_3^2} x_3^2=0.$$
The coefficient of $x_j$ is $0$, for $1 \leq j \leq 3$. So
\begin{align*}
1+\frac{\lambda_2\lambda_3}{d_2^2}+\frac{(1-\lambda_2)(1-\lambda_3)}{d_3^2}&=0\;, \\
- \frac{\lambda_2+\lambda_3}{d_2}+\frac{\lambda_2\lambda_3}{d_2^2} \frac{d_2^2-1-a d_3^2}{d_2} + \frac{\lambda_2+\lambda_3-2\lambda_2\lambda_3}{d_2d_3} ad_3 +\frac{(1-\lambda_2)(1-\lambda_3)}{d_3^2}  b d_2&=0 \;, \\
- \frac{2-\lambda_2-\lambda_3}{d_3} +\frac{\lambda_2\lambda_3}{d_2^2} a d_3 + \frac{\lambda_2+\lambda_3-2\lambda_2\lambda_3}{d_2d_3} bd_2 +\frac{(1-\lambda_2)(1-\lambda_3)}{d_3^2} \frac{d_3^2-1-b d_2^2}{d_3}&=0  \;.
\end{align*}
Take
\begin{align*}
\omega_1&=\lambda_2+\lambda_3\;, \\
\omega_2&=\lambda_2\lambda_3 \;.
\end{align*}
Then
\begin{align*}
1+\frac{\omega_2}{d_2^2}+\frac{1-\omega_1+\omega_2}{d_3^2}&=0\;, \\
- \frac{\omega_1}{d_2}+\frac{\omega_2}{d_2^2} \frac{d_2^2-1-a d_3^2}{d_2} + \frac{\omega_1-2\omega_2}{d_2d_3} ad_3 +\frac{1-\omega_1+\omega_2}{d_3^2}  b d_2&=0 \;, \\
- \frac{2-\omega_1}{d_3} +\frac{\omega_2}{d_2^2} a d_3 + \frac{\omega_1-2\omega_2}{d_2d_3} bd_2 +\frac{1-\omega_1+\omega_2}{d_3^2} \frac{d_3^2-1-b d_2^2}{d_3}&=0  \;.
\end{align*}
Solving the linear system, we have that
\begin{align*}
\omega_1&=a d_3^2-b d_2^2 +1\;; \\
\omega_2&=-bd_2^2 \;.
\end{align*}
Therefore, $\lambda_2, \lambda_3$ are the solutions of
\begin{align*}
\lambda_2+\lambda_3&=a d_3^2-b d_2^2 +1\;, \\
\lambda_2\lambda_3&=-bd_2^2 \;.
\end{align*}

\end{proof}

\begin{theorem} \label{thm:SchurCounterEx}
For the three dimensional canonical fusion bialgebras $(\bA, \bB, \fF, d, \tau)$ parameterized by $d_2, d_3, a$ in Proposition~\ref{Prop: Classification I}, the Schur product property does not hold on the dual in general, for example, $d_2=1000, d_3=500, a=0.750001.$
\end{theorem}

\begin{proof}
Fix $0<a<1$, and $b=1-a$. Take $d_2 \to \infty$ and $d_3^2-1=b d_2^2$, then $\lambda_3\to b^{-1}, \lambda_2 d_2^{-2} \to -b^2$. Moreover,
\begin{align*}
&d_2^{-2} d\left((\fF^{-1}(\nu_2 Q_2))^3\right)\\
=&d_2^{-2} \left(1-\left(\frac{\lambda_2}{d_2}\right)^3 d_2^{-2} d_2 - \left(\frac{1-\lambda_2}{d_3}\right)^3 d_3^{-2} d_3 \right) \\
=& d_2^{-2} \left(1-\left(\lambda_2 d_2^{-2}\right)^3 d_2^2 - \frac{(1-\lambda_2)^3}{d_3^4}\right) \\
=& -\left(\lambda_2 d_2^{-2}\right)^3  +d_2^{-2}\left(1- \frac{(1-\lambda_2)^3}{(1+bd_2^2)^2}\right)\\
\to & b^6-b^4 <0
\end{align*}
By Proposition~\ref{Prop: Schur}, the Schur product property does not hold in general on the dual.
Numerically, one can take $d_2=1000, d_3=500, a=0.750001$, then $d((\fF^{-1}(\nu_2 Q_2))^3)<0$.
\end{proof}

\begin{remark} Subsection~\ref{Sec:SmallRank} provides a complementary approach for the study of this family of rank 3 fusion bialgebras, leading to visualize the areas of parameters where Schur product property (on the dual) does not hold, and to a character table whose matrix (function of the fusion coefficients) is equal to the inverse of the one underlying Theorem~\ref{Prop: Schur} (function of the Frobenius-Perron dimensions).
\end{remark}

\section{Hausdorff-Young Inequality and Uncertainty Principles}
In this section, we will recall some inequalities for general von Neumann algebras first and then we will prove the Hausdorff-Young inequalities and uncertainty principles for fusion bialgebras.

\begin{proposition}[H\"{o}lder's inequality, see for example Proposition 4.3, 4.5 in \cite{JLW16}]\label{prop:holder}
Let $\mathcal{M}$ be a von Neumann algebra with a normal faithful tracial state $\tau$ and $1\leq p, q\leq \infty$ with $1/p+1/q=1$.
Then for any $x\in L^p(\mathcal{M})$, $y\in L^q(\mathcal{M})$, we have
$$\|xy\|_1\leq \|x\|_p\|y\|_q$$
Moreover
\begin{enumerate}[(1)]
\item for $1<p<\infty$, $\|xy\|_1= \|x\|_p\|y\|_q$ if and only if $ \frac{|x|^p}{\|x\|_p^p}=\frac{|y^*|^q}{\|y\|_q^q}$;
\item for $p=\infty$, $\|xy\|_1\leq \|x\|_\infty\|y\|_1$ if and only if the spectral projection of $|x|$ corresponding to $\|x\|_\infty$ contains the projection $\cR(y)$ as subprojection, where $\cR(y)$ is the range projection of $y$.
\end{enumerate}
\end{proposition}

\begin{corollary}\label{cor:holder1}
Let $\mathcal{M}$ be a von Neumann algebra with a normal faithful tracial state $\tau$ and $x\in \mathcal{M}$.
Then $\|x\|_2^2=\|x\|_\infty \|x\|_1$ if and only if $x$ is a multiple of a partial isometry.
\end{corollary}

\begin{proposition}[Interpolation Theorem, see for example Theorem 1.2 in \cite{Kos84}]\label{prop:inter}
Let $\mathcal{M}, \mathcal{N}$ be finite von Neumann algebras with normal faithful states $\tau_1,\tau_2$.
Suppose $T:\mathcal{M}\to \mathcal{N}$ is a linear map.
If
$$\|Tx\|_{p_1, \tau_2}\leq K_1\|x\|_{q_1, \tau_1}, \text{ and } \|Tx\|_{q_1, \tau_2}\leq K_2\|x\|_{q_2,\tau_1},$$
then
$$\|Tx\|_{p_\theta, \tau_2}\leq K_1^{1-\theta}K_2^\theta \|x\|_{q_\theta, \tau_1},$$
where $ \frac{1}{p_\theta}=\frac{1-\theta}{p_1}+\frac{\theta}{p_2}$, $ \frac{1}{q_\theta}=\frac{1-\theta}{q_1}+\frac{\theta}{q_2}$, $0\leq \theta \leq 1$.
\end{proposition}

\subsection{Hausdorff-Young Inequality}

\begin{proposition}\label{prop:hyinf1}
Let $(\bA, \bB, \fF, d, \tau)$ be a fusion bialgebra.
Then
$$\|\fF(x)\|_{\infty, \bB}\leq \|x\|_{1, \bA}, \quad x\in \bA $$
 and
 $$ \|\tfF(x)\|_{\infty,\bA}\leq \|x\|_{1, \bB}, \quad x\in \bB.$$
\end{proposition}
\begin{proof}
Let $x=\sum_{j=1}^m \lambda_j \fF^{-1}(x_j)\in \bA$.
Then
\begin{align*}
\|\fF(x)\|_{\infty, \bB}&= \left\| \sum_{j=1}^m \lambda_j x_j\right\|_{\infty, \bB}
\leq \sum_{j=1}^m |\lambda_j|\left\| x_j \right\|_{\infty, \bB}\\
&= \sum_{j=1}^m |\lambda_j|d(x_j)
=\|x\|_{1, \bA},
\end{align*}
This proves the first inequality.

For the second inequality, we let $x=\sum_{j=1}^m \lambda_jx_j$.
Then $ \displaystyle \|\tfF(x)\|_{\infty, \bA}=\max_{1\leq j\leq m}\frac{|\lambda_j|}{d(x_j)}.$
For any $k$ such that $\lambda_{k}\neq 0$, we have
\begin{align*}
\tau\left(\left|\sum_{j=1}^m \lambda_j  x_j\right|\right)
& \geq \frac{\tau\left(\overline{\lambda_{k}} x_{k^*}\sum_{j=1}^m \lambda_j  x_j\right)}{|\lambda_{k}| d(x_{k^*})}  =\frac{ |\lambda_k|}{d(x_k)}.
\end{align*}
Hence $ \|\tfF(x)\|_{\infty,\bA}\leq \|x\|_{1, \bB}.$
This completes the proof of the proposition.
\end{proof}

\begin{theorem}[Hausdorff-Young inequality]\label{thm:hy}
Let $(\bA, \bB, \fF, d, \tau)$ be a fusion bialgebra.
Then for any $1\leq p\leq 2$, $1/p+1/q=1$, we have
$$\|\fF(x)\|_{q, \bB}\leq \|x\|_{p, \bA}, \quad x\in \bA $$
 and
 $$ \|\tfF(x)\|_{q,\bA}\leq \|x\|_{p, \bB}, \quad x\in \bB.$$
\end{theorem}
\begin{proof}
It follows from Proposition~\ref{prop:planch},~\ref{prop:hyinf1} and Proposition~\ref{prop:inter}.
\end{proof}

We divide the first quadrant into three regions $R_{T},R_{F},R_{TF}$.
Recall that $\mu=\sum_{j=1}^m d(x_j)^2$ is the Frobenius-Perron dimension of $\bB$.
Let $K$ be a function on $[0,1]^2$ given by
\begin{equation}\label{K}
K(1/p,1/q)=\left\{
\begin{array}{lll}
1 &\text{ for } &(1/p,1/q)\in R_F,\\
\mu^{1/p+1/q-1} &\text{ for } &(1/p,1/q)\in R_T,\\
\mu^{1/q-1/2} &\text{ for }& (1/p,,1/q)\in R_{TF}.
\end{array}
\right.
\end{equation}
as illustrated in Figure~\ref{Fig: norm}.

\begin{figure}
\begin{center}
\begin{tikzpicture}[scale=2]
\draw [line width=2] [->](0,0)--(0,3);
\node at (-0.1,-0.1) {0};
\node [left] at (0, 3) {$\frac{1}{q}$};
\draw [line width=2]  [->](0,0)--(3,0);
\node [below] at (1,0) {$\frac{1}{2}$};
\node at (1,0) {$\bullet$};
\node [below] at (2, 0) {1};
\node at (2,0) {$\bullet$};
\node [below] at (3,0) {$\frac{1}{p}$};
\draw [blue, line width=2] (1,1)--(0,1);
\node [left] at (0,1) {$\frac{1}{2}$};
\node at (0,1) {$\bullet$};
\node [left] at (0, 2) {1};
\node at (0,2) {$\bullet$};
\node [orange] at (0.5,1.5) {$\mu^{\frac{1}{q}-\frac{1}{2}}$};
\node at (0.5,1.2) {$R_{TF}$};
\draw [red, line width=2] (1,1)--(2,0);
\node at (0.5, 0.5) {$R_F$};
\node [orange] at (1, 0.5) {$1$};
\draw [blue, line width=2] (1,1)--(1,2);
\node [orange] at (1.5, 1.5) {$\mu^{\frac{1}{p}+\frac{1}{q}-1}$};
\node at (1.5,1) {$R_T$};
\draw [dashed, line width=2] (0,2)--(2,2);
\draw [dashed, line width=2] (2,0)--(2,2);
\end{tikzpicture}
\end{center}
\caption{The norms of the Fourier Transform.}
\label{Fig: norm}
\end{figure}
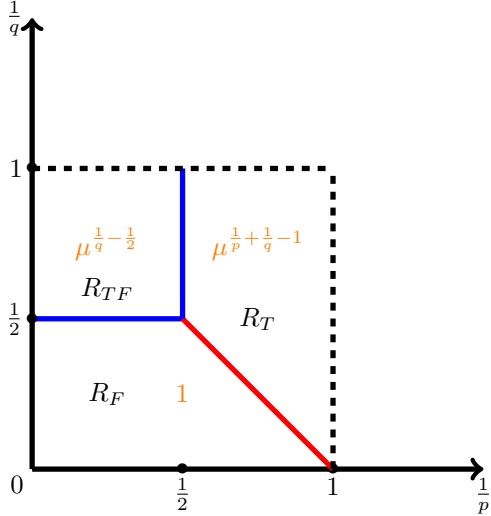

\begin{theorem}\label{thm:normfourier}
Let $(\bA, \bB, \fF, d, \tau)$ be a fusion bialgebra and $x\in \bB$
Then for any $1\leq p,q\leq \infty$, we have
$$K(1/p, 1/q)^{-1}\|x\|_{p, \bB} \leq \|\tfF(x)\|_{q, \bA}\leq K(1/p, 1/q)\|x\|_{p, \bB}$$
\end{theorem}
\begin{proof}
It follows from the proof of Theorem 3.13 in \cite{LiuWuscm}.
We leave the details to the readers.
\end{proof}

\subsection{Uncertainty Principles}
 We will prove the Donoho-Stark uncertainty principle, Hirschman-Beckner uncertainty principle and R\'{e}nyi entropic uncertainty principle for fusion bialgebras.
For any $x\in \bA$, we let $\cR(x)$ be the range projection of $x$ and $\cS(x)=d(\cR(x))$.
For any $x\in \bB$, $\cS(x)=\tau(\cR(x))$.

\begin{lemma}\label{lem:support}
Let $(\bA, \bB, \fF, d, \tau)$ be a fusion bialgebra.
Then we have
$$\cS(x)=\cS(x^\#)=\cS(J(x)), \quad x\in \bA$$
and
$$\cS(x)=\cS(x^*)=\cS(J_\bB(x)), \quad x\in \bB.$$
\end{lemma}

\begin{theorem}[Donoho-Stark uncertainty principle]\label{thm:DS}
Let $(\bA, \bB, \fF, d, \tau)$ be a fusion bialgebra.
Then for any $0\neq x\in \bA$, we have
$$\cS (x)\cS(\fF(x))\geq 1;$$
for any $0\neq x\in \bB$, we have
$$\cS (x)\cS(\tfF(x))\geq 1;$$
\end{theorem}
\begin{proof}
The second inequality is the reformualtion of the first one.
We only have to prove the first one.
In fact,
\begin{align*}
\|\fF(x)\|_{\infty,\bB} &\leq \|x\|_{1, \bA}\leq \|\cR (x)\|_{2, \bA}\|x\|_{2, \bA} \quad \text{Proposition~\ref{prop:hyinf1},~\ref{prop:holder}}\\
&=\cS(x)^{1/2}\|\fF(x)\|_{2,\bB}\\
&\leq \cS(x)^{1/2}\|\cR(\fF(x))\|_{2,\bB}\|\fF(x)\|_{\infty,\bB} \quad \text{ Proposition~\ref{prop:holder}}\\
&=\cS(x)^{1/2}\cS(\fF(x))^{1/2}\|\fF(x)\|_{\infty,\bB},
\end{align*}
i.e. $\cS (x)\cS(\fF(x))\geq 1$.
This completes the proof of the theorem.
\end{proof}

For any $x\in \bB$, the von Neumann entropy $H(|x|^2)$ is defined by $H(|x|^2)=-\tau(x^*x\log x^*x)$ and for any $x\in \bA$ the von Neumann entropy is defined by $H (|x|^2)=-d\left((x^\#\diamond x)\log(x^\#\diamond x ) \right)$.

\begin{theorem}[Hirschman-Beckner uncertainty principle]\label{thm:HB}
Let $(\bA, \bB, \fF, d, \tau)$ be a fusion bialgebra.
Then for any $x\in\bA$, we have
$$H(|x|^2)+H(|\fF(x)|^2)\geq -4\|x\|_{2,\bA}^2\log\|x\|_{2,\bA}.$$
\end{theorem}
\begin{proof}
We assume that $x\neq 0$.
Let $f(p)=\log\|\fF(x)\|_{p,\bB}-\log\|x\|_{q, \bA}$, where $p\geq 2$ and $1/p+1/q=1$.
By using the computations in the proof of Theorem 5.5 in \cite{JLW16}, we have
$$\left. \frac{d}{dp}\|\fF(x)\|_{p, \bB}^p\right|_{p=2}=-\frac{1}{2} H(|\fF(x)|^2)$$
and
$$\left. \frac{d}{dp}\log\|\fF(x)\|_{p,\bB} \right|_{p=2}=-\frac{1}{4}\log\|\fF(x)\|_{2,\bB}^2-\frac{1}{4}\frac{H(|\fF(x)|^2)}{\|x\|_{2, \bA}^2}.$$
We obtain that
$$f'(2)=-\frac{1}{2}\log\|x\|_{2,\bA}^2-\frac{1}{4\|x\|_{2, \bA}^2}(H(|\fF(x)|^2)+H(|x|^2)).$$
By Proposition~\ref{prop:planch}, we have that $f(2)=0$.
By Theorem~\ref{thm:hy}, we have $f(p)\leq 0$ for $p\geq 2$.
Hence $f'(2)\leq 0$ and
$$H(|\fF(x)|^2)+H(|x|^2)\geq -4\|x\|_{2,\bA}^2\log\|x\|_{2,\bA}.$$
\end{proof}

\begin{remark}
Let $(\bA, \bB, \fF, d, \tau)$ be a fusion bialgebra.
The Hirschman-Beckner uncertainty principle is also true for $x\in \bB$ with respect to the Fourier transform $\tfF$.
\end{remark}

We give a second proof of Theorem \ref{thm:DS}:

\begin{proof}
By using Theorem \ref{thm:HB} and the inequality $\log \cS(x)\geq H(|x|^2)$, for any $\| x\|_{2,\bA}=1$, we see that Theorem \ref{thm:DS} is true.
\end{proof}

For any $x\in \bA$ or $\bB$ and $t\in (0,1)\cup (1, \infty)$, the R\'{e}nyi entropy $H_t(x)$ is defined by
$$H_t(x)=\frac{t}{1-t}\|x\|_t.$$
Then $H_t(x)$ are decreasing function (see Lemma 4.3 in \cite{LiuWuscm}) with respect to $t$ for $\|x\|_{\infty,\bA} \leq 1$ and $\|x\|_{\infty, \bB}\leq 1$ respectively.

\begin{theorem}[R\'{e}nyi entropic uncertainty principles]
Let $(\bA, \bB, \fF, d, \tau)$ be a fusion bialgebra, $1\leq t, s\leq \infty$.
Then for any $x\in \bA$ with $\|x\|_{2,\bA}=1$, we have
$$(1/t-1/2)H_{t/2}(|\fF(x)|^2)+(1/2-1/s)H_{s/2}(|x|^2)\geq -\log K(1/t,1/s).$$
\end{theorem}
\begin{proof}
The proof is similar to the proof of Proposition 4.1 in \cite{LiuWuscm}, using Theorem~\ref{thm:normfourier}.
\end{proof}

\section{Young's Inequality}\label{Sec: Young}
In this section, we study Young's inequality for the dual of fusion bialgebra and the connections between Young's inequality and the Schur product property.

\begin{proposition}\label{prop:young1}
Let $(\bA, \bB, \fF, d, \tau)$ be a fusion bialgebra.
Then for any $x, y\in\bA$, we have
$$\|x*y\|_{\infty, \bA}\leq \|x\|_{\infty,\bA}\|y\|_{1,\bA}.$$
\end{proposition}
\begin{proof}
For any $x=\sum_{j=1}^m \lambda_j\fF^{-1}(x_j)$ and $y=\sum_{j=1}^m \lambda_j'\fF^{-1}(x_j)$, we have
\begin{align*}
\|x*y\|_{\infty, \bA}
&= \left\| \sum_{j,k=1}^m\lambda_j\lambda_k' \fF^{-1}(x_jx_k)\right\|_{\infty, \bA} = \left\| \sum_{j,k,s=1}^m\lambda_j\lambda_k'  N_{j,k}^s \fF^{-1}(x_s)\right\|_{\infty, \bA} \\
&=\max_{1\leq s\leq m}\frac{\left|\sum_{j,k=1}^m\lambda_j\lambda_k' N_{j,k}^s\right|}{d(x_s)}
\leq \max_{1\leq s\leq m}\frac{\sum_{j,k=1}^m \left|\lambda_j\lambda_k' N_{j,k}^s\right|}{d(x_s)}\\
&\leq \max_{1\leq j\leq m}\frac{|\lambda_j|}{d(x_j)}\max_{1\leq s\leq m}\frac{\sum_{j,k=1}^m \left|d(x_j)\lambda_k' N_{j,k}^s\right|}{d(x_s)}\\
&=\|x\|_{\infty, \bA}\sum_{k=1}^m |\lambda_k'|d(x_k)
=\|x\|_{\infty,\bA}\|y\|_{1,\bA}.
\end{align*}
This completes the proof of the proposition.
\end{proof}

\begin{remark}\label{rk:young1}
It would be natural to ask whether the following Young's inequality for the dual $(\bB, \bA, \tfF, \tau, d)$
\begin{equation}\label{eq:young12}
\|x*_\bB y\|_{\infty, \bB}\leq \|x\|_{\infty, \bB}\|y\|_{1, \bB}
\end{equation}
holds in general, but it does not, because we will see that it implies the Schur product property on the dual, which does not hold on many examples provided by Theorem~\ref{thm:SchurCounterEx}, Subsections~\ref{Sec:SmallRank} and~\ref{sub:simple}.
\end{remark}

\begin{proposition}\label{rem:young1}
Suppose that $\fF^{-1}(x)> 0$, $x\in \bB$, we actually have that Inequality (\ref{eq:young12}) is true.
Hence
\begin{equation}\label{eq:young13}
\|x*_\bB y\|_{\infty, \bB}\leq 4 \|x\|_{\infty, \bB}\|y\|_{1, \bB}.
\end{equation}
\end{proposition}
\begin{proof}
Let $x=\sum_{j=1}^m\lambda_j \fF(x_j)$ with $\lambda_j\geq 0$ and $y=\sum_{j=1}^m \lambda_j' \fF(x_j)$.
Then $\|x\|_{\infty,\bB}=\sum_{j=1}^m \lambda_j d(x_j)$ and
\begin{align*}
\|x*_\bB y\|_{\infty,\bB}&=\left\| \sum_{j=1}^m\lambda_j\lambda_j'd(x_j)^{-1} x_j\right\|_{\infty, \bB}
 \leq \sum_{j=1}^m | \lambda_j\lambda_j'|\\
& = \sum_{j=1}^m \lambda_jd(x_j)\frac{|\lambda_j'|}{d(x_j)}
\leq  \max_{1\leq j\leq m} \frac{|\lambda_j'|}{d(x_j)}\sum_{j=1}^m \lambda_jd(x_j)\\
&=\|\fF^{-1}(y)\|_{\infty, \bA}\|x\|_{\infty, \bB}
 \leq \|y\|_{1, \bB}\|x\|_{\infty, \bB} \quad \text{Proposition~\ref{prop:hyinf1}}
\end{align*}
Inequality~(\ref{eq:young13} follows directly by the fact that any element is a linear combination of four positive elements.
\end{proof}

\begin{proposition}
Let $(\bA, \bB, \fF, d, \tau)$ be a fusion bialgebra with $\bB$ commutative.
Then the Schur product property for the dual $(\bB, \bA, \tfF, \tau, d)$ implies inequality (\ref{eq:young12}).
\end{proposition}
\begin{proof}
By Definition~\ref{Def: Fusion bialgebras}, $(\bB, \bA, \tfF, \tau, d)$ is a fusion bialgebra. The proposition follows from Proposition~\ref{prop:young1}.
\end{proof}

\begin{proposition}
Let $(\bA, \bB, \fF, d, \tau)$ be a fusion bialgebra. If $\|\Delta\|\leq 1$ then Inequality (\ref{eq:young12}) holds.
\end{proposition}
\begin{proof}
As in the proof of Proposition \ref{prop:comulti3}, for any $x, y\in \bB$, we have that $\overline{x*_\bB y}=(\iota\otimes \tau)(\Delta(x)(1\otimes \overline{y}))$.
Then
\begin{align*}
\|(\iota\otimes \tau)(\Delta(x)(1\otimes y)) \|_{\infty, \bB}
&=\sup_{\|z\|_{1, \bB}=1}\left|(\tau\otimes \tau)(\Delta(x)(z\otimes \overline{y}))\right|\\
&\leq \sup_{\|z\|_{1, \bB}=1}\|\Delta(x)\|_{\infty, \bB} \|z\|_{1, \bB}\| y\|_{1, \bB}\\
&\leq \|x\|_{\infty, \bB}\|y\|_{1, \bB},
\end{align*}
This completes the proof.
\end{proof}

\begin{proposition}\label{prop:young2}
Let $(\bA, \bB, \fF, d, \tau)$ be a fusion bialgebra.
Then for any $x, y\in\bA$, we have
$$\|x*y\|_{1, \bA}=\|x\|_{1,\bA}\|y\|_{1,\bA}.$$
\end{proposition}
\begin{proof}
Suppose $x=\sum_{j=1}^m \lambda_j \fF^{-1}(x_j)$ and $y=\sum_{j=1}^m \lambda_j' \fF^{-1}(x_j)$.
Then
\begin{align*}
\|x*y\|_{1, \bA}&=\left\| \sum_{j=1, k=1,s=1}^m \lambda_j \lambda_k'N_{j,k}^s \fF^{-1}(x_s) \right\|_{1,\bA}
=\sum_{j=1, k=1,s=1}^m |\lambda_j \lambda_k'|N_{j,k}^s d(x_s)\\
&=\sum_{j=1, k=1}^m |\lambda_j \lambda_k'| d(x_j)d(x_k)
=\sum_{j=1}^m |\lambda_j| d(x_j) \sum_{k=1}^m|\lambda_k'| d(x_k)\\
&= \|x\|_{1, \bA} \|y\|_{1, \bA},
\end{align*}
i.e. $\|x*y\|_{1, \bA}=\|x\|_{1,\bA}\|y\|_{1,\bA}$.
\end{proof}

\begin{proposition}\label{prop:young22}
The following two statements are equivalent for $C>0$:
\begin{itemize}
\item[(1)] For any $x, y\in \bB$,  $\|x*_\bB y\|_{\infty, \bB}\leq C\|x\|_{1, \bB}\|y\|_{\infty, \bB}$.
\item[(2)] For any $x, y\in \bB$,  $\|x*_\bB y\|_{1, \bB}\leq C\|x\|_{1, \bB}\|y\|_{1, \bB}$.
\end{itemize}
\end{proposition}

\begin{proof}
$(1)\Rightarrow (2)$:
\begin{align*}
\|x*_\bB y\|_{1, \bB} &=\sup_{\|z\|_\infty=1}\tau((x*_\bB y)z)\\
&=\sup_{\|\tilde{z}\|_\infty=1}\tau((J_\bB(x) *_\bB z^*)y^*)\quad \text{Proposition~\ref{prop:rotation}}\\
&\leq  \sup_{\|\tilde{z}\|_\infty=1}\|J_\bB(x) *_\bB \tilde{z}^*\|_{\infty, \bB}\|y\|_{1, \bB}\\
& \leq C\|x\|_{1, \bB}\|y\|_{1, \bB}.
\end{align*}

$(2)\Rightarrow (1)$:
\begin{align*}
\|x*_\bB y\|_{\infty, \bB} &=\sup_{\|z\|_1=1}\tau((x*_\bB y)z)\\
&=\sup_{\|\tilde{z}\|_1=1}\tau((J_\bB(x) *_\bB z^*)y^*)\quad \text{Proposition~\ref{prop:rotation}}\\
&\leq  \sup_{\|\tilde{z}\|_1=1}\|J_\bB(x) *_\bB \tilde{z}^*\|_{1, \bB}\|y\|_{\infty, \bB}\\
& \leq C\|x\|_{1, \bB}\|y\|_{\infty, \bB}.
\end{align*}


\end{proof}

\begin{proposition}\label{prop:young92}
Let $(\bA, \bB, \fF, d, \tau)$ be a fusion bialgebra.
The following statements:
\begin{enumerate}
\item[(1)] the Schur product property holds on the dual.
\item[(2)] $\|x*_\bB y\|_{1, \bB}=\|x\|_{1,\bB}\|y\|_{1,\bB}.$ for any $x\geq 0, y\geq 0$ in $\bB$;
\item[(3)] $\|x*_\bB y\|_{1, \bB}\leq \|x\|_{1,\bB}\|y\|_{1,\bB}.$ for any $x, y$ in $\bB$;
\item[(4)] $\|x*_\bB y\|_{\infty, \bB}\leq \|x\|_{1,\bB}\|y\|_{\infty,\bB}.$ for any $x, y$ in $\bB$;
\end{enumerate}
satisfy that $(4) \Leftrightarrow (3)\Rightarrow (1)\Leftrightarrow (2)$.
\end{proposition}
\begin{proof}
Suppose $x=\sum_{j=1}^m \lambda_j x_j\geq 0$ and $y=\sum_{j=1}^m \lambda_j' x_j\geq 0$ in $\bB$.
Then
\begin{align*}
\|x*_\bB y\|_{1, \bB} \geq |\tau(x*_\bB y)|=\lambda_1\lambda_1' =\tau(x)\tau(y)=\|x\|_{1,\bB}\|y\|_{1, \bB}.
\end{align*}

$(1)\Rightarrow (2)$: By the Schur product property, we have $x*_\bB y\geq 0$ and $\|x*_\bB y\|_{1, \bB} =\tau(x*_\bB y)$.
This implies (2).

$(2)\Rightarrow (1)$: This implies that $\|x*_\bB y\|_{1, \bB} =\tau(x*_\bB y)$ for $x, y\geq 0$.
However this implies that $x*_\bB y\geq 0$, i.e. the Schur product property holds.

$(3)\Rightarrow (1)$: $(3)$ implies that  $\|x*_\bB y\|_{1, \bB} =\tau(x*_\bB y)$  for $x, y \geq 0$.
Thus the Schur product property holds.

$(3)\Leftrightarrow (4)$: It follows from Proposition \ref{prop:young22}.
\end{proof}

\begin{proposition}\label{prop:young3}
Let $(\bA, \bB, \fF, d, \tau)$ be a fusion bialgebra.
Then for any $x, y\in\bA$, $1\leq p\leq \infty$, we have
$$\|x*y\|_{p, \bA}\leq \|x\|_{p,\bA}\|y\|_{1,\bA}.$$
\end{proposition}
\begin{proof}
It follows from Proposition~\ref{prop:young1},~\ref{prop:young2} and Proposition~\ref{prop:inter}.
\end{proof}

%

\begin{proposition}\label{prop:young4}
Let $(\bA, \bB, \fF, d, \tau)$ be a fusion bialgebra.
Then for any $x, y\in\bA$, $1\leq p\leq \infty$, $1/p+1/q=1$, we have
$$\|x*y\|_{\infty, \bA}\leq \|x\|_{p,\bA}\|y\|_{q,\bA}.$$
\end{proposition}
\begin{proof}
We have
\begin{align*}
\|x*y\|_{\infty, \bA} &=\sup_{\|z\|_{1, \bA}=1}\left| d((x*y)\diamond z) \right|\\
&=\sup_{\|z\|_{1, \bA}=1} \left| d((J(z)*x^\#)\diamond J(y)) \right| \quad \text{Proposition~\ref{prop:rotation2}}\\
&\leq \sup_{\|z\|_{1, \bA}=1}\|J(z)* x^\#\|_{p, \bA} \|J(y)\|_{q,\bA} \quad \text{ Proposition~\ref{prop:holder}}\\
&\leq \|x^\#\|_{p, \bA} \|y\|_{q,\bA} \quad \text{ Proposition~\ref{prop:young3}}\\
&= \|x\|_{p, \bA} \|y\|_{q,\bA}.
\end{align*}
This completes the proof of the proposition.
\end{proof}

\begin{theorem}[Young's inequality]\label{thm:young}
Let $(\bA, \bB, \fF, d, \tau)$ be a fusion bialgebra.
Then for any $x, y\in\bA$, $1\leq p, q, r\leq \infty$, $1/p+1/q=1+1/r$, we have
$$\|x*y\|_{r, \bA}\leq \|x\|_{p, \bA}\|y\|_{q, \bA}.$$
\end{theorem}
\begin{proof}
It follows from Propositions~\ref{prop:young3},~\ref{prop:young4} and Proposition~\ref{prop:inter}.
\end{proof}

\begin{proposition}\label{prop:young91}
Let $(\bA, \bB, \fF, d, \tau)$ be a fusion bialgebra.
Then for any $x, y\in\bB$, $2\leq r\leq \infty$, $1\leq p, q\leq 2$, $1/p+1/q=1+1/r$, we have
$$\|x*_\bB y\|_{r, \bB}\leq \|x\|_{p, \bB}\|y\|_{q, \bB}.$$
\end{proposition}
\begin{proof}
For any $x, y\in \bB$, $1/r+1/r'=1$, $1/p+1/p'=1$, $1/q+1/q'=1$, we have
\begin{align*}
\|x*_\bB y\|_{r, \bB}
&=\|\fF(\tfF(x)\diamond \tfF(y))\|_{r, \bB} \\
&\leq \|\tfF(x)\diamond \tfF(y)\|_{r', \bA} \quad \text{Proposition~\ref{thm:hy}}\\
&=\|\tfF(x)\|_{p', \bA}\|\tfF(y)\|_{q', \bA} \quad \text{Proposition~\ref{prop:holder}}\\
&\leq \|x\|_{p, \bB}\|y\|_{q, \bB}. \quad \text{Proposition~\ref{thm:hy}}
\end{align*}
This completes the proof of the proposition.
\end{proof}

\begin{proposition}
Let $(\bA, \bB, \fF, d, \tau)$ be a self-dual fusion bialgebra.
Then Young's inequality holds on the dual.
\end{proposition}
\begin{proof}
It is directly from the definition.
\end{proof}

\begin{proposition}\label{prop:young7}
Let $(\bA, \bB, \fF, d, \tau)$ be a fusion bialgebra.
Then the following are equivalent:
\begin{enumerate}
\item $\|x*_\bB y\|_{r, \bB}\leq \|x\|_{p, \bB}\|y\|_{q, \bB}, $ $1\leq p, q, r\leq \infty$, $1/p+1/q=1+1/r$  for any $x, y\in\bB$;
\item $\|x*_\bB y\|_{1, \bB}\leq \|x\|_{1, \bB}\|y\|_{1, \bB}$  for any $x, y\in\bB$;
\item $\|x*_\bB y\|_{\infty, \bB}\leq \|x\|_{\infty, \bB}\|y\|_{1, \bB}$  for any $x, y\in\bB$.
\end{enumerate}
We say the dual has Young's property if one of the above statements is true.
\end{proposition}
\begin{proof}
It follows the similar proof of Proposition~\ref{prop:young22} and Proposition~\ref{prop:young3} and~\ref{prop:young4}.
\end{proof}

\begin{remark}
By Proposition~\ref{prop:young92}, we have that for the dual, Young's property implies Schur product property.
\end{remark}

\begin{proposition}\label{prop:range1}
Let $(\bA, \bB, \fF, d, \tau)$ be a fusion bialgebra.
Then for any $x, y\in \bA$, we have
$$\cR (x*y)\leq \cR (\cR(x)*\cR(y)).$$
In particular, $\cR (x*y)=\cR (\cR(x)*\cR(y))$ if $x\geq 0$, $y\geq  0$.
\end{proposition}
\begin{proof}
It follows from the Schur product property.
\end{proof}

\begin{remark}\label{prop:range2}
Let $(\bA, \bB, \fF, d, \tau)$ be a fusion bialgebra.
Suppose that the dual has Schur product property.
Then $\cR (x*_\bB y)=\cR (\cR(x)*_\bB \cR(y))$ if $x>0$, $y>0$ in $\bB$.
\end{remark}

\begin{proposition}
Young's property holds on a fusion bialgebra $(\bA, \bB, \fF, d, \tau)$ which can be subfactorized.
\end{proposition}
\begin{proof}
It follows from Theorem 4.13 in \cite{JLW16} for subfactors.
\end{proof}

\begin{theorem}[Sum set estimate]\label{thm:sumsetest}
Let $(\bA, \bB, \fF, d, \tau)$ be a fusion bialgebra.
Then for any $x, y\in\bA$, we have
$$\cS(\cR(x)*\cR(y))\geq \max\{\cS(x), \cS(y)\}.$$
\end{theorem}
\begin{proof}
We have that
\begin{align*}
\cS(x)\cS(y) &=\|\cR(x)\|_{1, \bA}\|\cR(y)\|_{1, \bA}\\
&=\|\cR(x)*\cR(y)\|_{1, \bA}\quad  \text{ Proposition~\ref{prop:young2}} \\
&\leq \|\cR(\cR(x)*\cR(y))\|_{2, \bA} \|\cR(x)*\cR(y)\|_{2, \bA}\quad \text{ Proposition~\ref{prop:holder}} \\
&\leq\cS(\cR(x)*\cR(y))^{1/2} \|\cR(x)\|_{1,\bA}\| \cR(y)\|_{2, \bA} \quad  \text{ Proposition~\ref{prop:young3}}\\
&=\cS(\cR(x)*\cR(y))^{1/2} \cS(x)\cS(y)^{1/2}.
\end{align*}
Hence $\cS(\cR(x)*\cR(y))\geq \cS(y)$.
\end{proof}

\begin{remark}
We thank Pavel Etingof for noticing us another proof of Theorem \ref{thm:sumsetest} from an algebraic point of view \cite{Etingof19}.
\end{remark}

\begin{theorem}[Sum set estimate]\label{thm:sumsetest2}
Let $(\bA, \bB, \fF, d, \tau)$ be a fusion bialgebra.
Suppose that the dual has the Schur product property.
Then for any $x, y\in\bB$, we have
$$\cS(\cR(x)*_\bB \cR(y))\geq \max\{\cS(x), \cS(y)\}.$$
\end{theorem}
\begin{proof}
By Proposition~\ref{prop:young91} and~\ref{prop:young92}, the proof is similar to the one of Theorem~\ref{thm:sumsetest}.
\end{proof}

\section{Fusion Subalgebras and Bishifts of Biprojections}\label{Sec: Fusion Subalgebras and Bishifts of Biprojections}

In this section, we define fusion subalgebras, biprojections and bishifts of biprojections for fusion bialgebras. We prove a correspondence between fusion subalgebras and biprojections. We prove partially that bishifts of biprojections are the extremizers of the inequalities proved in the previous sections.

\begin{definition}[Fusion subalgebra]
Let $(\bA, \bB, \fF, d, \tau)$ be a fusion bialgebra.
A subalgebra $\bA_0$ of $\bA$ is a fusion subalgebra if $(\bA_0, \fF(\bA_0), \fF, d, \tau)$ is a fusion bialgebra.
\end{definition}

\begin{definition}[Biprojection]
Let $(\bA, \bB, \fF, d, \tau)$ be a fusion bialgebra.
We say $x\in \fA$ is a biprojection if $x$ is projection and $\fF(x)$ is a multiple of a projection in $\bB$.
\end{definition}

\begin{proposition}\label{prop:biproj1}
Let $(\bA, \bB, \fF, d, \tau)$ be a fusion bialgebra and $P$ a biprojection.
Then there is a fusion subalgebra $\bA_0$ such that the range of $P$ is $\bA_0$.
\end{proposition}
\begin{proof}
We write $\fF(P)=\sum_{j=1}^m \lambda_j x_j$.
By the fact that $P$ is a projection and $\fF(P)$ is a multiple of a projection, we obtain that $\lambda_j=0$ or $\lambda_j=d(x_j)$,and
\begin{align}
\fF(P)^2 &=\lambda \fF(P), \quad \fF(P)^*= \frac{\overline{\lambda}}{\lambda}\fF(P)  \label{eq:proj}.
\end{align}
Solving the Equation (\ref{eq:proj}), we obtain that
\begin{align}\label{eq:solv}
\lambda \overline{\lambda_{j^*}}=\overline{\lambda}\lambda_j  \quad , \quad \lambda \lambda_s=\sum_{j,k=1}^m\lambda_j\lambda_k N_{j,k}^s.
\end{align}
Let
$$\bA_0=\text{span}\{\fF^{-1}(x_j):\lambda_j\neq 0\}$$
and
$$I_{\bA_0}=\{1\leq j\leq m: \lambda_j\neq 0\}\subset \{1, \ldots, m\}.$$
Then $\fF(P)=\sum_{j\in I_{\bA_0}} d(x_j) x_j$.
By Equations (\ref{eq:solv}), we have that
$$\lambda d(x_s)=\sum_{j,k\in I_{\bA_0}} d(x_j)d(x_k)N_{j,k}^s,$$
and
$$\lambda =\sum_{j,k\in I_{\bA_0}} \lambda_j\lambda_k\delta_{j^*, k}=\sum_{j\in I_{\bA_0}}|\lambda_j|^2=\sum_{j\in I_{\bA_0}} d(x_j)^2>0.$$
Let $\mu_{\bA_0}=\sum_{j\in I_{\bA_0}} d(x_j)^2$.
We have
\begin{align}\label{eq:result}
 d(x_s)=\mu_{\bA_0}^{-1} \sum_{j,k\in I_{\bA_0}} d(x_j)d(x_k)N_{j,k}^s, \quad \forall s\in I_{\bA_0}.
\end{align}
By Equation (\ref{eq:solv}) and (\ref{eq:result}), we have that the involution $*$ is invariant on $I_{\bA_0}$ and
\begin{align*}
\mu_{\bA_0}d(x_s) &=\sum_{j,k\in I_{\bA_0}} d(x_j)d(x_k) N_{j,k}^s\\
&=\sum_{j\in I_{\bA_0}} d(x_j)\sum_{k\in I_{\bA_0}} d(x_{k^*}) N_{s^*, j}^{k^*}\\
&\leq \sum_{j\in I_{\bA_0}} d(x_j)\sum_{k=1}^m d(x_{k^*}) N_{s^*, j}^{k^*}\\
&= \sum_{j\in I_{\bA_0}} d(x_j)d(x_{s^*})d(x_j)\\
&=\mu_{\bA_0} d(x_s),
\end{align*}
i.e. $N_{s,j}^k=0$ for any $k\notin I_{\bA_0}$.
Therefore $x_jx_k=\sum_{s\in I_{\bA_0}} N_{j,k}^s x_s$ for any $j,k\in I_{\bA_0}$, i.e. $\fF(\bA_0)$ is a $*$-algebra and $(\bA_0, \fF(\bA_0), \fF, d, \tau)$ is a fusion bialgebra.
\end{proof}

\begin{proposition}\label{prop:biproj2}
Let $(\bA, \bB, \fF, d, \tau)$ be a fusion bialgebra and $(\bA_0, \fF(\bA_0), \fF, d, \tau)$ is a fusion subalgebra.
Then there is a biprojection $P$ such that the range of $P$ is $\bA_0$.
\end{proposition}
\begin{proof}
Let $\{y_1, \ldots, y_{m'}\}$ be a $\bR_{\geq 0}$-basis of $\fF^{-1}(\bA_0)$ such that $y_1=1$, $y_j\in \bB$ and $y_jy_k=\sum_{s=1}^{m'}M_{j,k}^s y_s$, where $M_{j,k}^s\in \mathbb{N}$ and $M_{j,k}^1=\delta_{y_j^*, y_k}$.
Suppose that $y_j=\sum_{k=1}^{m}C_{j,k}x_k$, $C_{j,k}\in \bZ$.
Then
$$1=M_{j^*, j}^1=\tau(y_{j}^*y_j)=\sum_{k=1}^m C_{j,k}^2.$$
Hence $y_j=x_{m_j}$ for some $1\leq m_j\leq m$ and $M_{j,k}^s=N_{m_j, m_k}^{m_s}$ for $1\leq j,k,s\leq m'$.
Let $P=\sum_{j=1}^{m'} d(y_j)\fF^{-1}(y_j)$.
Then
\begin{align*}
\fF(P)^2 & =\mu_{\bA_0}\fF(P), \quad \fF(P)^*=\fF(P).
\end{align*}
Then $P$ is a biprojection.
\end{proof}

\begin{theorem}\label{thm:biproj}
Let $(\bA, \bB, \fF, d, \tau)$ be a fusion bialgebra.
Then there is a bijection between the set of fusion subalgebras and the set of biprojections.
\end{theorem}
\begin{proof}
It follows from Proposition~\ref{prop:biproj1} and~\ref{prop:biproj2}.
\end{proof}

\begin{definition}[Left and right shifts]\label{def:shift}
Let $(\bA, \bB, \fF, d, \tau)$ be a fusion bialgebra and $B$ a biprojection.
A projection $P\in \bB$ is a shift of $\cR(\fF(B))$ if $\tau(P)=\tau(\cR(\fF(B)))$ and $P*_\bB \cR(\fF(B))=\tau(\cR(\fF(B)))P$.
A projection $P\in \bA$ is a left shift of $B$ if $d(B)=d(P)$ and $P*B=d(B)P$;
A projection $P\in \bA$ is a right shift of $B$ if $d(B)=d(P)$ and $B*P=d(B)P$;
\end{definition}

\begin{lemma}\label{lem:extrem1}
Let $(\bA, \bB, \fF, d, \tau)$ be a fusion bialgebra and $B$ a biprojection.
Then
\begin{enumerate}[(1)]
\item $\cR(\fF(B))$ is a shift of $\cR(\fF(B))$, $B$ is a left (right) shift of $B$;
\item $J_\bB(P)$ is a shift of $\cR(\fF(B))$ when $P$ is a shift of $\cR(\fF(B))$;
\item $J(P)$ is a left (right) shift of $B$ when $P$ is a right (left) shift of $B$;
\item $\cS(P)\cS(\fF(P))=1$ when $P$ is a left (right) shift of $B$ or $\cS(\fF^{-1}(P))\cS(P)=1$ when $P$ is a shift of $\cR(\fF(B))$;
\item $\cR(\fF^{-1}(P))=B$ if $P$ is a left (right) shift of $\cR(\fF(B))$ and $\cR(\fF(P))=\cR(\fF(B))$ if $P$ is a shift of $B$.
\end{enumerate}
\end{lemma}
\begin{proof}
$(1)$ By Proposition~\ref{prop:biproj1}, we have $\cR(\fF(B))=d(B)^{-1}\fF(B)$,
$$B*B=d(B)B, \quad \cR(\fF(B))*_\bB \cR(\fF(B))=d(B)^{-1}\cR(\fF(B))$$
and
$$\cS(B)\cS(\fF(B))=\tau(\cR(\fF(B)))d(B)=\frac{1}{d(B)}d(B)=1.$$
It indicates that $\cR(\fF(B))$ is a shift of $\cR(\fF(B))$ and $B$ is a left (right) shift of $B$

$(2)$ and $(3)$ can be followed by the property of $J$ and $J_\bB$.

$(4)$ Suppose $P$ is a shift of $\cR(\fF(B))$.
Then $\cR(\fF^{-1}(P))\leq B.$ and
$$1\leq \cS(P)\cS (\fF^{-1}(P))=\tau(P)d(\cR(\fF^{-1}(P)))\leq \cS(\fF(B))d(B)=1,$$
i.e. $\cS(P)\cS (\fF^{-1}(P))=1$ and $\cR(\fF^{-1}(P))= B$.

Suppose $P$ is a right shift of $B$.
Then $\cR(\fF(P))\leq \cR(\fF(B))$ and
$$1\leq \cS(P)\cS(\fF(P))\leq \cS(B)\cS(\fF(B))=1,$$
i.e. $\cS(P)\cS(\fF(P))=1$.
Hence $\cR(\fF(P))=\cR(\fF(B))$.

Suppose $P$ is a left shift of $B$.
Then $J(P)$ is a right shift of $B$.
By Lemma~\ref{lem:support}, we have $\cS(P)\cS (\fF(P))=1$.
\end{proof}

\begin{definition}[Bishift of biprojection]
Let $(\bA, \bB, \fF, d, \tau)$ be a fusion bialgebra and $B$ a biprojection.
A nonzero element $x$ is a bishift of the biprojection $B$ if there is $y\in \bA$, a shift $\tB_g$ of $\cR(\fF(B))$ and a right shift $B_{h}$ of $B$ such that $x=\fF^{-1}(\tB_{g}) *(y\diamond B_{h})$.
\end{definition}

\begin{lemma}\label{lem: extrem2}
Let $(\bA, \bB, \fF, d, \tau)$ be a fusion bialgebra and $B$ a biprojection.
Suppose $x=\fF^{-1}(\tB_{g}) *(y\diamond B_{h})$ is a bishift of a biprojection $B$.
Then $\cR(x)=\tB_g$, $\cR (\fF(x))=B_{h}$ and
$$\cS(x)\cS(\fF(x))=1.$$
\end{lemma}
\begin{proof}
By Proposition~\ref{prop:range1}, we have
\begin{align*}
\cR (x) \leq \cR (\cR(\fF^{-1}(B_g))*\cR(y\diamond B_{h}))
\leq  \cR (B* B_{h}) =B_h
\end{align*}
Then we obtain that
\begin{align*}
1\leq \cS(x)\cS(\fF(x)) \leq \cS(B_h)\cS(\tB_{g})= \cS(B)\cS(\fF(B))=1.
\end{align*}
Hence the inequalities above are equalities, i.e. $\cR(x)=B_h$, $\cR (\fF(x))=\tB_{g}$.
\end{proof}

\begin{definition}
Let $(\bA, \bB, \fF, d, \tau)$ be a fusion bialgebra.
An element $x\in \bB$ is said to be extremal if $\|\fF^{-1}(x)\|_{\infty, \bA}=\|x\|_{1,\bB}$.
An element $x\in \bA$ is said to be extremal if $\|\fF(x)\|_{\infty, \bB}=\|x\|_{1, \bA}$.
\end{definition}

\begin{definition}
Let $(\bA, \bB, \fF, d, \tau)$ be a fusion bialgebra.
An element $x\in \bA$ is a bi-partial isometry if $x$ and $\fF(x)$ are multiples of partial isometries.
An element $x\in \bA$ is an extremal bi-partial isometry if $x$ is a bi-partial isometry and $x$, $\fF(x)$ are extremal.
\end{definition}

\begin{theorem}
Let $(\bA, \bB, \fF, d, \tau)$ be a fusion bialgebra.
Then the following statements are equivalent:
\begin{enumerate}[(1)]
\item $H(|x|^2)+H (|\fF(x)|^2)=-4\|x\|_{2, \bA}^2\log \|x\|_{2, \bA}$;
\item $\cS(x)\cS(\fF(x))=1$;
\item $x$ is an extremal bi-partial isometry.
\end{enumerate}
\end{theorem}
\begin{proof}
The arguments are similar to the one of Theorem 6.4 in \cite{JLW16}, since only the Hausdorff-Young inequality is involved.
\end{proof}

\begin{proposition}\label{prop:rightshift}
Let $(\bA, \bB, \fF, d, \tau)$ be a fusion bialgebra and $w$ an extremal bi-partial isometry.
Suppose that $w$ is a projection.
Then $\tilde{w}$ is a right shift of a biprojection.
\end{proposition}
\begin{proof}
Let $w=\sum_{j\in J} d(x_j)\fF^{-1}( x_j)$.
Then
\begin{equation}\label{eq:partial1}
\|w\|_{\infty, \bA}=1, \quad \|w\|_{2, \bA}= \|\fF(w)\|_{2, \bB}=\left(\sum_{j\in J}d(x_j)^2\right)^{1/2}, \quad \|w\|_{1, \bA}=\sum_{j\in J} d(x_j)^2.
\end{equation}
By the assumption, we have
\begin{equation}\label{eq:partial2}
\|\fF(w)\|_{1, \bB}=\|w\|_{\infty, \bA}=1, \quad \|\fF(w)\|_{\infty, \bB}=\|w\|_{1, \bA}=\sum_{j\in J} d(x_j)^2.
\end{equation}
Let $P=\fF(w)\fF(w)^*$.
Then $P$ is a multiple of a projection in $\bB$ and
\begin{equation}\label{eq:partial3}
\|P\|_{\infty, \bB}=\left(\sum_{j\in J}d(x_j)^2\right)^2, \quad \|P\|_{1, \bB}=\|w\|_{2,\bA}^2=\sum_{j\in J}d(x_j)^2, \quad \|P\|_{2, \bB}=\left(\sum_{j\in J}d(x_j)^2\right)^{3/2}
\end{equation}
We will show that $\fF^{-1}(P)$ is a multiple of partial isometry.
By Corollary~\ref{cor:holder1}, we have to check $\|\fF^{-1}(P)\|_{2,\bA}^2=\|\fF^{-1}(P)\|_{\infty, \bA}\|\fF^{-1}(P)\|_{1, \bA}$.
In fact,
\begin{align*}
\left(\sum_{j\in J} d(x_j)^2 \right)^{3}&=\|\fF^{-1}(P)\|_{2, \bA}^2 \\
&\leq \|\fF^{-1}(P)\|_{\infty, \bA} \|\fF^{-1}(P)\|_{1, \bA} \quad \text{Equation (\ref{eq:partial3}) and Proposition~\ref{prop:holder}}\\
&\leq \|P\|_{1, \bB}\|w*J(w)\|_{1, \bA} \quad \text{Proposition~\ref{prop:hyinf1}}\\
&= \left(\sum_{j\in J}d(x_j)^2\right)\|w\|_{1, \bA}^2 \quad \text{Proposition~\ref{prop:young2}}\\
&= \left(\sum_{j\in J}d(x_j)^2\right)^3,
\end{align*}
i.e. $\|\fF^{-1}(P)\|_{2,\bA}^2=\|\fF^{-1}(P)\|_{\infty, \bA}\|\fF^{-1}(P)\|_{1, \bA}$ and $\fF^{-1}(P)$ is a multiple of a partial isometry.
By Schur product property, we have that $\fF^{-1}(P)>0$ and $\fF^{-1}(P)$ is a multiple of a projection.
Hence
\begin{equation}\label{eq:partial4}
\begin{aligned}
\|\fF^{-1}(P)\|_{\infty, \bA}&=\sum_{j\in J} d(x_j)^2, \\
\cR(\fF^{-1}(P))& = \left(\sum_{j\in J}d(x_j)^2\right)^{-1}\fF^{-1}(P),\\
d(\cR(\fF^{-1}(P)))&=\sum_{j\in J}d(x_j)^2.
\end{aligned}
\end{equation}
and $ \left(\sum_{j\in J}d(x_j)^2\right)^{-1}\fF^{-1}(P)$ is a biprojection

By Equation (\ref{eq:partial4}), we have
\begin{align*}
\cR(\fF^{-1}(P))* w&= \left(\sum_{j\in J}d(x_j)^2\right)^{-1}\fF^{-1}(P)*w
= \left(\sum_{j\in J}d(x_j)^2\right)^{-1}\fF^{-1}(P\fF(w))\\
&=\|w\|_{\infty, \bB}^2\left(\sum_{j\in J}d(x_j)^2\right)^{-1}w
=\left(\sum_{j\in J}d(x_j)^2\right) w\\
&= d(\cR(\fF^{-1}(P))) w=\|w\|_{1, \bA}w.
\end{align*}
Hence $w$ is a right shift of $\cR(\fF^{-1}(P))$.
\end{proof}

\begin{corollary}
Let $(\bA, \bB, \fF, d, \tau)$ be a fusion bialgebra.
Then a left shift of a biprojection is a right shift of a biprojection.
\end{corollary}
\begin{proof}
It follows from Lemma~\ref{lem:extrem1} and Proposition~\ref{prop:rightshift}.
\end{proof}

\begin{question}
Are the minimizers of the Donoho-Stark uncertainty principle bishifts of biprojections?
\end{question}

\begin{theorem}
Let $(\bA, \bB, \fF, d, \tau)$ be a fusion bialgebra.
Suppose that the dual has Young's property.
Then the minimizers of the Donoho-Stark uncertainty principle are bishifts of biprojections.
\end{theorem}
\begin{proof}
The proof is similar to the proof of Proposition~\ref{prop:rightshift}.
We leave the details to the reader.
\end{proof}

\begin{theorem}[Exact inverse sum set theorem]\label{thm:sumset}
Let $(\bA, \bB, \fF, d, \tau)$ be a fusion bialgebra and $P, Q$ projections in $\bA$.
Then the following are equivalent:
\begin{enumerate}[(1)]
\item $\cS(P*Q)=\cS(P)$;
\item $ \frac{1}{\cS(Q)}P*Q$ is a projection;
\item there is a biprojection $B$ such that $Q\leq B_{h}$ and $P=\cR(x*B)$ for some $x\neq 0$ in $\bA$, where $B_{h}$ is a right shift of $B$.
\end{enumerate}
\end{theorem}
\begin{proof}
$(1)\Rightarrow (2)$ By Theorem~\ref{thm:sumsetest}, we have that $\cS(P*Q)\geq \cS (P)$.
By the assumption and the proof of Theorem~\ref{thm:sumsetest}, we have
\begin{equation}\label{eq:sumset1}
\|P*Q\|_{1, \bA}=\|\cR(P*Q)\|_{2, \bA}\|P*Q\|_{2, \bA}, \quad \|P*Q\|_{2, \bA}=\|P\|_{2,\bA}\|Q\|_{1, \bA}.
\end{equation}
By Proposition~\ref{prop:holder}, we have that
$$P*Q=\lambda \cR(P*Q)$$
and
$$\lambda \cS(P*Q)^{1/2}=\|P*Q\|_{2, \bA}=\cS(P)^{1/2}\cS(Q).$$
Therefore $\lambda=\cS(Q)$ and $ \frac{1}{\cS(Q)}P*Q$ is a projection.

$(2)\Rightarrow (1)$ By the assumption and Proposition~\ref{prop:young2}, we have
$$\cS(P*Q)=\cS\left(\frac{1}{\cS(Q)}P*Q\right)=\frac{1}{\cS(Q)}d(P*Q)=\frac{\cS(P)\cS(Q)}{\cS (Q)}=\cS(P).$$

$(2)\Rightarrow (3)$ Let $P_1=\cR (P*Q)$.
Then by Proposition~\ref{prop:rotation2}, we have
\begin{align*}
d((P_1*J(Q))\diamond P) &= d((J(P)*P_1)\diamond Q)
= d((J(Q)*J(P))\diamond J(P_1))\\
&= d((P*Q)\diamond P_1)= d(P*Q)\\
&=d(P)d(Q)=d(P_1*J(Q)).
\end{align*}
Hence $\cR (P_1*J(Q))\leq P$.
But by Theorem~\ref{thm:sumsetest}, we have
\begin{align*}
d(\cR (P_1*J(Q)))=\cS (P_1*J(Q))\geq \cS (P_1)=\cS (P).
\end{align*}
Then
$$\frac{1}{\cS (J(Q))}P_1*J(Q)=\cR (P_1*J(Q))=P.$$
Expanding the expression, we have
$$\frac{1}{\cS (Q)^2}P* Q*J(Q)=P.$$
Note that $\|\fF(Q)\|_{\infty, \bB}\leq \|Q\|_{1, \bA}=\cS (Q)$.
Let
\begin{equation}\label{eq:sumset2}
\fF(B)=\lim_{n\to \infty} \frac{1}{\cS (Q)^{2n}}(\fF(Q)\fF(Q)^*)^n,
\end{equation}
Then $\fF(B)$ is the spectral projection of $\cS (Q)^{-2}\fF(Q)\fF(Q)^*$ corresponding to $ \frac{\|\fF(Q)\|_{\infty, \bB}^2}{\cS (Q)^2}$.
Moreover $B>0$, $P*B=P$, $P=\cR (P*B)$ , $B\neq 0$ and
\begin{equation}\label{eq:sumset3}
\|\fF(B)\|_{\infty, \bB}=1=\|B\|_{1, \bA}, \quad \|B\|_{\infty, \bA}\leq \|\fF(B)\|_{1, \bB}=\|\fF(B)\|_{2, \bB}^2=\|B\|_{2, \bA}^2
\end{equation}
Hence
$$\|B\|_{\infty, \bA}\|B\|_{1, \bA}\leq \|B\|_{2, \bA}^2\leq \|B\|_{\infty, \bA}\|B\|_{1, \bA} .$$
By Corollary~\ref{cor:holder1}, we have that $B$ is a multiple of a partial isometry and then $B$ is a multiple of a projection.
Therefore $\|B\|_{2, \bA}^{-2}B$ is a biprojection.

Let $Q_1=\cR (B*Q)$.
Then
\begin{align*}
\cR (Q_1*J(Q_1)) &=\cR (\cR  (B*Q)J(\cR (B*Q)))\\
&=\cR (\cR  (B*Q)\cR (J(Q)*B))\\
&=\cR  (B*Q*J(Q)*B)\quad  \text{Proposition~\ref{prop:range1}}\\
&=\cR (B).
\end{align*}
Hence $\cS (Q_1*J(Q_1))=\cS (B)$.
On the other hand, by Theorem~\ref{thm:sumsetest}, we have
\begin{align*}
\cS (Q_1*J(Q_1))\geq \cS (Q_1)=\cS (B*Q)\geq \cS (B).
\end{align*}
Now we obtain that $\cS (B*Q)=\cS (B)$.
By ``$(1)\Rightarrow (2)$", we have $Q_1=\frac{1}{\cS (Q)}\cR (B)*Q.$
Note that
$$d(Q_1)=\cS (Q_1)=\cS (B)=d(\cR (B))$$
and
$$\cR (B)*Q_1=\|B\|_{2, \bA}^{-2}B*Q_1=d(\cR (B))Q_1.$$
We see $Q_1$ is a right shift of $\cR (B)$.

Now we have to check $Q\leq  Q_1$.
Note that $\fF^{-1}(1_\bB)\leq \|B\|_{\infty, \diamond}B$.
Then by Proposition~\ref{prop:range1}, we have $Q\leq Q_1$.

$(3)\Rightarrow (1)$ By the assumption and Proposition~\ref{prop:range1}, we have
$$P* B=\cR (x*B*B)=P.$$
By Proposition~\ref{prop:rightshift}, we have
$$Q*J(Q)\leq B_{h}*J(B_h)=\cS (B)B.$$
Then by Proposition~\ref{prop:range1} and Theorem~\ref{thm:sumsetest}, we have
\begin{align*}
\cS (P)&\leq \cS (P*Q)=\cS (\cR (P*Q))\\
&\leq \cS (\cR (P*Q)*J(Q))=\cS (P*Q*J(Q))\\
&\leq \cS (P*B)=\cS (P).
\end{align*}
Hence $\cS (P*Q)=\cS (P)$.
\end{proof}

\begin{remark}
Let $(\bA, \bB, \fF, d, \tau)$ be a fusion bialgebra.
If the Schur product property holds on the dual, the results in Theorem~\ref{thm:sumset} are true for projections in $\bB$.
\end{remark}

Following the proofs in \cite{JLWscm}, one can obtain the following theorems:
\begin{theorem}[Extremizers of Young's inequality]
Let $(\bA, \bB, \fF, d, \tau)$ be a fusion bialgebra.
Suppose the dual has Young's property,
Then the following are equivalent:
\begin{enumerate}[(1)]
\item $\|x*y\|_{r, \bA}=\|x\|_{t, \bA}\|y\|_{s, \bA}$ for some $1<r, t,s<\infty$ such that $1/r+1=1/t+1/s$;
\item $\|x*y\|_{r, \bA}=\|x\|_{t, \bA}\|y\|_{s, \bA}$ for any $1\leq r, t,s\leq \infty$ such that $1/r+1=1/t+1/s$;
\item $x$, $y$ are bishifts of biprojection such that $\cR (\fF(x))=\cR (\fF(y))$.
\end{enumerate}
\end{theorem}

\begin{theorem}[Extremizers of the Hausdorff-Young inequality]
Let $(\bA, \bB, \fF, d, \tau)$ be a fusion bialgebra.
Suppose the dual has Young's property,
Then the following are equivalent:
\begin{enumerate}[(1)]
\item $\|x\|_{\frac{t}{t-1}, \bB}=\|x\|_{t, \bA}$ for some $1<t<2$;
\item $\|x\|_{\frac{t}{t-1}, \bB}=\|x\|_{t, \bA}$ for any $1\leq t\leq 2$;
\item $x$ is a bishift of a biprojection.
\end{enumerate}
\end{theorem}

\section{Quantum Schur Product Theorem on Unitary Fusion Categories}
In this section, we reformulate the quantum Schur product theorem (Theorem 4.1 in \cite{Liuex}) in categorical language. Planar algebras can be regarded as a topological axiomatization of pivotal categories (or 2-category in general). Subfactor planar algebras satisfy particular conditions designed for subfactor theory, see Page 9-13 of \cite{Jon99} for Jones' original motivation.
A subfactor planar algebra is equivalent to a rigid $C^*$-tensor category with a Frobenius *-algebra. The correspondence between subfactor planar algebras and unitary fusion categories was discussed by M\"{u}ger, particularly for Frobenius algebras in \cite{Mug03I} and for the quantum double in \cite{Mug03II}.

Let $\sD$ be a unitary fusion category, (or a rigid $C^*$-tensor category in general). Let $(\gamma,m,\eta)$ be a Frobenius *-algebra of $\sD$, $\gamma$ is an object of $\sD$, $m \in \hom_{\sD}(\gamma \otimes \gamma,\gamma)$, $\eta \in \hom_{\sD}(1,\gamma)$, where $1$ is the unit object of $\sD$, such that $(\gamma,m,\eta)$ is a monoid object and $(\gamma,m^*,\eta^*)$ is a comonoid object.
Let $\cup_{\gamma}=\eta^*m$ be the evaluation map and $\cap_{\gamma}=m^*\eta$ be the co-evaluation map. Then $\cup_{\gamma}^*=\cap_{\gamma}$.

We construct a quintuple $(\bA, \conv, J, d, \tau)$ from the Frobenius algebra:
Take the $C^*$ algebra
\[
\bA=\hom_{\sD}(\gamma,\gamma),
\]
with the ordinary multiplication and adjoint operation.
For $x,y\in \bA$, their convolution is
\[
x*y=m(x\otimes y) m^*.
\]
The modular conjugation $J$ is the restriction of the dual map of $\sD$ on $\bA$.
The Haar measure is
\[
d(x)= \cup_{\gamma} (x \otimes 1_\gamma) \cap_{\gamma},
\]
where $1_{\cdot}$ is the identity map on the object $\cdot$.
The Dirac measure is
\[
\tau(x) = \eta^* x \eta.
\]

We reformulate the quantum Schur product theorem on subfactor planar algebras and its proof as follows:
\begin{theorem}[Theorem 4.1 in \cite{Liuex}]\label{Thm: QSP for Frobenius algebras}
Given a Frobenius algebra $(\gamma,m,\eta)$ of a rigid $C^*$-tensor category, for any $x, y \in \bA:=\hom_{\sD}(\gamma,\gamma)$, $x,y>0$, we have that
\[
x*y:=m(x\otimes y) m^*>0.
\]
\end{theorem}

\begin{proof}
Let $\sqrt{x}$ and $\sqrt{y}$ be the positive square roots of the positive operators $x$ and $y$ respectively.
Then
\[
x*y=((\sqrt{x} \otimes \sqrt{y}) m^*)^*((\sqrt{x} \otimes \sqrt{y}) m^*) \geq 0.
\]
Note that $d$ is a faithful state, so $d(x), d(y)>0$. Moreover, $d(x*y)$ is a positive multiple of $d(x)d(y)$, so $d(x*y)>0$ and $x*y>0$.
\end{proof}

By the associativity of $m$ and $J(m)=m$, the vector space $\hom_{\sD}(\gamma,\gamma)$ forms another $C^*$-algebra $\bB$, with a multiplication $*$ and involution $J$. The identity map induces a unitary transformation $\fF: \bA \to \bB$, due to the Plancherel's formula,
\[
\tau(x*J(x))=\cup_{\gamma}(x \otimes J(x))\cap_{\gamma} = \cup_{\gamma}(x^*x \otimes 1_{\gamma})\cap_{\gamma}=d(x^*x).
\]

\begin{proposition}
When $\bA$ is commutative, the quintuple $(\bA, \bB, \fF, d, \tau)$ is a canonical fusion bialgebra.
\end{proposition}

\begin{proof}
The Schur product property follows from Theorem \ref{Thm: QSP for Frobenius algebras}.
The modulo conjugation property holds, as the duality map is an anti-linear *-isomorphism.
The Jones projection property holds, as $\fF(1_{1})$ is the identity of $\bB$. Moreover, $1_1$ is a minimal central projection and $d(1_1)=1$, so  $(\bA, \bB, \fF, d, \tau)$ is a canonical fusion bialgebra.
\end{proof}

Following the well-known correspondence between subfactor planar algebras and a rigid $C^*$ tensor category with a Frobenius *-algebra, we reformulate the subfactorization of a fusion bialgebra as follows:
\begin{definition}
A fusion bialgebra is subfactorizable if and only if it is the quintuple $(\bA, \bB, \fF, d, \tau)$ arisen from a Frobenius *-algebra in a rigid $C^*$ tensor category constructed above.
\end{definition}

There is another way to construct the dual $\bB$ using the the dual of $\sD$ w.r.t. the Frobenius algebra $\gamma$, which is compatible with the Fourier duality of subfactor planar algebras.
The dual $\hat{\sD}$ of $\sD$ w.r.t. the Frobenius algebra $(\gamma,m,\eta)$ is defined as the $\gamma-\gamma$ bimodule category over $\sD$, with the unit object $\gamma$. The dual Frobenius *-algebra of $(\gamma,m,\eta)$ is $(\hat{\gamma},\hat{m},\hat{\eta})$, $\hat{\gamma}=\gamma\otimes \gamma$, $\hat{m}=1_{\gamma}\otimes \cup_{\gamma} \otimes 1_{\gamma}$, $\hat{\eta}=m^*$.
Then the quintuple from the Frobenius algebra $\hat{\gamma}$ of $\hat{\sD}$ is dual to the quintuple from the Frobenius algebra
$\gamma$ of $\sD$. In particular, the $C^*$-algebra $\bB$ can be implemented by $\hom_{\hat{\sD}}(\hat{\gamma},\hat{\gamma})$ with ordinary multiplication and adjoint operation.

Let $\sC$ be a unitary fusion category. Take $\sD=\mathscr{C} \boxtimes \overline{\mathscr{C}}$, then $\sD$ has a canonical Frobenius algebra $(\gamma,m,\eta)$.
Here $\displaystyle \gamma=\bigoplus_{j=1}^m X_j \otimes \overline{X_j}$, where $\{X_1, X_2, \ldots, X_m\}$ is the set of irreducible (or simple) objects of $\mathscr{C}$, $X_1$ is the unit and $\overline{X_j}=X_{j^*}$;
\[
m=\FPdim(\sC)^{1/4}\bigoplus_{j,k,s=1}^m  (\FPdim(X_j)\FPdim(X_k)\FPdim(X_s))^{1/2} \sum_{\alpha \in ONB(X_j,X_k;X_s)} \alpha \boxtimes \overline{\alpha},
\]
where $\FPdim(X_j)$ is the Frobenius-Perron dimension of $X_j$, $\displaystyle \FPdim(\sC)=\sum_{j=1}^m \FPdim(X_j)^2$ is the Frobenius-Perron dimension of $\sC$ and $ONB(X_j,X_k;X_s)$ is an orthonormal basis of $\hom_{\sC}(X_j\otimes X_k, X_s)$;
and $\FPdim(\sC)^{1/4}\eta \in \hom_{\sD}(1, \gamma)$ is the canonical inclusion (in particular, $\hat{\gamma}$ is the image of the unit of $\sC$ under the action of the adjoint functor of the forgetful functor from $Z(\sC)$ to $\sC$).
Its dual $\hat{\sD}$ is isomorphic to the Drinfeld center $Z(\sC)$ of $\sC$ as a fusion category. This construction is well-known as the quantum double construction. Consequently,

\begin{proposition}
Let $\cR$ be the Grothendieck ring of a unitary fusion category $\sC$. Then the canonical fusion bialgebra associated to the fusion ring $\cR$ is isomorphic to the one $(\bA, \bB, \fF, d, \tau)$ associated to the canonical Frobenius algebra $\gamma$ of $\sC \otimes \overline{\sC}$ in the quantum double construction. So it is subfactorizable.
\end{proposition}

\begin{proof}
Following the notations above, we take $x_j:=\FPdim(X)^{-1} 1_{X_j} \boxtimes 1_{\overline{X_j}}$. Then
\begin{align*}
d(x_j)&=PF(X_j) \; \\
x_j \diamond x_k&=\delta_{j,k} d(x_j)^{-1} x_j \; ; \\
(x_j)^{\#}&= x_j \; ;\\
\tau(x_j)&=\delta_{1,j} \; ; \\
J(x_j)&=x_{j^*} \; ;\\
x_j*x_k&=\sum_{k \in Irr} \dim  \ket{Z} \; ,\\
\end{align*}
where $\diamond$ and $\#$ are the ordinary multiplication and adjoint operator on the commutative $C^*$-algebra $\bA=\hom_{\sC}(\gamma,\gamma)$ respectively, $\delta_{j,k}$ is the Kronecker delta and $N_{j,k}^s=\hom_{\sC}(X_j\otimes X_k,X_s)$.
Therefore, the fusion bialgebra associated to the Grothendieck ring is isomorphic to the fusion bialgebra arisen from the canonical Frobenius algebra $(\gamma, m ,\eta)$ of $\sC\boxtimes \overline{\sC}$ in the quantum double construction. So it is subfactorizable.
\end{proof}

\begin{remark}
To encode the fusion rule of $\sC$ as the convolution on $\bA$ exactly, our normalization of the Frobenius algebra $(\gamma,m,\eta)$ is slightly different from the usual one identical to planar tangles in planar algebras, see e.g. Equation (18), Propositions 4.1 and 4.2 in \cite{Liu19}.
\end{remark}

\begin{proposition}\label{Prop: QSP dual of G}
Let $(\bA, \bB, \fF, d, \tau)$ be the canonical fusion bialgebra associated with the Grothendieck ring $\cR$ of the unitary fusion category $\sC$.
Then the Schur product property holds on $\bB$, the dual of the fusion ring $\cR$.
\end{proposition}

\begin{proof}
Applying the quantum Schur product theorem, Theorem \ref{Thm: QSP for Frobenius algebras}, to the Frobenius algebra $(\hat{\gamma}, \hat{m},\hat{\eta})$ of the Drinfeld center $Z(\sC)$, we obtain the Schur product property on $\bB$.
\end{proof}

We obtain an equivalent statement on $\bA$ as follows (see another equivalent statement in Proposition \ref{SchurRef}):
\begin{proposition}
Let $(\bA, \bB, \fF, d, \tau)$ (or $(\bA, \conv, J, d, \tau)$ as in Remark \ref{Rem: Equivalent Definition of Fusion Bialgebras}) be the canonical fusion bialgebra associated with the Grothendieck ring $\cR$ of the unitary fusion category $\sC$.
Then
\begin{align*}
d((J(x)*x)\diamond (J(y)*y) \diamond (J(z)*z)) \geq 0, ~\forall x,y,z \in \bA.
\end{align*}
\end{proposition}

\begin{proof}
It follows from Propositions \ref{SchurProp} and \ref{Prop: QSP dual of G}.

We give a second proof without passing through the Drinfeld center $Z(\sC)$.

\[
d((J(x)*x)\diamond (J(y)*y) \diamond (J(z)*z))=
\raisebox{-1.5cm}{
\begin{tikzpicture}
\draw (0,0.5-.2)--++(0,.4);
\draw (0,1.5-.2)--++(0,.4);
\draw (0,2.5-.2)--++(0,.4);
\draw (2,0.5-.2)--++(0,.4);
\draw (2,1.5-.2)--++(0,.4);
\draw (2,2.5-.2)--++(0,.4);
\draw (0,0.5-.2)--++(1,-.1);
\draw (0,1.5-.2)--++(1,-.1);
\draw (0,2.5-.2)--++(1,-.1);
\draw (0,0.5+.2)--++(1,+.1);
\draw (0,1.5+.2)--++(1,+.1);
\draw (0,2.5+.2)--++(1,+.1);
\draw (2,0.5-.2)--++(-1,-.1);
\draw (2,1.5-.2)--++(-1,-.1);
\draw (2,2.5-.2)--++(-1,-.1);
\draw (2,0.5+.2)--++(-1,+.1);
\draw (2,1.5+.2)--++(-1,+.1);
\draw (2,2.5+.2)--++(-1,+.1);
\draw (1,.8)--++(0,.4);
\draw (1,1.8)--++(0,.4);
\draw (1,2.8)--++(0,.2) arc (180:0: .7 and .2) --++ (0,-3) arc (0:-180: .7 and .2) --++(0,.2);
\node at (1-.3,0.9) {$m^*$};
\node at (1-.3,1.9) {$m^*$};
\node at (1-.3,2.9) {$m^*$};
\node at (1-.4,0.1) {$m$};
\node at (1-.4,1.1) {$m$};
\node at (1-.4,2.1) {$m$};
\node at (-.3,.5) {$J(x)$};
\node at (-.3,1.5) {$J(y)$};
\node at (-.3,2.5) {$J(z)$};
\node at (2-.2,.5) {$x$};
\node at (2-.2,1.5) {$y$};
\node at (2-.2,2.5) {$z$};
\end{tikzpicture}}
=
\raisebox{-1.5cm}{
\begin{tikzpicture}
\draw (0,0+.1) --++ (0,3) arc (0:180: .5 and .2) --++(0,-3) arc (-180: 0: .5 and .2);
\draw (2,0+.3) --++ (0,3) arc (180: 0: .5 and .2) --++(0,-3) arc (0: -180: .5 and .2);;
\draw (0,1-.1) --++(2,0+.2);
\draw (0,2-.1) --++(2,0+.2);
\draw (0,3-.1) --++(2,0+.2);
\node at (-.3,.5) {$J(x)$};
\node at (-.3,1.5) {$J(y)$};
\node at (-.3,2.5) {$J(z)$};
\node at (2-.2,.5) {$x$};
\node at (2-.2,1.5) {$y$};
\node at (2-.2,2.5) {$z$};
\node at (-.2,1-.1) {$m$};
\node at (-.2,2-.1) {$m$};
\node at (-.2,3-.1) {$m$};
\node at (2-.2,1+.2) {$m^*$};
\node at (2-.2,2+.2) {$m^*$};
\node at (2-.2,3+.2) {$m^*$};
\draw[red,dashed] (1,.1)--(1,3.3);
\end{tikzpicture}}
\geq 0 .
\]
The last inequality follows from reflection positivity of the horizontal reflection, namely the dual functor on $\sC$.

\end{proof}

\begin{remark}
Let us mention \cite{ENO21} which contains a reformulation of our first proof together with a discussion on some integrality properties of the numbers appearing in the Schur product criterion.
\end{remark}

In particular, if Grothendieck ring $\cR$ is commutative, then $\bB$ is commutative. There is a one-to-one correspondence between minimal projections $P_j$ in $\bB$ and characters $\chi_j$ of $\cR$, $j=1,2,\ldots, m$:
\[
P_j=\sum_{k=1}^n d(x_k)^{-1} \chi_j(x_k) x_{k^*}.
\]
Take
\[P_j*_{B}P_k=\sum_{s=1}^m \hat{N}_{j,k}^s P_s,\]
then $\hat{N}_{j,k}^s\geq0$, due to the Schur product property on $\bB$.



The dual of the fusion ring $\cR$ is independent of its categorification. The Schur product property may not hold on the dual of a fusion ring in general. Therefore, the Schur product property is an analytic obstruction of unitary categorification of fusion rings. We discuss its applications in \S \ref{Sec: Application}. Similarly, Young's inequality and sumset estimates are also analytic obstructions of unitary categorification of fusion rings.

\section{Applications and Conclusions}\label{Sec: Application}
In this section, we show that the Schur product property on the dual is an analytic obstruction for the unitary categorification of fusion rings. Furthermore, this obstruction is very efficient to rule out the fusion rings of high ranks (we apply it on simple integral fusion rings).
The inequalities for the fusion coefficients (Proposition~\ref{incoeff}) in the next subsection are essential for finding new fusion rings more efficiently.

\subsection{Upper Bounds on the Fusion Coefficients}
In this subsection, we obtain inequalities for fusion rings from the inequalities proved in previous sections.

\begin{proposition}  \label{incoeff}
Let $\fA$ be a fusion ring.
Then
\begin{enumerate}[(1)]
\item $\sum_{\ell=1}^m \left(N_{j,k}^\ell\right)^2\leq \min\{ d(x_j)^2, d(x_k)^2\}$;
\item $N_{j,k}^\ell\leq d(x_\ell) d(x_j)^{\frac{2-t}{t}}d(x_k)^{\frac{t-2}{t}}$ for any $t\geq 1$;
\item $N_{j,k}^\ell \leq \min\{d(x_j), d(x_k), d(x_\ell)\}$;
\item $\sum_{s=1}^m N_{j_1,j_2}^s N_{j_3, j_4}^s\leq \min_{j\neq j'\in\{j_1, j_2, j_3, j_4\}}d(x_{j}) d(x_{j'}).$
\end{enumerate}
\end{proposition}

\begin{proof}
Let $(\bA, \bB, \fF, d, \tau)$ be the fusion bialgebra arising from the fusion ring $\fA$.
By Theorem~\ref{thm:young}, we have for any $1/r+1=1/p+1/q$,
$$\left\|\sum_{\ell=1}^m N_{j,k}^\ell \fF^{-1}( x_\ell)\right\|_{r,\bA}=\|\fF^{-1}(x_j)*\fF^{-1}( x_k)\|_{r, \bA}\leq \|\fF^{-1}(x_j)\|_{p, \bA}\|\fF^{-1}(x_k)\|_{q, \bA}.$$
If $r<\infty$, then we obtain that
\begin{equation}\label{eq:young81}
\left(\sum_{\ell=1}^m d(x_\ell)^{2-r}\left( N_{j,k}^\ell \right)^r\right)^{1/r}\leq d(x_j)^{\frac{2-p}{p}}d(x_k)^{\frac{2-q}{q}}.
\end{equation}
If $r=\infty$, then we have
\begin{equation}\label{eq:young82}
\max_{1\leq \ell \leq m}\frac{N_{j,k}^\ell}{d(x_\ell)}\leq  d(x_j)^{\frac{2-p}{p}}d(x_k)^{\frac{2-q}{q}}.
\end{equation}
In Inequality (\ref{eq:young81}), let $r=2$, $p=1$, $q=2$, we have $\sum_{\ell=1}^m \left(N_{j,k}^\ell\right)^2\leq d(x_j)^2$; let $r=2$, $p=2$, $q=1$, we have $\sum_{\ell=1}^m \left(N_{j,k}^\ell\right)^2\leq d(x_k)^2$.
Hence
$$\sum_{\ell=1}^m \left(N_{j,k}^\ell\right)^2\leq \min\{ d(x_j)^2, d(x_k)^2\}.$$
This proves $(1)$.

In Inequality (\ref{eq:young82}), let $p=t$ and $q=\frac{t}{t-1}$ for any $t\geq 1$.
Then
$$N_{j,k}^\ell\leq d(x_\ell) d(x_j)^{\frac{2-t}{t}}d(x_k)^{\frac{t-2}{t}}.$$
This shows $(2)$ is true.

Take $p=q=2$ in Inequality (\ref{eq:young82}), we have $N_{j,k}^\ell\leq d(x_\ell)$.
By Equation (\ref{Equ: rotation}), we have
$$N_{j,k}^\ell\leq \min\{d(x_j), d(x_k), d(x_\ell)\}.$$
This indicates that $(4)$ is true.

By Theorem~\ref{thm:young} again, we have
\begin{align*}
\left\|\sum_{t=1}^m\sum_{s=1}^m N_{j,k}^s N_{s,\ell}^t \fF^{-1}(x_t)\right\|_{\infty, \bA}
&=\|\fF^{-1}(x_j)*\fF^{-1}(x_k)*\fF^{-1}(x_\ell)\|_{\infty, \bA}\\
&\leq \|\fF^{-1}(x_j)\|_{1, \bA} \|\fF^{-1}(x_k)\|_{2, \bA}\|\fF^{-1}(x_\ell)\|_{2, \bA}.
\end{align*}
Then
$$\frac{\sum_{s=1}^m N_{j,k}^s N_{s,\ell}^t}{d(x_t)}\leq d(x_j).$$
We have $\sum_{s=1}^m N_{j,k}^s N_{s,\ell}^t\leq d(x_t) d(x_j)$.
By Equation (\ref{Equ: rotation}), we have
$$\sum_{s=1}^m N_{j,k}^s N_{\ell, t}^s\leq d(x_t) d(x_j).$$
Note that $j,k,\ell, t$ can be interchanged, we see $(5)$ is true.
\end{proof}

\begin{proposition}
Let $\fA$ be a fusion ring.
Suppose that the fusion bialgebra arising from $\fA$ is self-dual.
Let $S$ be the $S$-matrix associated to $\fA$.
Then we have the following inequalities:
\begin{enumerate}[(1)]
\item $\sum_{j=1}^md(x_j)^{\frac{t-2}{t-1}}|S_{kj}|^{\frac{t}{t-1}}\leq d(x_k)^{\frac{2-t}{t-1}}d(\fA)^{\frac{t-2}{2t-2}},$ $1<t\leq 2$;
\item $d(x_j)^{t-2}d(\fA)^{1-t/2}\leq \sum_{k=1}^m|S_{jk}|^td(x_k)^{2-t},$ $1<t\leq 2$;
\item $|S_{jk}|\leq d(x_j)d(x_k)d(\fA)^{-1/2}$,
\end{enumerate}
where $d(\fA)$ is the Frobenius-Perron dimension of $\fA$.
\end{proposition}

\begin{proof}
It follows from the Hausdorff-Young inequalities.
\end{proof}

\subsection{Schur Product Property Reformulated}
In this subsection we reformulate Schur product property (on the dual) using the irreducible complex representation of the fusion algebra, which in the commutative case, becomes a purely combinatorial property of the character table.

Note that Proposition~\ref{prop:ufusionschur} states that if the fusion ring $\fA$ is the Grothendieck ring of unitary fusion category, then Schur product property holds on the dual of $\fA$, so it can be seen as a criterion for unitary categorification.


 \begin{proposition}[Non-Commutative Schur Product Criterion] \label{SchurRef}
The Schur product property holds on the dual of a fusion ring/algebra $\mathfrak{A}$ with basis $\{x_1 = 1, \dots, x_r \}$ if and only if for all triple of irreducible unital *-representations $(\pi_s,V_s)_{s=1,2,3}$ of the fusion ring/algebra $\mathfrak{A}$ over $\mathbb{C}$, and for all $v_s \in V_s$, we have
\begin{align}\label{Equ: Schur Criterion of Irreps}
\sum_i \frac{1}{d(x_i)} \prod_{s=1}^3 \left(v_s^* \pi_s(x_i) v_s\right) &\ge 0.
\end{align}
 \end{proposition}

\begin{proof}
Let $\mathfrak{A}$ be a fusion ring/algebra with basis $\{x_1 = 1, \dots, x_r \}$ and $(\bA, \bB, \fF, d, \tau)$ the fusion bialgebra arising from $\fA$.
By Proposition~\ref{SchurProp} and the fact that $d$ is multiplicative, the Schur product property holds on $\bB$ if and only if
$$ d\left(\prod_{s=1}^3 X_s *_\bA X^{*}_s\right) \ge 0 $$
for all $X_s \in \bA$.
Now, $X_s = \sum_i \alpha_{s,i}x_i$, so it is equivalent to
$$ \sum_i \frac{1}{d(x_i)} \prod_{s=1}^3 \left(\sum_{j,k} \alpha_{s,k} \overline{\alpha}_{s,j^{*}} N_{k,j^{*}}^i\right) \ge 0 $$
for all $\alpha_{s,i} \in \mathbb{C}$. Now let $M_i$ be the matrix $(N_{k,j^{*}}^i)$ which is also $(N_{i,j}^k)$ by Frobenius reciprocity, so that $M_i$ is the fusion matrix of $x_i$. Let $u_s$ be the vector $(\alpha_{s,i})$.  Then $$\sum_{j,k} \alpha_{s,k} \overline{\alpha}_{s,j^{*}} N_{k,j^{*}}^i  =  u_s^*M_iu_s.$$
Then the criterion is equivalent to have
 \begin{align}\label{Equ: Schur Criterion of Irreps 2}
 \sum_i \frac{1}{d(x_i)} \prod_{s=1}^3 (u_s^* M_i u_s) &\ge 0.
 \end{align}
for all $u_s \in \mathbb{C}^r$.
Recall that the map $\pi: x_i \to M_i$ is a unital *-representation of $\mathfrak{A}$.
So Equation \eqref{Equ: Schur Criterion of Irreps} implies Equation \eqref{Equ: Schur Criterion of Irreps 2}.
On the other hand, $\pi$ is faithful, so Equation \eqref{Equ: Schur Criterion of Irreps 2} implies Equation \eqref{Equ: Schur Criterion of Irreps}. \end{proof}

Assume that the fusion ring/algebra $\mathfrak{A}$ is commutative, then for all $i$, $x_i x_{i^*} = x_{i^*} x_i$, so that the fusion matrices $M_i$ are normal (so diagonalizable) and commuting, so they are simultaneously diagonalizable, i.e. there is an invertible matrix $P$ such that $P^{-1}M_iP = diag(\lambda_{i,1}, \dots, \lambda_{i,r})$, so that the maps $\pi_j: M_i \mapsto \lambda_{i,j}$ completely characterize the irreducible complex representations $\pi_j$ of $\mathfrak{A}$.  We can assume that $\pi_1=d$, so that $\lambda_{i,1} = d(x_i) = \Vert M_i \Vert$.
 \begin{definition} The matrix $\Lambda(\mathfrak{A}):=(\lambda_{i,j})$ is called the \emph{character table} of the commutative fusion ring $\mathfrak{A}$.
 \end{definition}
Note that for a finite group $G$, if $\mathfrak{A}_G$ is the Grothendieck ring of the unitary fusion category $\mathrm{Rep}(G)$, then $\Lambda(\mathfrak{A}_G)$ is the usual character table of $G$.
 \begin{corollary}[Commutative Schur Product Criterion]  \label{SchurCom}
 The Schur product property holds on the dual of a commutative fusion ring/algebra $\mathfrak{A}$ with character table $\Lambda=(\lambda_{i,j})$ if and only if for all triple $(j_1,j_2,j_3)$ $$\sum_i \frac{\lambda_{i,j_1}\lambda_{i,j_2}\lambda_{i,j_3}}{\lambda_{i,1}} \ge 0.$$
 \end{corollary}
 \begin{proof}
 Immediate from Proposition~\ref{SchurRef}, because here the irreducible representations are one-dimensional, so that there we have $v_s^* \pi_s(M_i) v_s = \Vert v_s \Vert^2 \pi_s(M_i)$.
 \end{proof}
\begin{remark}
The formula in \cite[Theorem 7.2.1]{Ser} looks similar to the one in Corollary~\ref{SchurCom} (applied to the group case), because the underlying planar diagrams are the same.
However, the results are different as the underlying planar algebras are so: the subfactor planar algebra of a finite group $G$ for the first formula, and for the second one, the dual of the quantum double of $Vec(G)$, namely the Drinfeld center of $Vec(G)$.
\end{remark}

In order to test the efficiency of Schur product criterion, we wrote a code computing the character table of a commutative fusion ring/algebra and checking whether Schur product property holds (on the dual) using Corollary~\ref{SchurCom}. The next two subsections presents the first results.

\subsection{Fusion Algebras of Small Rank} \label{Sec:SmallRank}
Ostrik\cite{Ost15} already classified the pivotal fusion category of rank $3$. In this section we would like to show how efficient is Schur product criterion in this case. We will next consider two families of rank $4$ fusion rings/algebras found by David Penneys and his collaborators\footnote{at the 2014 AMS MRC on Mathematics of Quantum Phases of Matter and Quantum Information.}\cite{Pen19}, and finally look to a family of rank $5$ fusion rings/algebras.

Recall \cite[Proposition 3.1]{Ost15} that a fusion ring $\mathfrak{A}$ of rank $3$ and basis $\{x_1 = 1, x_2, x_3 \}$ satisfies either $x_2^*=x_3$ and then is $\mathbb{C}C_3$, or $x_i^* = x_i$ and then is of the following form (extended to fusion algebras):
$$\left\{
\begin{array}{ll}
x_2x_2 = x_1 + px_2+mx_3, \\
x_2x_3 = mx_2+nx_3, \\
x_3x_3 = x_1 + nx_2+qx_3,
\end{array}
\right.$$
with $m,n,p,q \in \mathbb{R}_{\ge 0}$ and $m^2+n^2=1+mq+np$ (given by associativity). Note that $x_3x_2=x_2x_3$ by Frobenius reciprocity, so that the fusion algebra is commutative. We can assume (up to equivalence) that $m \le n$, and then $n > 0$ (because if $n=0$ then $m=0$ and the above associativity relation becomes $0=1$, contradiction), so that $p=(m^2+n^2-1-mq)/n$; and it is a fusion ring if and only if in addition $m,n,p,q \in \mathbb{Z}_{\ge 0}$ and $n$ divides $(m^2-1-mq)$. Recall \cite[Section 4.5]{Ost15} that it admits a pivotal categorification if and only if $(m,n,q) = (0, 1, 0), (0, 1, 1), (0, 1, 2),(1, 1, 1)$, and there is a unitary model for all of them.

Let $M_i$ be the fusion matrix of $x_i$, written below (with $p=(m^2+n^2-1-mr)/n$):
$$ \left(\begin{matrix}1&0&0  \\ 0&1&0  \\ 0&0&1\end{matrix} \right), \ \left(\begin{matrix}0&1&0  \\ 1&p&m  \\ 0&m&n \end{matrix} \right), \ \left(\begin{matrix}0&0&1  \\ 0&m&n  \\ 1&n&q \end{matrix} \right) $$
Let $\chi_i$ be the characteristic polynomial of $M_i$:
$$\chi_2(x) = x^3-(p+n)x^2+(pn-1-m^2)x+n$$
$$\chi_3(x) = x^3-(q+m)x^2+(qm-1-n^2)x+m$$
The matrix $M_i$ is self-adjoint thus its eigenvalues (and so the roots of $\chi_i$) are real. By using \cite[Theorem A.4]{Jan10}, we can deduce the following character table:
$$  \left[ \begin{matrix}
1&1&1 \\
\frac{b_2}{3}+2r_2c_2 & \frac{b_2}{3}-r_2(c_2-\sqrt{3}s_2)  & \frac{b_2}{3}-r_2(c_2+\sqrt{3}s_2) \\
\frac{b_3}{3}+2r_3c_3 & \frac{b_3}{3}-r_3(c_3+\sqrt{3}s_3) & \frac{b_3}{3}-r_3(c_3-\sqrt{3}s_3)
\end{matrix} \right]
$$
with $c_i = \cos(\frac{\varphi_i}{3})$, $s_i = \sin(\frac{\varphi_i}{3})$,  $\varphi_i= \arccos\left(\frac{q_i/2}{(p_i/3)^{3/2}}\right)$, $r_i = \sqrt{\frac{p_i}{3}}$, $p_i= \frac{b_i^2}{3}-a_i$, $q_i= \frac{2b_i^3}{27}-\frac{b_ia_i}{3}-d_i$, $a_2=pn-1-m^2$, $b_2=p+n$, $d_2=n$, $a_3=qm-1-n^2$, 	$b_3=q+m$ and $d_3=m$.

We observe that about $30\%$ of over 10000 samples can be ruled out by Schur's criterion\footnote{It is nontrivial to characterize the set of all the triples $(m,n,q)$ for which Schur product property (on the dual) does not hold. Using the above character table together with Theorem~\ref{SchurCom} and computer assistance, for $q,n,m \in \mathbb{Z}$, $0 \le q \le 30$, $1 \le n \le 30$ and $0 \le m \le n$, there are exactly $14509$ fusion bialgebras (resp. $542$ fusion rings), and among them, $4757$ (resp. $198$) ones can be ruled out from subfactorization (resp. unitary categorification) by Schur product criterion.}. Note that Ostrik used the inequality in \cite[Theorem 2.21]{Ost15} to rule out some fusion rings. See Figure~\ref{fig:r3} to visualize the efficienty of Schur product criterion and Ostrik's criterion for this family. Note that Ostrik's criterion works for the fusion rings only (not algebras\footnote{Consider the $(\delta_1,\delta_2)$-Bisch-Jones subfactor, its 2-box space provides a fusion algebra in this family with $$(m,n,q)=(0,(\delta_2^2-1)^{\frac{1}{2}},\delta_2(\delta_1^2-2)(\delta_1^2-1)^{-\frac{1}{2}}),$$ which is often in the colored area of the figure for Ostrik's criterion, for example if $(\delta_1,\delta_2)=(\sqrt{2},10/3)$ then $(m,n,q) = (0,\sqrt{91}/3,5)$. Let us also mention here that for a fusion ring, \emph{subfactorizable} is strictly weaker than \emph{unitarily categorifiable}, because if $(\delta_1^2,\delta_2^2)=(6+2\sqrt{6},2)$ then $(m,n,q)=(0,1,4)$, which is ruled out from pivotal categorification by Ostrik's paper \cite{Ost15}.}) and is no more efficient for higher ranks, whereas Schur product criterion does (see Subsection~\ref{sub:simple}).

\begin{figure}
  \includegraphics[scale=0.4]{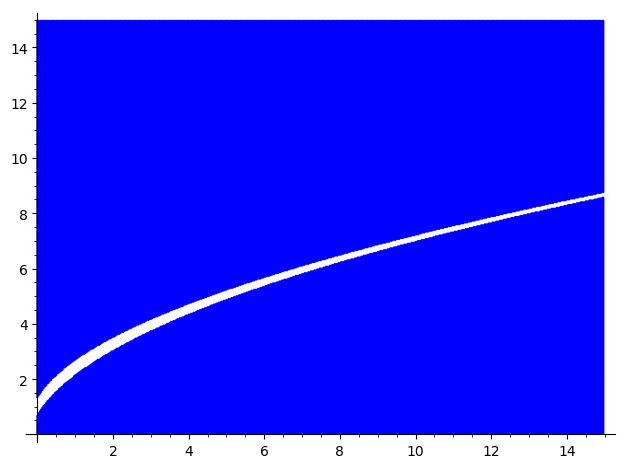}   \hspace{1cm}  \includegraphics[scale=0.4]{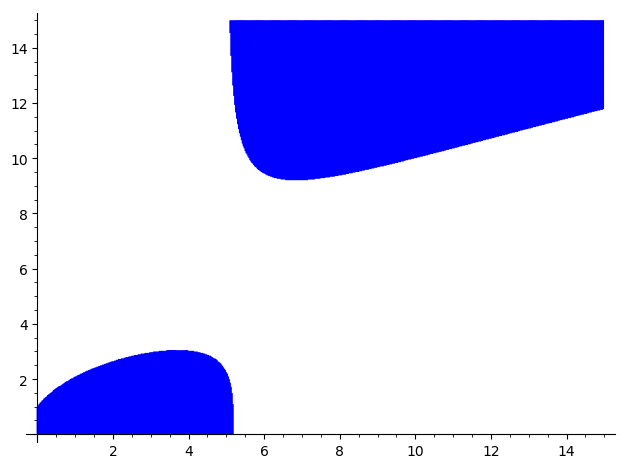}
  \caption{Rank 3: for $q=5$, the set of $(m,n)$ such that Schur product property on the dual (resp. Ostrik's inequality) does not hold is (numerically) given by the right (resp. left) figure (where, for clarity, neither $m \le n$ nor $m^2+n^2-1-mq \ge 0$ is assumed). About the right one, there are two areas, one (at the bottom) is finite, the other infinite; moreover, the projection of these two areas on the $m$-axis overlap around $m=q$. Each area corresponds to the application of Theorem~\ref{SchurCom} on one column. The form appears for all the samples of $q$ we tried, so it is not hard to believe that it is the generic form, and in particular that Schur product property (on the dual) does not hold if $q+1 \le m \le n$ and $n \ge 2q+2$, with $m,n,q \in \mathbb{R}_{\ge 0}$ (so that the corresponding fusion bialgebras admit no subfactorization); it should be provable using the given character table (we did not make the computation).}
  \label{fig:r3}
\end{figure}

Then, let us mention two families (denoted $K_3$ and $K_4$) of fusion algebras of rank $4$ with self-adjoint objects provided by David Penneys and his collaborators \cite{Pen19}. Visualize the obstructions on Figures~\ref{fig:K3} and~\ref{fig:K4}.

\newsavebox{\smlmat}
\savebox{\smlmat}{$\left(\begin{matrix} 1&0&0&0 \\ 0&1&0&0 \\ 0&0&1&0 \\ 0&0&0&1\end{matrix} \right), \left(\begin{matrix}0&1&0&0 \\ 1&d-b&b&0 \\ 0&b&d&1 \\ 0&0&1&0\end{matrix} \right), \left(\begin{matrix}0&0&1&0 \\ 0&b&d&1 \\ 1&d&d+b&1 \\ 0&1&1&0\end{matrix} \right), \left(\begin{matrix}0&0&0&1 \\ 0&0&1&0 \\ 0&1&1&0 \\ 1&0&0&1 \end{matrix} \right)$}

\begin{figure}
  \includegraphics[scale=0.4]{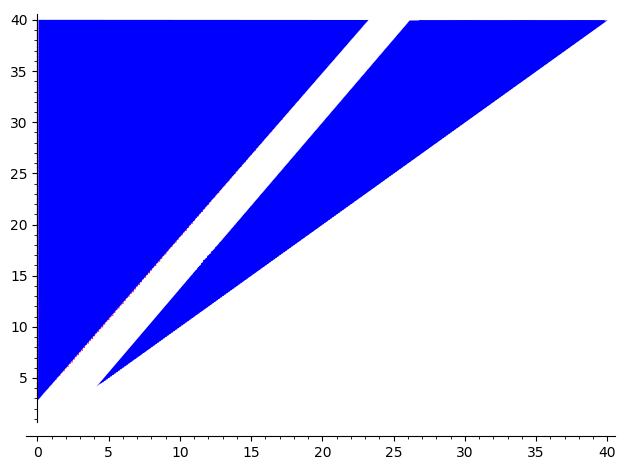} \hspace{1cm} \includegraphics[scale=0.4]{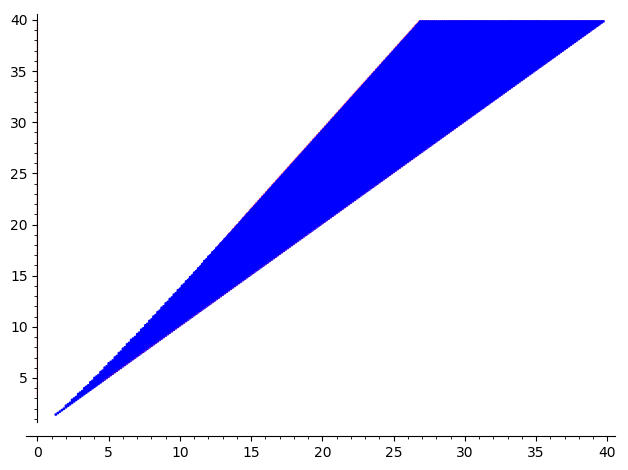}
  \caption{Rank 4, family of fusion algebras $K_3(b,d)$: \newline \usebox{\smlmat}  \newline Schur product property on the dual (resp. Ostrik's inequality) does not hold for $(b,d)$ in the right (resp. left) figure. The area $b>d$ is not considered because $d-b \ge 0$.}
  \label{fig:K3}
\end{figure}

\newsavebox{\smlmatt}
\savebox{\smlmatt}{$ \left(\begin{matrix} 1&0&0&0 \\ 0&1&0&0 \\ 0&0&1&0 \\ 0&0&0&1 \end{matrix} \right), \left(\begin{matrix}0&1&0&0 \\ 1&a&b&1 \\ 0&b&d&0 \\ 0&1&0&0 \end{matrix} \right), \left(\begin{matrix}0&0&1&0 \\ 0&b&d&0 \\ 1&d&g&1 \\ 0&0&1&0\end{matrix} \right), \left(\begin{matrix}0&0&0&1 \\ 0&1&0&0 \\ 0&0&1&0 \\ 1&0&0&0 \end{matrix} \right) $}

\begin{figure}
  \includegraphics[scale=0.4]{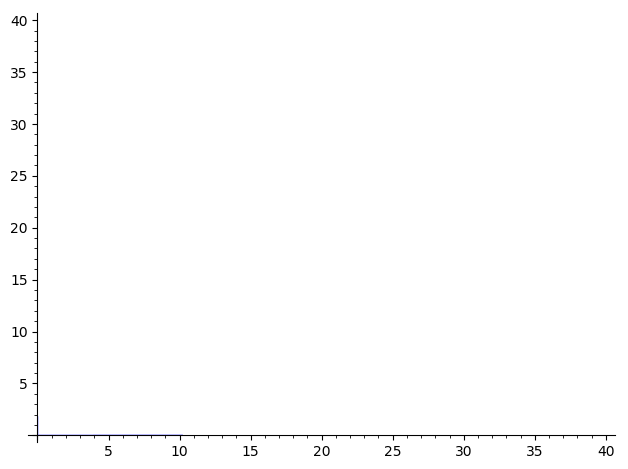} \hspace{1cm}  \includegraphics[scale=0.4]{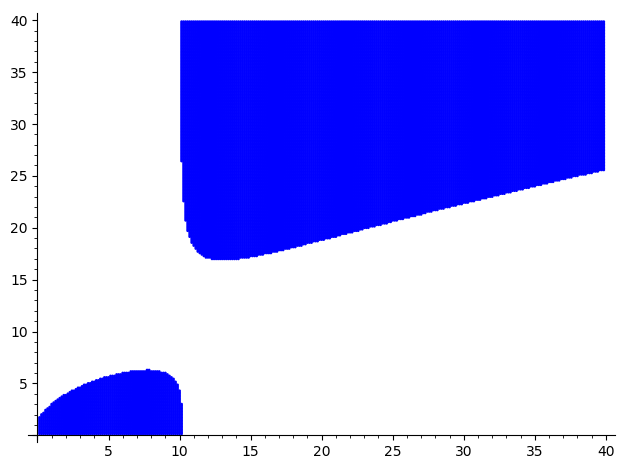}
  \caption{Rank 4, family $K_4(a,b,d,g)$ with $bd>0$ and $a=(b^2+d^2-2-bg)/d$: \newline \usebox{\smlmatt}  \newline  Same convention as above (to simplify, $a$ is not assumed non-negative) with $g=10$. It should be the generic shape for fixed $g$. Ostrik's inequality always holds, so the left figure is empty. About the figure for Schur product criterion on the right, the structure is similar to Figure~\ref{fig:r3}, one finite area on the bottom, one infinite area, and the projection of both on the $b$-axis should overlap around $b=g$.}
  \label{fig:K4}
\end{figure}

Finally, let us consider the family of fusion rings of rank $5$ with exactly three self-adjoint simple objects. By Frobenius reciprocity, the fusion rules must be as follows (with $16$ parameters):

$$  \left( \begin{matrix}1&0&0&0&0\\0&1&0&0&0\\0&0&1&0&0\\0&0&0&1&0\\0&0&0&0&1\end{matrix} \right),
 \left( \begin{matrix}0&1&0&0&0\\0&a&k&f&j\\1&a&a&b&c\\0&d&f&g&h\\0&e&j&i&l\end{matrix} \right)   ,
 \left( \begin{matrix}0&0&1&0&0\\1&a&a&d&e\\0&k&a&f&j\\0&f&b&g&i\\0&j&c&h&l\end{matrix} \right)   ,
 \left( \begin{matrix}0&0&0&1&0\\0&b&f&g&i\\0&f&d&g&h\\1&g&g&m&o\\0&i&h&o&p\end{matrix} \right)   ,
 \left( \begin{matrix}0&0&0&0&1\\0&c&j&h&l\\0&j&e&i&l\\0&h&i&o&p\\1&l&l&p&n\end{matrix} \right)  $$

\noindent such that $n_{i,j}^k \in \mathbb{Z}_{\ge 0}$, and $\sum_s n_{i,j}^sn_{s,k}^t = \sum_s n_{j,k}^sn_{i,s}^t$ (associativity). We found (up to equivalence) exactly $47$ ones at multiplicity $\le 4$ (by brute-force computation), $4$ of which are simple. The Schur product property on the dual (resp. Ostrik's inequality) does not hold on exactly $6$ (resp. $1$) among the $47$ ones, and on exactly $2$ (resp. $1$) among the $4$ simple ones. Schur product criterion may be more efficient at higher multiplicity. Here are the two simple ones on which the Schur product property on the dual (and Ostrik's inequality) holds (note that they are also of Frobenius type).

$$\left( \begin{matrix}1&0&0&0&0\\0&1&0&0&0\\0&0&1&0&0\\0&0&0&1&0\\0&0&0&0&1\end{matrix} \right),
\left( \begin{matrix}0&1&0&0&0\\0&0&1&1&0\\1&0&0&0&1\\0&0&1&0&1\\0&1&0&1&1\end{matrix} \right),
\left( \begin{matrix}0&0&1&0&0\\1&0&0&0&1\\0&1&0&1&0\\0&1&0&0&1\\0&0&1&1&1\end{matrix} \right),
\left( \begin{matrix}0&0&0&1&0\\0&0&1&0&1\\0&1&0&0&1\\1&0&0&1&1\\0&1&1&1&1\end{matrix} \right),
\left( \begin{matrix}0&0&0&0&1\\0&1&0&1&1\\0&0&1&1&1\\0&1&1&1&1\\1&1&1&1&2\end{matrix} \right)$$
Note that $(d(b_1),d(b_2),d(b_3),d(b_4),d(b_5))=(1,\alpha,\alpha,\beta,\gamma)$ with $\alpha=1+2\cos(2\pi/7)\simeq 2.2469$, $\beta=1-2\cos(6\pi/7)\simeq 2.8019$, $\gamma=\alpha+\beta-1 \simeq 4.0489$, so that $\FPdim \simeq 36.650$.

$$
\left( \begin{matrix}1&0&0&0&0\\0&1&0&0&0\\0&0&1&0&0\\0&0&0&1&0\\0&0&0&0&1\end{matrix} \right) ,
\left( \begin{matrix}0&1&0&0&0\\0&0&2&0&1\\1&0&0&2&0\\0&2&0&1&2\\0&0&1&2&2\end{matrix} \right) ,
\left( \begin{matrix}0&0&1&0&0\\1&0&0&2&0\\0&2&0&0&1\\0&0&2&1&2\\0&1&0&2&2\end{matrix} \right) ,
\left( \begin{matrix}0&0&0&1&0\\0&2&0&1&2\\0&0&2&1&2\\1&1&1&4&3\\0&2&2&3&4\end{matrix} \right) ,
\left( \begin{matrix}0&0&0&0&1\\0&0&1&2&2\\0&1&0&2&2\\0&2&2&3&4\\1&2&2&4&4\end{matrix} \right)  $$
Note that $(d(b_1),d(b_2),d(b_3),d(b_4),d(b_5))=(1,\alpha,\alpha,\beta,\gamma)$ with $\alpha= 3+\sqrt{6} \simeq 5.4494$, $\beta= 4+2\sqrt{6}\simeq 8.8989$, $\gamma= 5+2\sqrt{6}\simeq 9.8989$, so that $\FPdim = 120+48\sqrt{6} \simeq 237.57$.

\subsection{Simple Integral Fusion Rings (High Rank)} \label{sub:simple}
A fusion category (resp. ring) is called \emph{simple} if it has no nontrivial proper fusion subcategories (resp. subrings). Here is a result of Etingof, Nikshych and Ostrik \cite[Proposition 9.11]{ENO11}:
\begin{proposition} \label{weak}
A weakly group-theoretical simple fusion category has the Grothendieck ring of $\mathrm{Rep}(G)$, with $G$ a finite simple group.
\end{proposition}
A fusion category (resp. ring) is called \emph{integral} if the Frobenius-Perron dimension of every (simple) object is an integer. Here is the strong version of \cite[Question 2]{ENO11}:
\begin{question} \label{Qweak}
Is there an integral fusion category which is not weakly group-theoretical?
\end{question}

Then it seems relevant to look for integral simple fusion rings which are not Grothendieck rings of any $\mathrm{Rep}(G)$ with $G$ finite (simple) group, because according to Proposition~\ref{weak}, the categorification of one of them would provide a positive answer to Question~\ref{Qweak}.

\begin{definition}
Let $\mathfrak{A}$ be a fusion ring of basis $\{x_1 = 1, \dots, x_r \}$ with $d(x_1) \le d(x_2) \le \cdots  \le d(x_r)$. Let us call $r$ its \emph{rank}, $\sum_i d(x_i)^2$ its \emph{Frobenius-Perron dimension} (or $\FPdim(\mathfrak{A})$), and $[d(x_1), d(x_2), \dots , d(x_r)]$ its \emph{type}, which will also be written by $[[n_1,m_1],[n_2,m_2], \dots, [n_s,m_s]]$ where $m_i$ is the number of $x_j$ with $d(x_j) = n_i$, $\sum_i m_i = r$ and $1 = n_1 < n_2 < \cdots < n_s$.
\end{definition}

Recall that the Grothendieck ring of $\mathrm{Rep}(G)$ is simple if and only if $G$ is simple because every fusion subring is given by $\Rep(G/N)$ with $N$ a normal subgroup. Now a non-abelian simple group is perfect (i.e. $[G,G]$ generates $G$), and there is also a way to characterize the perfect groups at the fusion ring level:

\begin{proposition} A finite group $G$ is perfect if and only if the type of the Grothendieck ring of $\mathrm{Rep}(G)$ satisfies $m_1=1$ (i.e. every one-dimensional representation must be trivial). \end{proposition}
\begin{proof}
Let $G$ be a perfect group and let $\pi$ be a one-dimensional representation of $G$. By assumption, every $g \in G$ is a product of commutators, but $\pi(G)$ is abelian (because $\pi$ is one-dimensional), so that $\pi(g) = \pi(1)$. It follows that $\pi$ is trivial.

Now assume that every one-dimensional representation is trivial, and consider the quotient map $p: G \to Z$ with $Z=G/\langle [G,G]\rangle$ which is abelian. Then $p$ induces a representation $\pi$ of $G$ with $\pi(G)$ abelian, so that $\pi$ is a direct sum of one-dimensional representations. It follows by assumption that $Z=\pi(G)$ is trivial, which means that $G$ is perfect.
\end{proof}

This proposition leads us to call \emph{perfect} a fusion ring with $m_1=1$. Note that a non-perfect simple fusion ring is given by a prime order cyclic group. The fusion ring $\mathfrak{A}$ is called of \emph{Frobenius type} if $\frac{\FPdim(\mathfrak{A})}{d(x_i)}$ is an algebraic integer for all $i$, and if $\mathfrak{A}$ is integral, this means that $d(x_i)$ divides $\FPdim(\mathfrak{A})$. Kaplansky's 6th conjecture \cite{Kap75} states that for every finite dimensional semisimple Hopf algebra $H$ over $\mathbb{C}$, the integral fusion category $\Rep(H)$ is of Frobenius type. If in addition $H$ has a $*$-structure (i.e. is a Kac algebra), then $\Rep(H)$ is unitary. For a \textit{first step} in the proof of this conjecture, see \cite[Theorem 2]{Kac72}. Note that there exist simple integral fusion rings which are not of Frobenius type (see Subsection \ref{sub:notFrob}).
 The integral simple (and perfect) fusion rings of Frobenius type are classified in the following cases (with $\FPdim \neq p^aq^b, pqr$, by \cite{ENO11}), with computer assistance, significantly boosted by Proposition~\ref{incoeff}\footnote{In preparing this paper, we first had a classification up to some smaller dimensions, given by a previous version of the computer program. Then the use of inequalities in Proposition~\ref{incoeff} boosted the computation, allowing us to extend the bounds significantly.}.

$$\begin{array}{c|c|c|c|c|c|c|c}
\text{rank} & \le 5 & 6     & 7   & 8   & 9 & 10 &  \text{all}  \\ \hline
\FPdim <  & 1000000 & 150000 & 15000 & 4080 & 504 & 240 & 132
\end{array}$$

\noindent We found exactly $34$ ones, and each of them is commutative; which leads to:
\begin{question}
Is there a non-commutative simple integral fusion ring (of Frobenius type)?
\end{question}
Let us first summarize these results of the computer search: four of them are Grothendieck rings of $\mathrm{Rep}(G)$ with $G$ a non-abelian finite simple group, $28$ (among the remaining $30$) are ruled out (from unitary categorification) by Schur product property (on the dual), and none can be ruled out by already known obstructions (as Ostrik's inequality); the existence of a unitary categorification is unknown for each of the remaining two. Here are the results in details, where $\#$ counts the number of fusion rings, whereas $\#$\text{Schur} counts those checking the \textit{commutative} Schur product criterion (all the fusion matrices are available in Appendix):

$$\begin{array}{c|c|c|c|c|c}
\# & \text{rank} & \text{FPdim}     & \text{type} & \# \text{Schur}  & \text{group} \\ \hline
 1 & 5 & 60 & [[1,1],[3,2],[4,1],[5,1]] & 1 & \PSL(2,5)  \\ \hline
 1 & 6 & 168 & [[1,1],[3,2],[6,1],[7,1],[8,1]] & 1 & \PSL(2,7) \\ \hline
 2 & 7 & 210 & [[1,1],[5,3],[6,1],[7,2]] & 1 & \\ \hline
 2 & 7 & 360 & [[1,1],[5,2],[8,2],[9,1],[10,1]] & 1 & \PSL(2,9)  \\ \hline
 4 & 7 & 7980 & [[1, 1], [19, 1], [20, 1], [21, 1], [42, 2], [57, 1]] & 0 & \\ \hline
 15& 8 & 660 & [[1,1],[5,2],[10,2],[11,1],[12,2]] & 2 & \PSL(2,11) \\ \hline
 5 & 8 & 990 & [[1,1],[9,1],[10,1],[11,4],[18,1]] & 0 & \\ \hline
 2 & 8 & 1260 & [[1, 1], [6, 1], [7, 2], [10, 1], [15, 1], [20, 2]] & 0 & \\ \hline
 2 & 8 & 1320 & [[1, 1], [6, 2], [10, 1], [11, 1], [15, 2], [24, 1]] & 0 &
\end{array}$$

\begin{question} Are there only finitely many simple integral fusion rings of a given rank (assuming Frobenius type and perfect)? Is the above list the full classification at rank $\le 6$? If the Schur product property (on the dual) is assumed to hold, is it full at rank $ \le 8$?
\end{question}

Let us write here the fusion matrices and character tables for the first fusion ring ruled out written above, and for the two which were not.

First the simple integral fusion ring of rank $7$, $\FPdim$  $210$, type $[[1,1],[5,3],[6,1],[7,2]]$ and fusion matrices:
$$ {\begin{smallmatrix}
1 & 0 & 0 & 0& 0& 0& 0 \\
0 & 1 & 0 & 0& 0& 0& 0 \\
0 & 0 & 1 & 0& 0& 0& 0 \\
0 & 0 & 0 & 1& 0& 0& 0 \\
0 & 0 & 0 & 0& 1& 0& 0 \\
0 & 0 & 0 & 0& 0& 1& 0 \\
0 & 0 & 0 & 0& 0& 0& 1
\end{smallmatrix}, \
\begin{smallmatrix}
0 & 1 & 0 & 0& 0& 0& 0 \\
1 & 1 & 0 & 1& 0& 1& 1 \\
0 & 0 & 1 & 0& 1& 1& 1 \\
0 & 1 & 0 & 0& 1& 1& 1 \\
0 & 0 & 1 & 1& 1& 1& 1 \\
0 & 1 & 1 & 1& 1& 1& 1 \\
0 & 1 & 1 & 1& 1& 1& 1
\end{smallmatrix} , \
\begin{smallmatrix}
0 & 0 & 1 & 0& 0& 0& 0 \\
0 & 0 & 1 & 0& 1& 1& 1 \\
1 & 1 & 1 & 0& 0& 1& 1 \\
0 & 0 & 0 & 1& 1& 1& 1 \\
0 & 1 & 0 & 1& 1& 1& 1 \\
0 & 1 & 1 & 1& 1& 1& 1 \\
0 & 1 & 1 & 1& 1& 1& 1
\end{smallmatrix} , \
\begin{smallmatrix}
0 & 0 & 0 & 1& 0& 0& 0 \\
0 & 1 & 0 & 0& 1& 1& 1 \\
0 & 0 & 0 & 1& 1& 1& 1 \\
1 & 0 & 1 & 1& 0& 1& 1 \\
0 & 1 & 1 & 0& 1& 1& 1 \\
0 & 1 & 1 & 1& 1& 1& 1 \\
0 & 1 & 1 & 1& 1& 1& 1
\end{smallmatrix} , \
\begin{smallmatrix}
0 & 0 & 0 & 0& 1& 0& 0 \\
0 & 0 & 1 & 1& 1& 1& 1 \\
0 & 1 & 0 & 1& 1& 1& 1 \\
0 & 1 & 1 & 0& 1& 1& 1 \\
1 & 1 & 1 & 1& 1& 1& 1 \\
0 & 1 & 1 & 1& 1& 2& 1 \\
0 & 1 & 1 & 1& 1& 1& 2
\end{smallmatrix} , \
\begin{smallmatrix}
0 & 0 & 0 & 0& 0& 1& 0 \\
0 & 1 & 1 & 1& 1& 1& 1 \\
0 & 1 & 1 & 1& 1& 1& 1 \\
0 & 1 & 1 & 1& 1& 1& 1 \\
0 & 1 & 1 & 1& 1& 2& 1 \\
1 & 1 & 1 & 1& 2& 0& 3 \\
0 & 1 & 1 & 1& 1& 3& 1
\end{smallmatrix} , \
\begin{smallmatrix}
0 & 0 & 0 & 0& 0& 0& 1 \\
0 & 1 & 1 & 1& 1& 1& 1 \\
0 & 1 & 1 & 1& 1& 1& 1 \\
0 & 1 & 1 & 1& 1& 1& 1 \\
0 & 1 & 1 & 1& 1& 1& 2 \\
0 & 1 & 1 & 1& 1& 3& 1 \\
1 & 1 & 1 & 1& 2& 1& 2
\end{smallmatrix}} $$

Its character table is the following:
$$  \left[ \begin{matrix}
1 & 1 & 1 & 1 & 1 & 1 & 1  \\
5 & -1 & -\zeta_7 -\zeta_7^6 & -\zeta_7^5 - \zeta_7^2 & -\zeta_7^4 - \zeta_7^3 & 0 & 0 \\
5 & -1 & -\zeta_7^5 - \zeta_7^2 & -\zeta_7^4 - \zeta_7^3 & -\zeta_7 -\zeta_7^6 & 0 & 0  \\
5 & -1 & -\zeta_7^4 - \zeta_7^3 & -\zeta_7 -\zeta_7^6 & -\zeta_7^5 - \zeta_7^2 & 0 & 0  \\
6 & 0 & -1 & -1 & -1 & 1 & 1  \\
7 & 1 & 0 & 0 & 0 & 0 & -3  \\
7 & 1 & 0 & 0 & 0 & -1 & 2
\end{matrix} \right] $$
It is possible to see why it was ruled out by Schur product criterion by observing this character table (in particular its last column) together with Corollary~\ref{SchurCom}:  $$ \frac{1^3}{1} + \frac{0^3}{5} + \frac{0^3}{5} + \frac{0^3}{5} + \frac{1^3}{6} + \frac{(-3)^3}{7} + \frac{2^3}{7} =  -\frac{65}{42}<0.$$

\begin{remark}
Here we applied Corollary~\ref{SchurCom} by using three times the same block (i.e. irreducible representation, or column here), but it is not always possible.
For example, the simple fusion ring of type $[[1,1],[5,2],[8,2],[9,1],[10,1]]$ (the one not given by $\PSL(2,9)$) required two blocks to be ruled out.
\end{remark}

Next, the fusion matrices of the simple integral fusion ring of same type as above, for which Schur product property (on the dual) holds:
$$ {\begin{smallmatrix}
1 & 0 & 0 & 0& 0& 0& 0 \\
0 & 1 & 0 & 0& 0& 0& 0 \\
0 & 0 & 1 & 0& 0& 0& 0 \\
0 & 0 & 0 & 1& 0& 0& 0 \\
0 & 0 & 0 & 0& 1& 0& 0 \\
0 & 0 & 0 & 0& 0& 1& 0 \\
0 & 0 & 0 & 0& 0& 0& 1
\end{smallmatrix} , \
\begin{smallmatrix}
0 & 1 & 0 & 0& 0& 0& 0 \\
1 & 1 & 0 & 1& 0& 1& 1 \\
0 & 0 & 1 & 0& 1& 1& 1 \\
0 & 1 & 0 & 0& 1& 1& 1 \\
0 & 0 & 1 & 1& 1& 1& 1 \\
0 & 1 & 1 & 1& 1& 1& 1 \\
0 & 1 & 1 & 1& 1& 1& 1
\end{smallmatrix} , \
\begin{smallmatrix}
0 & 0 & 1 & 0& 0& 0& 0 \\
0 & 0 & 1 & 0& 1& 1& 1 \\
1 & 1 & 1 & 0& 0& 1& 1 \\
0 & 0 & 0 & 1& 1& 1& 1 \\
0 & 1 & 0 & 1& 1& 1& 1 \\
0 & 1 & 1 & 1& 1& 1& 1 \\
0 & 1 & 1 & 1& 1& 1& 1
\end{smallmatrix} , \
\begin{smallmatrix}
0 & 0 & 0 & 1& 0& 0& 0 \\
0 & 1 & 0 & 0& 1& 1& 1 \\
0 & 0 & 0 & 1& 1& 1& 1 \\
1 & 0 & 1 & 1& 0& 1& 1 \\
0 & 1 & 1 & 0& 1& 1& 1 \\
0 & 1 & 1 & 1& 1& 1& 1 \\
0 & 1 & 1 & 1& 1& 1& 1
\end{smallmatrix} , \
\begin{smallmatrix}
0 & 0 & 0 & 0& 1& 0& 0 \\
0 & 0 & 1 & 1& 1& 1& 1 \\
0 & 1 & 0 & 1& 1& 1& 1 \\
0 & 1 & 1 & 0& 1& 1& 1 \\
1 & 1 & 1 & 1& 1& 1& 1 \\
0 & 1 & 1 & 1& 1& 2& 1 \\
0 & 1 & 1 & 1& 1& 1& 2
\end{smallmatrix} , \
\begin{smallmatrix}
0 & 0 & 0 & 0& 0& 1& 0 \\
0 & 1 & 1 & 1& 1& 1& 1 \\
0 & 1 & 1 & 1& 1& 1& 1 \\
0 & 1 & 1 & 1& 1& 1& 1 \\
0 & 1 & 1 & 1& 1& 2& 1 \\
1 & 1 & 1 & 1& 2& 1& 2 \\
0 & 1 & 1 & 1& 1& 2& 2
\end{smallmatrix} , \
\begin{smallmatrix}
0 & 0 & 0 & 0& 0& 0& 1 \\
0 & 1 & 1 & 1& 1& 1& 1 \\
0 & 1 & 1 & 1& 1& 1& 1 \\
0 & 1 & 1 & 1& 1& 1& 1 \\
0 & 1 & 1 & 1& 1& 1& 2 \\
0 & 1 & 1 & 1& 1& 2& 2 \\
1 & 1 & 1 & 1& 2& 2& 1
\end{smallmatrix}} $$
Let us call $\mathfrak{F}_{210}$ the corresponding fusion ring (mentioned after \cite[Problem 4.12]{Pal18}). Its character table is:
$$  \left[ \begin{matrix}
1 & 1 & 1 & 1 & 1 & 1 & 1  \\
5 & -1 & -\zeta_7 -\zeta_7^6 & -\zeta_7^5 - \zeta_7^2 & -\zeta_7^4 - \zeta_7^3 & 0 & 0 \\
5 & -1 & -\zeta_7^5 - \zeta_7^2 & -\zeta_7^4 - \zeta_7^3 & -\zeta_7 -\zeta_7^6 & 0 & 0  \\
5 & -1 & -\zeta_7^4 - \zeta_7^3 & -\zeta_7 -\zeta_7^6 & -\zeta_7^5 - \zeta_7^2 & 0 & 0  \\
6 & 0 & -1 & -1 & -1 & 1 & 1  \\
7 & 1 & 0 & 0 & 0 & \zeta_5+\zeta_5^4 & \zeta_5^2+\zeta_5^3  \\
7 & 1 & 0 & 0 & 0 & \zeta_5^2+\zeta_5^3 & \zeta_5+\zeta_5^4
\end{matrix} \right]$$

Finally, the fusion matrices of the only simple integral fusion ring (not given by a group) of rank $8$, $\FPdim$  $660$, type $[[1,1],[5,2],[10,2],[11,1],[12,2]]$ on which Schur product property (on the dual) holds:
$$\begin{smallmatrix}1 & 0 & 0 & 0 & 0 & 0 & 0 & 0 \\ 0 & 1 & 0 & 0 & 0 & 0 & 0 & 0 \\ 0 & 0 & 1 & 0 & 0 & 0 & 0 & 0 \\ 0 & 0 & 0 & 1 & 0 & 0 & 0 & 0 \\ 0 & 0 & 0 & 0 & 1 & 0 & 0 & 0 \\ 0 & 0 & 0 & 0 & 0 & 1 & 0 & 0 \\ 0 & 0 & 0 & 0 & 0 & 0 & 1 & 0 \\ 0 & 0 & 0 & 0 & 0 & 0 & 0 & 1\end{smallmatrix},  \ \begin{smallmatrix}0 & 1 & 0 & 0 & 0 & 0 & 0 & 0 \\ 0 & 0 & 1 & 1 & 1 & 0 & 0 & 0 \\ 1 & 0 & 0 & 0 & 0 & 0 & 1 & 1 \\ 0 & 0 & 1 & 0 & 1 & 1 & 1 & 1 \\ 0 & 0 & 1 & 1 & 0 & 1 & 1 & 1 \\ 0 & 0 & 0 & 1 & 1 & 1 & 1 & 1 \\ 0 & 1 & 0 & 1 & 1 & 1 & 1 & 1 \\ 0 & 1 & 0 & 1 & 1 & 1 & 1 & 1\end{smallmatrix},  \ \begin{smallmatrix}0 & 0 & 1 & 0 & 0 & 0 & 0 & 0 \\ 1 & 0 & 0 & 0 & 0 & 0 & 1 & 1 \\ 0 & 1 & 0 & 1 & 1 & 0 & 0 & 0 \\ 0 & 1 & 0 & 0 & 1 & 1 & 1 & 1 \\ 0 & 1 & 0 & 1 & 0 & 1 & 1 & 1 \\ 0 & 0 & 0 & 1 & 1 & 1 & 1 & 1 \\ 0 & 0 & 1 & 1 & 1 & 1 & 1 & 1 \\ 0 & 0 & 1 & 1 & 1 & 1 & 1 & 1\end{smallmatrix},  \ \begin{smallmatrix}0 & 0 & 0 & 1 & 0 & 0 & 0 & 0 \\ 0 & 0 & 1 & 0 & 1 & 1 & 1 & 1 \\ 0 & 1 & 0 & 0 & 1 & 1 & 1 & 1 \\ 1 & 0 & 0 & 3 & 1 & 1 & 2 & 2 \\ 0 & 1 & 1 & 1 & 1 & 2 & 2 & 2 \\ 0 & 1 & 1 & 1 & 2 & 2 & 2 & 2 \\ 0 & 1 & 1 & 2 & 2 & 2 & 2 & 2 \\ 0 & 1 & 1 & 2 & 2 & 2 & 2 & 2\end{smallmatrix}, \begin{smallmatrix}0 & 0 & 0 & 0 & 1 & 0 & 0 & 0 \\ 0 & 0 & 1 & 1 & 0 & 1 & 1 & 1 \\ 0 & 1 & 0 & 1 & 0 & 1 & 1 & 1 \\ 0 & 1 & 1 & 1 & 1 & 2 & 2 & 2 \\ 1 & 0 & 0 & 1 & 3 & 1 & 2 & 2 \\ 0 & 1 & 1 & 2 & 1 & 2 & 2 & 2 \\ 0 & 1 & 1 & 2 & 2 & 2 & 2 & 2 \\ 0 & 1 & 1 & 2 & 2 & 2 & 2 & 2\end{smallmatrix},  \ \begin{smallmatrix}0 & 0 & 0 & 0 & 0 & 1 & 0 & 0 \\ 0 & 0 & 0 & 1 & 1 & 1 & 1 & 1 \\ 0 & 0 & 0 & 1 & 1 & 1 & 1 & 1 \\ 0 & 1 & 1 & 1 & 2 & 2 & 2 & 2 \\ 0 & 1 & 1 & 2 & 1 & 2 & 2 & 2 \\ 1 & 1 & 1 & 2 & 2 & 2 & 2 & 2 \\ 0 & 1 & 1 & 2 & 2 & 2 & 3 & 2 \\ 0 & 1 & 1 & 2 & 2 & 2 & 2 & 3\end{smallmatrix},  \ \begin{smallmatrix}0 & 0 & 0 & 0 & 0 & 0 & 1 & 0 \\ 0 & 1 & 0 & 1 & 1 & 1 & 1 & 1 \\ 0 & 0 & 1 & 1 & 1 & 1 & 1 & 1 \\ 0 & 1 & 1 & 2 & 2 & 2 & 2 & 2 \\ 0 & 1 & 1 & 2 & 2 & 2 & 2 & 2 \\ 0 & 1 & 1 & 2 & 2 & 2 & 3 & 2 \\ 1 & 1 & 1 & 2 & 2 & 3 & 2 & 3 \\ 0 & 1 & 1 & 2 & 2 & 2 & 3 & 3\end{smallmatrix},  \ \begin{smallmatrix}0 & 0 & 0 & 0 & 0 & 0 & 0 & 1 \\ 0 & 1 & 0 & 1 & 1 & 1 & 1 & 1 \\ 0 & 0 & 1 & 1 & 1 & 1 & 1 & 1 \\ 0 & 1 & 1 & 2 & 2 & 2 & 2 & 2 \\ 0 & 1 & 1 & 2 & 2 & 2 & 2 & 2 \\ 0 & 1 & 1 & 2 & 2 & 2 & 2 & 3 \\ 0 & 1 & 1 & 2 & 2 & 2 & 3 & 3 \\ 1 & 1 & 1 & 2 & 2 & 3 & 3 & 2\end{smallmatrix} $$
Let us call $\mathfrak{F}_{660}$ the corresponding fusion ring. Its character table is:
$$ \left[ \begin{matrix}
1&1&1&1&1&1&1&1 \\ 5&0&-1&1&-1&0&\zeta_{11}+\zeta_{11}^3+\zeta_{11}^4+\zeta_{11}^5+\zeta_{11}^9&\zeta_{11}^2+\zeta_{11}^6+\zeta_{11}^7+\zeta_{11}^8+\zeta_{11}^{10} \\ 5&0&-1&1&-1&0&\zeta_{11}^2+\zeta_{11}^6+\zeta_{11}^7+\zeta_{11}^8+\zeta_{11}^{10}&\zeta_{11}+\zeta_{11}^3+\zeta_{11}^4+\zeta_{11}^5+\zeta_{11}^9 \\ 10&0&1+\sqrt{3}&0&1-\sqrt{3}&0&-1&-1 \\ 10&0&1-\sqrt{3}&0&1+\sqrt{3}&0&-1&-1 \\ 11&1&-1&-1&-1&1&0&0 \\ 12&\zeta_{5}+\zeta_{5}^4&0&0&0&\zeta_{5}^2+\zeta_{5}^3&1&1 \\ 12&\zeta_{5}^2+\zeta_{5}^3&0&0&0&\zeta_{5}+\zeta_{5}^4&1&1
\end{matrix} \right]$$
\begin{question}
Do $\mathfrak{F}_{210}$ or $\mathfrak{F}_{660}$ admit a unitary categorification?
\end{question}

Subsection \ref{ExtraPerfect} of Appendix mentions $2561$ extra perfect integral fusion rings of rank $\le 10$. Among them, $7$ ones are simple and $9$ ones are noncommutative (none both). In the commutative case, $2072$ ones can be ruled out from unitary categorification by Corollary \ref{SchurCom} (more than $80\%$).

\section{Appendix}
\subsection{List of simple integral fusion rings of Frobenius type}
We provide here the fusion matrices for the $34$ simple integral fusion rings mentioned above (all commutative) together with what we know about them. Those ruled out have no additional data.

\begin{itemize}
\item Rank $5$ and $\FPdim$  $60$, one of type $[[1,1],[3,2],[4,1],[5,1]]$, given by the group $\PSL(2,5)$:
$$\begin{smallmatrix} 1&0&0&0&0 \\ 0&1&0&0&0 \\ 0&0&1&0&0 \\ 0&0&0&1&0 \\ 0&0&0&0&1\end{smallmatrix} ,   \ \begin{smallmatrix}0&1&0&0&0 \\ 1&1&0&0&1 \\ 0&0&0&1&1 \\ 0&0&1&1&1 \\ 0&1&1&1&1\end{smallmatrix} ,   \ \begin{smallmatrix}0&0&1&0&0 \\ 0&0&0&1&1 \\ 1&0&1&0&1 \\ 0&1&0&1&1 \\ 0&1&1&1&1\end{smallmatrix} ,   \ \begin{smallmatrix}0&0&0&1&0 \\ 0&0&1&1&1 \\ 0&1&0&1&1 \\ 1&1&1&1&1 \\ 0&1&1&1&2\end{smallmatrix} ,   \ \begin{smallmatrix}0&0&0&0&1 \\ 0&1&1&1&1 \\ 0&1&1&1&1 \\ 0&1&1&1&2 \\ 1&1&1&2&2\end{smallmatrix} $$

\item Rank $6$ and  $\FPdim$  $168$, one of type $[[1,1],[3,2],[6,1],[7,1],[8,1]]$, given by $\PSL(2,7)$:
$$ \begin{smallmatrix}1&0&0&0&0&0 \\ 0&1&0&0&0&0 \\ 0&0&1&0&0&0 \\ 0&0&0&1&0&0 \\ 0&0&0&0&1&0 \\ 0&0&0&0&0&1\end{smallmatrix} ,   \ \begin{smallmatrix}0&1&0&0&0&0 \\ 0&0&1&1&0&0 \\ 1&0&0&0&0&1 \\ 0&0&1&0&1&1 \\ 0&0&0&1&1&1 \\ 0&1&0&1&1&1\end{smallmatrix} ,   \ \begin{smallmatrix}0&0&1&0&0&0 \\ 1&0&0&0&0&1 \\ 0&1&0&1&0&0 \\ 0&1&0&0&1&1 \\ 0&0&0&1&1&1 \\ 0&0&1&1&1&1\end{smallmatrix} ,   \ \begin{smallmatrix}0&0&0&1&0&0 \\ 0&0&1&0&1&1 \\ 0&1&0&0&1&1 \\ 1&0&0&2&1&2 \\ 0&1&1&1&2&2 \\ 0&1&1&2&2&2\end{smallmatrix} ,   \ \begin{smallmatrix}0&0&0&0&1&0 \\ 0&0&0&1&1&1 \\ 0&0&0&1&1&1 \\ 0&1&1&1&2&2 \\ 1&1&1&2&2&2 \\ 0&1&1&2&2&3\end{smallmatrix} ,   \ \begin{smallmatrix}0&0&0&0&0&1 \\ 0&1&0&1&1&1 \\ 0&0&1&1&1&1 \\ 0&1&1&2&2&2 \\ 0&1&1&2&2&3 \\ 1&1&1&2&3&3\end{smallmatrix} $$

\item Rank $7$ and $\FPdim$ $210$, two of type $[[1,1],[5,3],[6,1],[7,2]]$:
$$ \begin{smallmatrix}1&0&0&0&0&0&0 \\ 0&1&0&0&0&0&0 \\ 0&0&1&0&0&0&0 \\ 0&0&0&1&0&0&0 \\ 0&0&0&0&1&0&0 \\ 0&0&0&0&0&1&0 \\ 0&0&0&0&0&0&1\end{smallmatrix} ,   \ \begin{smallmatrix}0&1&0&0&0&0&0 \\ 1&1&0&1&0&1&1 \\ 0&0&1&0&1&1&1 \\ 0&1&0&0&1&1&1 \\ 0&0&1&1&1&1&1 \\ 0&1&1&1&1&1&1 \\ 0&1&1&1&1&1&1\end{smallmatrix} ,   \ \begin{smallmatrix}0&0&1&0&0&0&0 \\ 0&0&1&0&1&1&1 \\ 1&1&1&0&0&1&1 \\ 0&0&0&1&1&1&1 \\ 0&1&0&1&1&1&1 \\ 0&1&1&1&1&1&1 \\ 0&1&1&1&1&1&1\end{smallmatrix} ,   \ \begin{smallmatrix}0&0&0&1&0&0&0 \\ 0&1&0&0&1&1&1 \\ 0&0&0&1&1&1&1 \\ 1&0&1&1&0&1&1 \\ 0&1&1&0&1&1&1 \\ 0&1&1&1&1&1&1 \\ 0&1&1&1&1&1&1\end{smallmatrix} ,   \ \begin{smallmatrix}0&0&0&0&1&0&0 \\ 0&0&1&1&1&1&1 \\ 0&1&0&1&1&1&1 \\ 0&1&1&0&1&1&1 \\ 1&1&1&1&1&1&1 \\ 0&1&1&1&1&2&1 \\ 0&1&1&1&1&1&2\end{smallmatrix} ,   \ \begin{smallmatrix}0&0&0&0&0&1&0 \\ 0&1&1&1&1&1&1 \\ 0&1&1&1&1&1&1 \\ 0&1&1&1&1&1&1 \\ 0&1&1&1&1&2&1 \\ 1&1&1&1&2&0&3 \\ 0&1&1&1&1&3&1\end{smallmatrix} ,   \ \begin{smallmatrix}0&0&0&0&0&0&1 \\ 0&1&1&1&1&1&1 \\ 0&1&1&1&1&1&1 \\ 0&1&1&1&1&1&1 \\ 0&1&1&1&1&1&2 \\ 0&1&1&1&1&3&1 \\ 1&1&1&1&2&1&2\end{smallmatrix} $$

\noindent  The one satisfying Schur product property (on the dual):
$$ \begin{smallmatrix}1&0&0&0&0&0&0 \\ 0&1&0&0&0&0&0 \\ 0&0&1&0&0&0&0 \\ 0&0&0&1&0&0&0 \\ 0&0&0&0&1&0&0 \\ 0&0&0&0&0&1&0 \\ 0&0&0&0&0&0&1\end{smallmatrix} ,   \ \begin{smallmatrix}0&1&0&0&0&0&0 \\ 1&1&0&1&0&1&1 \\ 0&0&1&0&1&1&1 \\ 0&1&0&0&1&1&1 \\ 0&0&1&1&1&1&1 \\ 0&1&1&1&1&1&1 \\ 0&1&1&1&1&1&1\end{smallmatrix} ,   \ \begin{smallmatrix}0&0&1&0&0&0&0 \\ 0&0&1&0&1&1&1 \\ 1&1&1&0&0&1&1 \\ 0&0&0&1&1&1&1 \\ 0&1&0&1&1&1&1 \\ 0&1&1&1&1&1&1 \\ 0&1&1&1&1&1&1\end{smallmatrix} ,   \ \begin{smallmatrix}0&0&0&1&0&0&0 \\ 0&1&0&0&1&1&1 \\ 0&0&0&1&1&1&1 \\ 1&0&1&1&0&1&1 \\ 0&1&1&0&1&1&1 \\ 0&1&1&1&1&1&1 \\ 0&1&1&1&1&1&1\end{smallmatrix} ,   \ \begin{smallmatrix}0&0&0&0&1&0&0 \\ 0&0&1&1&1&1&1 \\ 0&1&0&1&1&1&1 \\ 0&1&1&0&1&1&1 \\ 1&1&1&1&1&1&1 \\ 0&1&1&1&1&2&1 \\ 0&1&1&1&1&1&2\end{smallmatrix} ,   \ \begin{smallmatrix}0&0&0&0&0&1&0 \\ 0&1&1&1&1&1&1 \\ 0&1&1&1&1&1&1 \\ 0&1&1&1&1&1&1 \\ 0&1&1&1&1&2&1 \\ 1&1&1&1&2&1&2 \\ 0&1&1&1&1&2&2\end{smallmatrix} ,   \ \begin{smallmatrix}0&0&0&0&0&0&1 \\ 0&1&1&1&1&1&1 \\ 0&1&1&1&1&1&1 \\ 0&1&1&1&1&1&1 \\ 0&1&1&1&1&1&2 \\ 0&1&1&1&1&2&2 \\ 1&1&1&1&2&2&1\end{smallmatrix} $$

\item Rank $7$ and $\FPdim$ $360$, two of type $[[1,1],[5,2],[8,2],[9,1],[10,1]]$:
$$ \begin{smallmatrix}1&0&0&0&0&0&0 \\ 0&1&0&0&0&0&0 \\ 0&0&1&0&0&0&0 \\ 0&0&0&1&0&0&0 \\ 0&0&0&0&1&0&0 \\ 0&0&0&0&0&1&0 \\ 0&0&0&0&0&0&1\end{smallmatrix} ,   \ \begin{smallmatrix}0&1&0&0&0&0&0 \\ 1&1&0&0&0&1&1 \\ 0&0&0&1&1&1&0 \\ 0&0&1&1&1&1&1 \\ 0&0&1&1&1&1&1 \\ 0&1&1&1&1&1&1 \\ 0&1&0&1&1&1&2\end{smallmatrix} ,   \ \begin{smallmatrix}0&0&1&0&0&0&0 \\ 0&0&0&1&1&1&0 \\ 1&0&1&0&0&1&1 \\ 0&1&0&1&1&1&1 \\ 0&1&0&1&1&1&1 \\ 0&1&1&1&1&1&1 \\ 0&0&1&1&1&1&2\end{smallmatrix} ,   \ \begin{smallmatrix}0&0&0&1&0&0&0 \\ 0&0&1&1&1&1&1 \\ 0&1&0&1&1&1&1 \\ 1&1&1&1&2&1&2 \\ 0&1&1&2&0&2&2 \\ 0&1&1&1&2&2&2 \\ 0&1&1&2&2&2&2\end{smallmatrix} ,   \ \begin{smallmatrix}0&0&0&0&1&0&0 \\ 0&0&1&1&1&1&1 \\ 0&1&0&1&1&1&1 \\ 0&1&1&2&0&2&2 \\ 1&1&1&0&3&1&2 \\ 0&1&1&2&1&2&2 \\ 0&1&1&2&2&2&2\end{smallmatrix} ,   \ \begin{smallmatrix}0&0&0&0&0&1&0 \\ 0&1&1&1&1&1&1 \\ 0&1&1&1&1&1&1 \\ 0&1&1&1&2&2&2 \\ 0&1&1&2&1&2&2 \\ 1&1&1&2&2&2&2 \\ 0&1&1&2&2&2&3\end{smallmatrix} ,   \ \begin{smallmatrix}0&0&0&0&0&0&1 \\ 0&1&0&1&1&1&2 \\ 0&0&1&1&1&1&2 \\ 0&1&1&2&2&2&2 \\ 0&1&1&2&2&2&2 \\ 0&1&1&2&2&2&3 \\ 1&2&2&2&2&3&2\end{smallmatrix} $$

\noindent The one given by the finite simple group $\PSL(2,9)$:
$$ \begin{smallmatrix}1&0&0&0&0&0&0 \\ 0&1&0&0&0&0&0 \\ 0&0&1&0&0&0&0 \\ 0&0&0&1&0&0&0 \\ 0&0&0&0&1&0&0 \\ 0&0&0&0&0&1&0 \\ 0&0&0&0&0&0&1\end{smallmatrix} ,   \ \begin{smallmatrix}0&1&0&0&0&0&0 \\ 1&1&0&0&0&1&1 \\ 0&0&0&1&1&1&0 \\ 0&0&1&1&1&1&1 \\ 0&0&1&1&1&1&1 \\ 0&1&1&1&1&1&1 \\ 0&1&0&1&1&1&2\end{smallmatrix} ,   \ \begin{smallmatrix}0&0&1&0&0&0&0 \\ 0&0&0&1&1&1&0 \\ 1&0&1&0&0&1&1 \\ 0&1&0&1&1&1&1 \\ 0&1&0&1&1&1&1 \\ 0&1&1&1&1&1&1 \\ 0&0&1&1&1&1&2\end{smallmatrix} ,   \ \begin{smallmatrix}0&0&0&1&0&0&0 \\ 0&0&1&1&1&1&1 \\ 0&1&0&1&1&1&1 \\ 1&1&1&2&1&1&2 \\ 0&1&1&1&1&2&2 \\ 0&1&1&1&2&2&2 \\ 0&1&1&2&2&2&2\end{smallmatrix} ,   \ \begin{smallmatrix}0&0&0&0&1&0&0 \\ 0&0&1&1&1&1&1 \\ 0&1&0&1&1&1&1 \\ 0&1&1&1&1&2&2 \\ 1&1&1&1&2&1&2 \\ 0&1&1&2&1&2&2 \\ 0&1&1&2&2&2&2\end{smallmatrix} ,   \ \begin{smallmatrix}0&0&0&0&0&1&0 \\ 0&1&1&1&1&1&1 \\ 0&1&1&1&1&1&1 \\ 0&1&1&1&2&2&2 \\ 0&1&1&2&1&2&2 \\ 1&1&1&2&2&2&2 \\ 0&1&1&2&2&2&3\end{smallmatrix} ,   \ \begin{smallmatrix}0&0&0&0&0&0&1 \\ 0&1&0&1&1&1&2 \\ 0&0&1&1&1&1&2 \\ 0&1&1&2&2&2&2 \\ 0&1&1&2&2&2&2 \\ 0&1&1&2&2&2&3 \\ 1&2&2&2&2&3&2\end{smallmatrix} $$

\item Rank $7$ and $\FPdim$ $7980$, four of type $[[1, 1], [19, 1], [20, 1], [21, 1], [42, 2], [57, 1]]$:
$$ \begin{smallmatrix}1&0&0&0&0&0&0 \\ 0&1&0&0&0&0&0 \\ 0&0&1&0&0&0&0 \\ 0&0&0&1&0&0&0 \\ 0&0&0&0&1&0&0 \\ 0&0&0&0&0&1&0 \\ 0&0&0&0&0&0&1\end{smallmatrix} ,   \ \begin{smallmatrix}0&1&0&0&0&0&0 \\ 1&0&0&1&2&2&3 \\ 0&0&1&1&2&2&3 \\ 0&1&1&1&2&2&3 \\ 0&2&2&2&4&4&6 \\ 0&2&2&2&4&4&6 \\ 0&3&3&3&6&6&7\end{smallmatrix} ,   \ \begin{smallmatrix}0&0&1&0&0&0&0 \\ 0&0&1&1&2&2&3 \\ 1&1&1&1&2&2&3 \\ 0&1&1&2&2&2&3 \\ 0&2&2&2&5&4&6 \\ 0&2&2&2&4&5&6 \\ 0&3&3&3&6&6&8\end{smallmatrix} ,   \ \begin{smallmatrix}0&0&0&1&0&0&0 \\ 0&1&1&1&2&2&3 \\ 0&1&1&2&2&2&3 \\ 1&1&2&2&0&4&3 \\ 0&2&2&0&9&2&6 \\ 0&2&2&4&2&7&6 \\ 0&3&3&3&6&6&9\end{smallmatrix} ,   \ \begin{smallmatrix}0&0&0&0&1&0&0 \\ 0&2&2&2&4&4&6 \\ 0&2&2&2&5&4&6 \\ 0&2&2&0&9&2&6 \\ 1&4&5&9&2&15&12 \\ 0&4&4&2&15&6&12 \\ 0&6&6&6&12&12&18\end{smallmatrix} ,   \ \begin{smallmatrix}0&0&0&0&0&1&0 \\ 0&2&2&2&4&4&6 \\ 0&2&2&2&4&5&6 \\ 0&2&2&4&2&7&6 \\ 0&4&4&2&15&6&12 \\ 1&4&5&7&6&12&12 \\ 0&6&6&6&12&12&18\end{smallmatrix} ,   \ \begin{smallmatrix}0&0&0&0&0&0&1 \\ 0&3&3&3&6&6&7 \\ 0&3&3&3&6&6&8 \\ 0&3&3&3&6&6&9 \\ 0&6&6&6&12&12&18 \\ 0&6&6&6&12&12&18 \\ 1&7&8&9&18&18&22\end{smallmatrix} $$

$$ \begin{smallmatrix}1&0&0&0&0&0&0 \\ 0&1&0&0&0&0&0 \\ 0&0&1&0&0&0&0 \\ 0&0&0&1&0&0&0 \\ 0&0&0&0&1&0&0 \\ 0&0&0&0&0&1&0 \\ 0&0&0&0&0&0&1\end{smallmatrix} ,   \ \begin{smallmatrix}0&1&0&0&0&0&0 \\ 1&0&0&1&2&2&3 \\ 0&0&1&1&2&2&3 \\ 0&1&1&1&2&2&3 \\ 0&2&2&2&4&4&6 \\ 0&2&2&2&4&4&6 \\ 0&3&3&3&6&6&7\end{smallmatrix} ,   \ \begin{smallmatrix}0&0&1&0&0&0&0 \\ 0&0&1&1&2&2&3 \\ 1&1&1&1&2&2&3 \\ 0&1&1&2&2&2&3 \\ 0&2&2&2&5&4&6 \\ 0&2&2&2&4&5&6 \\ 0&3&3&3&6&6&8\end{smallmatrix} ,   \ \begin{smallmatrix}0&0&0&1&0&0&0 \\ 0&1&1&1&2&2&3 \\ 0&1&1&2&2&2&3 \\ 1&1&2&2&0&4&3 \\ 0&2&2&0&7&4&6 \\ 0&2&2&4&4&5&6 \\ 0&3&3&3&6&6&9\end{smallmatrix} ,   \ \begin{smallmatrix}0&0&0&0&1&0&0 \\ 0&2&2&2&4&4&6 \\ 0&2&2&2&5&4&6 \\ 0&2&2&0&7&4&6 \\ 1&4&5&7&7&11&12 \\ 0&4&4&4&11&9&12 \\ 0&6&6&6&12&12&18\end{smallmatrix} ,   \ \begin{smallmatrix}0&0&0&0&0&1&0 \\ 0&2&2&2&4&4&6 \\ 0&2&2&2&4&5&6 \\ 0&2&2&4&4&5&6 \\ 0&4&4&4&11&9&12 \\ 1&4&5&5&9&10&12 \\ 0&6&6&6&12&12&18\end{smallmatrix} ,   \ \begin{smallmatrix}0&0&0&0&0&0&1 \\ 0&3&3&3&6&6&7 \\ 0&3&3&3&6&6&8 \\ 0&3&3&3&6&6&9 \\ 0&6&6&6&12&12&18 \\ 0&6&6&6&12&12&18 \\ 1&7&8&9&18&18&22\end{smallmatrix} $$

$$ \begin{smallmatrix}1&0&0&0&0&0&0 \\ 0&1&0&0&0&0&0 \\ 0&0&1&0&0&0&0 \\ 0&0&0&1&0&0&0 \\ 0&0&0&0&1&0&0 \\ 0&0&0&0&0&1&0 \\ 0&0&0&0&0&0&1\end{smallmatrix} ,   \ \begin{smallmatrix}0&1&0&0&0&0&0 \\ 1&0&0&1&2&2&3 \\ 0&0&1&1&2&2&3 \\ 0&1&1&1&2&2&3 \\ 0&2&2&2&4&4&6 \\ 0&2&2&2&4&4&6 \\ 0&3&3&3&6&6&7\end{smallmatrix} ,   \ \begin{smallmatrix}0&0&1&0&0&0&0 \\ 0&0&1&1&2&2&3 \\ 1&1&1&1&2&2&3 \\ 0&1&1&2&2&2&3 \\ 0&2&2&2&5&4&6 \\ 0&2&2&2&4&5&6 \\ 0&3&3&3&6&6&8\end{smallmatrix} ,   \ \begin{smallmatrix}0&0&0&1&0&0&0 \\ 0&1&1&1&2&2&3 \\ 0&1&1&2&2&2&3 \\ 1&1&2&2&0&4&3 \\ 0&2&2&0&5&6&6 \\ 0&2&2&4&6&3&6 \\ 0&3&3&3&6&6&9\end{smallmatrix} ,   \ \begin{smallmatrix}0&0&0&0&1&0&0 \\ 0&2&2&2&4&4&6 \\ 0&2&2&2&5&4&6 \\ 0&2&2&0&5&6&6 \\ 1&4&5&5&8&11&12 \\ 0&4&4&6&11&8&12 \\ 0&6&6&6&12&12&18\end{smallmatrix} ,   \ \begin{smallmatrix}0&0&0&0&0&1&0 \\ 0&2&2&2&4&4&6 \\ 0&2&2&2&4&5&6 \\ 0&2&2&4&6&3&6 \\ 0&4&4&6&11&8&12 \\ 1&4&5&3&8&12&12 \\ 0&6&6&6&12&12&18\end{smallmatrix} ,   \ \begin{smallmatrix}0&0&0&0&0&0&1 \\ 0&3&3&3&6&6&7 \\ 0&3&3&3&6&6&8 \\ 0&3&3&3&6&6&9 \\ 0&6&6&6&12&12&18 \\ 0&6&6&6&12&12&18 \\ 1&7&8&9&18&18&22\end{smallmatrix} $$

$$ \begin{smallmatrix}1&0&0&0&0&0&0 \\ 0&1&0&0&0&0&0 \\ 0&0&1&0&0&0&0 \\ 0&0&0&1&0&0&0 \\ 0&0&0&0&1&0&0 \\ 0&0&0&0&0&1&0 \\ 0&0&0&0&0&0&1\end{smallmatrix} ,   \ \begin{smallmatrix}0&1&0&0&0&0&0 \\ 1&0&0&1&2&2&3 \\ 0&0&1&1&2&2&3 \\ 0&1&1&1&2&2&3 \\ 0&2&2&2&4&4&6 \\ 0&2&2&2&4&4&6 \\ 0&3&3&3&6&6&7\end{smallmatrix} ,   \ \begin{smallmatrix}0&0&1&0&0&0&0 \\ 0&0&1&1&2&2&3 \\ 1&1&1&1&2&2&3 \\ 0&1&1&2&2&2&3 \\ 0&2&2&2&5&4&6 \\ 0&2&2&2&4&5&6 \\ 0&3&3&3&6&6&8\end{smallmatrix} ,   \ \begin{smallmatrix}0&0&0&1&0&0&0 \\ 0&1&1&1&2&2&3 \\ 0&1&1&2&2&2&3 \\ 1&1&2&2&0&4&3 \\ 0&2&2&0&3&8&6 \\ 0&2&2&4&8&1&6 \\ 0&3&3&3&6&6&9\end{smallmatrix} ,   \ \begin{smallmatrix}0&0&0&0&1&0&0 \\ 0&2&2&2&4&4&6 \\ 0&2&2&2&5&4&6 \\ 0&2&2&0&3&8&6 \\ 1&4&5&3&5&15&12 \\ 0&4&4&8&15&3&12 \\ 0&6&6&6&12&12&18\end{smallmatrix} ,   \ \begin{smallmatrix}0&0&0&0&0&1&0 \\ 0&2&2&2&4&4&6 \\ 0&2&2&2&4&5&6 \\ 0&2&2&4&8&1&6 \\ 0&4&4&8&15&3&12 \\ 1&4&5&1&3&18&12 \\ 0&6&6&6&12&12&18\end{smallmatrix} ,   \ \begin{smallmatrix}0&0&0&0&0&0&1 \\ 0&3&3&3&6&6&7 \\ 0&3&3&3&6&6&8 \\ 0&3&3&3&6&6&9 \\ 0&6&6&6&12&12&18 \\ 0&6&6&6&12&12&18 \\ 1&7&8&9&18&18&22\end{smallmatrix} $$

\item Rank $8$ and $\FPdim$ $660$,  fifteen of type $[[1,1],[5,2],[10,2],[11,1],[12,2]]$:
$$ \begin{smallmatrix}1&0&0&0&0&0&0&0 \\ 0&1&0&0&0&0&0&0 \\ 0&0&1&0&0&0&0&0 \\ 0&0&0&1&0&0&0&0 \\ 0&0&0&0&1&0&0&0 \\ 0&0&0&0&0&1&0&0 \\ 0&0&0&0&0&0&1&0 \\ 0&0&0&0&0&0&0&1\end{smallmatrix} ,   \ \begin{smallmatrix}0&1&0&0&0&0&0&0 \\ 0&0&1&1&1&0&0&0 \\ 1&0&0&0&0&0&1&1 \\ 0&0&1&0&1&1&1&1 \\ 0&0&1&1&0&1&1&1 \\ 0&0&0&1&1&1&1&1 \\ 0&1&0&1&1&1&1&1 \\ 0&1&0&1&1&1&1&1\end{smallmatrix} ,   \ \begin{smallmatrix}0&0&1&0&0&0&0&0 \\ 1&0&0&0&0&0&1&1 \\ 0&1&0&1&1&0&0&0 \\ 0&1&0&0&1&1&1&1 \\ 0&1&0&1&0&1&1&1 \\ 0&0&0&1&1&1&1&1 \\ 0&0&1&1&1&1&1&1 \\ 0&0&1&1&1&1&1&1\end{smallmatrix} ,   \ \begin{smallmatrix}0&0&0&1&0&0&0&0 \\ 0&0&1&0&1&1&1&1 \\ 0&1&0&0&1&1&1&1 \\ 1&0&0&2&2&1&2&2 \\ 0&1&1&2&0&2&2&2 \\ 0&1&1&1&2&2&2&2 \\ 0&1&1&2&2&2&2&2 \\ 0&1&1&2&2&2&2&2\end{smallmatrix} ,   \ \begin{smallmatrix}0&0&0&0&1&0&0&0 \\ 0&0&1&1&0&1&1&1 \\ 0&1&0&1&0&1&1&1 \\ 0&1&1&2&0&2&2&2 \\ 1&0&0&0&4&1&2&2 \\ 0&1&1&2&1&2&2&2 \\ 0&1&1&2&2&2&2&2 \\ 0&1&1&2&2&2&2&2\end{smallmatrix} ,   \ \begin{smallmatrix}0&0&0&0&0&1&0&0 \\ 0&0&0&1&1&1&1&1 \\ 0&0&0&1&1&1&1&1 \\ 0&1&1&1&2&2&2&2 \\ 0&1&1&2&1&2&2&2 \\ 1&1&1&2&2&2&2&2 \\ 0&1&1&2&2&2&3&2 \\ 0&1&1&2&2&2&2&3\end{smallmatrix} ,   \ \begin{smallmatrix}0&0&0&0&0&0&1&0 \\ 0&1&0&1&1&1&1&1 \\ 0&0&1&1&1&1&1&1 \\ 0&1&1&2&2&2&2&2 \\ 0&1&1&2&2&2&2&2 \\ 0&1&1&2&2&2&3&2 \\ 1&1&1&2&2&3&0&5 \\ 0&1&1&2&2&2&5&1\end{smallmatrix} ,   \ \begin{smallmatrix}0&0&0&0&0&0&0&1 \\ 0&1&0&1&1&1&1&1 \\ 0&0&1&1&1&1&1&1 \\ 0&1&1&2&2&2&2&2 \\ 0&1&1&2&2&2&2&2 \\ 0&1&1&2&2&2&2&3 \\ 0&1&1&2&2&2&5&1 \\ 1&1&1&2&2&3&1&4\end{smallmatrix} $$

$$ \begin{smallmatrix}1&0&0&0&0&0&0&0 \\ 0&1&0&0&0&0&0&0 \\ 0&0&1&0&0&0&0&0 \\ 0&0&0&1&0&0&0&0 \\ 0&0&0&0&1&0&0&0 \\ 0&0&0&0&0&1&0&0 \\ 0&0&0&0&0&0&1&0 \\ 0&0&0&0&0&0&0&1\end{smallmatrix} ,   \ \begin{smallmatrix}0&1&0&0&0&0&0&0 \\ 0&0&1&1&1&0&0&0 \\ 1&0&0&0&0&0&1&1 \\ 0&0&1&0&1&1&1&1 \\ 0&0&1&1&0&1&1&1 \\ 0&0&0&1&1&1&1&1 \\ 0&1&0&1&1&1&1&1 \\ 0&1&0&1&1&1&1&1\end{smallmatrix} ,   \ \begin{smallmatrix}0&0&1&0&0&0&0&0 \\ 1&0&0&0&0&0&1&1 \\ 0&1&0&1&1&0&0&0 \\ 0&1&0&0&1&1&1&1 \\ 0&1&0&1&0&1&1&1 \\ 0&0&0&1&1&1&1&1 \\ 0&0&1&1&1&1&1&1 \\ 0&0&1&1&1&1&1&1\end{smallmatrix} ,   \ \begin{smallmatrix}0&0&0&1&0&0&0&0 \\ 0&0&1&0&1&1&1&1 \\ 0&1&0&0&1&1&1&1 \\ 1&0&0&2&2&1&2&2 \\ 0&1&1&2&0&2&2&2 \\ 0&1&1&1&2&2&2&2 \\ 0&1&1&2&2&2&2&2 \\ 0&1&1&2&2&2&2&2\end{smallmatrix} ,   \ \begin{smallmatrix}0&0&0&0&1&0&0&0 \\ 0&0&1&1&0&1&1&1 \\ 0&1&0&1&0&1&1&1 \\ 0&1&1&2&0&2&2&2 \\ 1&0&0&0&4&1&2&2 \\ 0&1&1&2&1&2&2&2 \\ 0&1&1&2&2&2&2&2 \\ 0&1&1&2&2&2&2&2\end{smallmatrix} ,   \ \begin{smallmatrix}0&0&0&0&0&1&0&0 \\ 0&0&0&1&1&1&1&1 \\ 0&0&0&1&1&1&1&1 \\ 0&1&1&1&2&2&2&2 \\ 0&1&1&2&1&2&2&2 \\ 1&1&1&2&2&2&2&2 \\ 0&1&1&2&2&2&3&2 \\ 0&1&1&2&2&2&2&3\end{smallmatrix} ,   \ \begin{smallmatrix}0&0&0&0&0&0&1&0 \\ 0&1&0&1&1&1&1&1 \\ 0&0&1&1&1&1&1&1 \\ 0&1&1&2&2&2&2&2 \\ 0&1&1&2&2&2&2&2 \\ 0&1&1&2&2&2&3&2 \\ 1&1&1&2&2&3&1&4 \\ 0&1&1&2&2&2&4&2\end{smallmatrix} ,   \ \begin{smallmatrix}0&0&0&0&0&0&0&1 \\ 0&1&0&1&1&1&1&1 \\ 0&0&1&1&1&1&1&1 \\ 0&1&1&2&2&2&2&2 \\ 0&1&1&2&2&2&2&2 \\ 0&1&1&2&2&2&2&3 \\ 0&1&1&2&2&2&4&2 \\ 1&1&1&2&2&3&2&3\end{smallmatrix} $$

$$ \begin{smallmatrix}1&0&0&0&0&0&0&0 \\ 0&1&0&0&0&0&0&0 \\ 0&0&1&0&0&0&0&0 \\ 0&0&0&1&0&0&0&0 \\ 0&0&0&0&1&0&0&0 \\ 0&0&0&0&0&1&0&0 \\ 0&0&0&0&0&0&1&0 \\ 0&0&0&0&0&0&0&1\end{smallmatrix} ,   \ \begin{smallmatrix}0&1&0&0&0&0&0&0 \\ 0&0&1&1&1&0&0&0 \\ 1&0&0&0&0&0&1&1 \\ 0&0&1&0&1&1&1&1 \\ 0&0&1&1&0&1&1&1 \\ 0&0&0&1&1&1&1&1 \\ 0&1&0&1&1&1&1&1 \\ 0&1&0&1&1&1&1&1\end{smallmatrix} ,   \ \begin{smallmatrix}0&0&1&0&0&0&0&0 \\ 1&0&0&0&0&0&1&1 \\ 0&1&0&1&1&0&0&0 \\ 0&1&0&0&1&1&1&1 \\ 0&1&0&1&0&1&1&1 \\ 0&0&0&1&1&1&1&1 \\ 0&0&1&1&1&1&1&1 \\ 0&0&1&1&1&1&1&1\end{smallmatrix} ,   \ \begin{smallmatrix}0&0&0&1&0&0&0&0 \\ 0&0&1&0&1&1&1&1 \\ 0&1&0&0&1&1&1&1 \\ 1&0&0&2&2&1&2&2 \\ 0&1&1&2&0&2&2&2 \\ 0&1&1&1&2&2&2&2 \\ 0&1&1&2&2&2&2&2 \\ 0&1&1&2&2&2&2&2\end{smallmatrix} ,   \ \begin{smallmatrix}0&0&0&0&1&0&0&0 \\ 0&0&1&1&0&1&1&1 \\ 0&1&0&1&0&1&1&1 \\ 0&1&1&2&0&2&2&2 \\ 1&0&0&0&4&1&2&2 \\ 0&1&1&2&1&2&2&2 \\ 0&1&1&2&2&2&2&2 \\ 0&1&1&2&2&2&2&2\end{smallmatrix} ,   \ \begin{smallmatrix}0&0&0&0&0&1&0&0 \\ 0&0&0&1&1&1&1&1 \\ 0&0&0&1&1&1&1&1 \\ 0&1&1&1&2&2&2&2 \\ 0&1&1&2&1&2&2&2 \\ 1&1&1&2&2&2&2&2 \\ 0&1&1&2&2&2&3&2 \\ 0&1&1&2&2&2&2&3\end{smallmatrix} ,   \ \begin{smallmatrix}0&0&0&0&0&0&1&0 \\ 0&1&0&1&1&1&1&1 \\ 0&0&1&1&1&1&1&1 \\ 0&1&1&2&2&2&2&2 \\ 0&1&1&2&2&2&2&2 \\ 0&1&1&2&2&2&3&2 \\ 1&1&1&2&2&3&2&3 \\ 0&1&1&2&2&2&3&3\end{smallmatrix} ,   \ \begin{smallmatrix}0&0&0&0&0&0&0&1 \\ 0&1&0&1&1&1&1&1 \\ 0&0&1&1&1&1&1&1 \\ 0&1&1&2&2&2&2&2 \\ 0&1&1&2&2&2&2&2 \\ 0&1&1&2&2&2&2&3 \\ 0&1&1&2&2&2&3&3 \\ 1&1&1&2&2&3&3&2\end{smallmatrix} $$
$$ \begin{smallmatrix}1&0&0&0&0&0&0&0 \\ 0&1&0&0&0&0&0&0 \\ 0&0&1&0&0&0&0&0 \\ 0&0&0&1&0&0&0&0 \\ 0&0&0&0&1&0&0&0 \\ 0&0&0&0&0&1&0&0 \\ 0&0&0&0&0&0&1&0 \\ 0&0&0&0&0&0&0&1\end{smallmatrix} ,   \ \begin{smallmatrix}0&1&0&0&0&0&0&0 \\ 0&0&1&1&1&0&0&0 \\ 1&0&0&0&0&0&1&1 \\ 0&0&1&0&1&1&1&1 \\ 0&0&1&1&0&1&1&1 \\ 0&0&0&1&1&1&1&1 \\ 0&1&0&1&1&1&1&1 \\ 0&1&0&1&1&1&1&1\end{smallmatrix} ,   \ \begin{smallmatrix}0&0&1&0&0&0&0&0 \\ 1&0&0&0&0&0&1&1 \\ 0&1&0&1&1&0&0&0 \\ 0&1&0&0&1&1&1&1 \\ 0&1&0&1&0&1&1&1 \\ 0&0&0&1&1&1&1&1 \\ 0&0&1&1&1&1&1&1 \\ 0&0&1&1&1&1&1&1\end{smallmatrix} ,   \ \begin{smallmatrix}0&0&0&1&0&0&0&0 \\ 0&0&1&0&1&1&1&1 \\ 0&1&0&0&1&1&1&1 \\ 1&0&0&3&1&1&2&2 \\ 0&1&1&1&1&2&2&2 \\ 0&1&1&1&2&2&2&2 \\ 0&1&1&2&2&2&2&2 \\ 0&1&1&2&2&2&2&2\end{smallmatrix} ,   \ \begin{smallmatrix}0&0&0&0&1&0&0&0 \\ 0&0&1&1&0&1&1&1 \\ 0&1&0&1&0&1&1&1 \\ 0&1&1&1&1&2&2&2 \\ 1&0&0&1&3&1&2&2 \\ 0&1&1&2&1&2&2&2 \\ 0&1&1&2&2&2&2&2 \\ 0&1&1&2&2&2&2&2\end{smallmatrix} ,   \ \begin{smallmatrix}0&0&0&0&0&1&0&0 \\ 0&0&0&1&1&1&1&1 \\ 0&0&0&1&1&1&1&1 \\ 0&1&1&1&2&2&2&2 \\ 0&1&1&2&1&2&2&2 \\ 1&1&1&2&2&2&2&2 \\ 0&1&1&2&2&2&3&2 \\ 0&1&1&2&2&2&2&3\end{smallmatrix} ,   \ \begin{smallmatrix}0&0&0&0&0&0&1&0 \\ 0&1&0&1&1&1&1&1 \\ 0&0&1&1&1&1&1&1 \\ 0&1&1&2&2&2&2&2 \\ 0&1&1&2&2&2&2&2 \\ 0&1&1&2&2&2&3&2 \\ 1&1&1&2&2&3&0&5 \\ 0&1&1&2&2&2&5&1\end{smallmatrix} ,   \ \begin{smallmatrix}0&0&0&0&0&0&0&1 \\ 0&1&0&1&1&1&1&1 \\ 0&0&1&1&1&1&1&1 \\ 0&1&1&2&2&2&2&2 \\ 0&1&1&2&2&2&2&2 \\ 0&1&1&2&2&2&2&3 \\ 0&1&1&2&2&2&5&1 \\ 1&1&1&2&2&3&1&4\end{smallmatrix} $$

$$ \begin{smallmatrix}1&0&0&0&0&0&0&0 \\ 0&1&0&0&0&0&0&0 \\ 0&0&1&0&0&0&0&0 \\ 0&0&0&1&0&0&0&0 \\ 0&0&0&0&1&0&0&0 \\ 0&0&0&0&0&1&0&0 \\ 0&0&0&0&0&0&1&0 \\ 0&0&0&0&0&0&0&1\end{smallmatrix} ,   \ \begin{smallmatrix}0&1&0&0&0&0&0&0 \\ 0&0&1&1&1&0&0&0 \\ 1&0&0&0&0&0&1&1 \\ 0&0&1&0&1&1&1&1 \\ 0&0&1&1&0&1&1&1 \\ 0&0&0&1&1&1&1&1 \\ 0&1&0&1&1&1&1&1 \\ 0&1&0&1&1&1&1&1\end{smallmatrix} ,   \ \begin{smallmatrix}0&0&1&0&0&0&0&0 \\ 1&0&0&0&0&0&1&1 \\ 0&1&0&1&1&0&0&0 \\ 0&1&0&0&1&1&1&1 \\ 0&1&0&1&0&1&1&1 \\ 0&0&0&1&1&1&1&1 \\ 0&0&1&1&1&1&1&1 \\ 0&0&1&1&1&1&1&1\end{smallmatrix} ,   \ \begin{smallmatrix}0&0&0&1&0&0&0&0 \\ 0&0&1&0&1&1&1&1 \\ 0&1&0&0&1&1&1&1 \\ 1&0&0&3&1&1&2&2 \\ 0&1&1&1&1&2&2&2 \\ 0&1&1&1&2&2&2&2 \\ 0&1&1&2&2&2&2&2 \\ 0&1&1&2&2&2&2&2\end{smallmatrix} ,   \ \begin{smallmatrix}0&0&0&0&1&0&0&0 \\ 0&0&1&1&0&1&1&1 \\ 0&1&0&1&0&1&1&1 \\ 0&1&1&1&1&2&2&2 \\ 1&0&0&1&3&1&2&2 \\ 0&1&1&2&1&2&2&2 \\ 0&1&1&2&2&2&2&2 \\ 0&1&1&2&2&2&2&2\end{smallmatrix} ,   \ \begin{smallmatrix}0&0&0&0&0&1&0&0 \\ 0&0&0&1&1&1&1&1 \\ 0&0&0&1&1&1&1&1 \\ 0&1&1&1&2&2&2&2 \\ 0&1&1&2&1&2&2&2 \\ 1&1&1&2&2&2&2&2 \\ 0&1&1&2&2&2&3&2 \\ 0&1&1&2&2&2&2&3\end{smallmatrix} ,   \ \begin{smallmatrix}0&0&0&0&0&0&1&0 \\ 0&1&0&1&1&1&1&1 \\ 0&0&1&1&1&1&1&1 \\ 0&1&1&2&2&2&2&2 \\ 0&1&1&2&2&2&2&2 \\ 0&1&1&2&2&2&3&2 \\ 1&1&1&2&2&3&1&4 \\ 0&1&1&2&2&2&4&2\end{smallmatrix} ,   \ \begin{smallmatrix}0&0&0&0&0&0&0&1 \\ 0&1&0&1&1&1&1&1 \\ 0&0&1&1&1&1&1&1 \\ 0&1&1&2&2&2&2&2 \\ 0&1&1&2&2&2&2&2 \\ 0&1&1&2&2&2&2&3 \\ 0&1&1&2&2&2&4&2 \\ 1&1&1&2&2&3&2&3\end{smallmatrix} $$

$$ \begin{smallmatrix}1&0&0&0&0&0&0&0 \\ 0&1&0&0&0&0&0&0 \\ 0&0&1&0&0&0&0&0 \\ 0&0&0&1&0&0&0&0 \\ 0&0&0&0&1&0&0&0 \\ 0&0&0&0&0&1&0&0 \\ 0&0&0&0&0&0&1&0 \\ 0&0&0&0&0&0&0&1\end{smallmatrix} ,   \ \begin{smallmatrix}0&1&0&0&0&0&0&0 \\ 0&0&1&1&1&0&0&0 \\ 1&0&0&0&0&0&1&1 \\ 0&0&1&1&0&1&1&1 \\ 0&0&1&0&1&1&1&1 \\ 0&0&0&1&1&1&1&1 \\ 0&1&0&1&1&1&1&1 \\ 0&1&0&1&1&1&1&1\end{smallmatrix} ,   \ \begin{smallmatrix}0&0&1&0&0&0&0&0 \\ 1&0&0&0&0&0&1&1 \\ 0&1&0&1&1&0&0&0 \\ 0&1&0&1&0&1&1&1 \\ 0&1&0&0&1&1&1&1 \\ 0&0&0&1&1&1&1&1 \\ 0&0&1&1&1&1&1&1 \\ 0&0&1&1&1&1&1&1\end{smallmatrix} ,   \ \begin{smallmatrix}0&0&0&1&0&0&0&0 \\ 0&0&1&1&0&1&1&1 \\ 0&1&0&1&0&1&1&1 \\ 1&1&1&0&3&1&2&2 \\ 0&0&0&3&0&2&2&2 \\ 0&1&1&1&2&2&2&2 \\ 0&1&1&2&2&2&2&2 \\ 0&1&1&2&2&2&2&2\end{smallmatrix} ,   \ \begin{smallmatrix}0&0&0&0&1&0&0&0 \\ 0&0&1&0&1&1&1&1 \\ 0&1&0&0&1&1&1&1 \\ 0&0&0&3&0&2&2&2 \\ 1&1&1&0&3&1&2&2 \\ 0&1&1&2&1&2&2&2 \\ 0&1&1&2&2&2&2&2 \\ 0&1&1&2&2&2&2&2\end{smallmatrix} ,   \ \begin{smallmatrix}0&0&0&0&0&1&0&0 \\ 0&0&0&1&1&1&1&1 \\ 0&0&0&1&1&1&1&1 \\ 0&1&1&1&2&2&2&2 \\ 0&1&1&2&1&2&2&2 \\ 1&1&1&2&2&2&2&2 \\ 0&1&1&2&2&2&3&2 \\ 0&1&1&2&2&2&2&3\end{smallmatrix} ,   \ \begin{smallmatrix}0&0&0&0&0&0&1&0 \\ 0&1&0&1&1&1&1&1 \\ 0&0&1&1&1&1&1&1 \\ 0&1&1&2&2&2&2&2 \\ 0&1&1&2&2&2&2&2 \\ 0&1&1&2&2&2&3&2 \\ 1&1&1&2&2&3&0&5 \\ 0&1&1&2&2&2&5&1\end{smallmatrix} ,   \ \begin{smallmatrix}0&0&0&0&0&0&0&1 \\ 0&1&0&1&1&1&1&1 \\ 0&0&1&1&1&1&1&1 \\ 0&1&1&2&2&2&2&2 \\ 0&1&1&2&2&2&2&2 \\ 0&1&1&2&2&2&2&3 \\ 0&1&1&2&2&2&5&1 \\ 1&1&1&2&2&3&1&4\end{smallmatrix} $$

$$ \begin{smallmatrix}1&0&0&0&0&0&0&0 \\ 0&1&0&0&0&0&0&0 \\ 0&0&1&0&0&0&0&0 \\ 0&0&0&1&0&0&0&0 \\ 0&0&0&0&1&0&0&0 \\ 0&0&0&0&0&1&0&0 \\ 0&0&0&0&0&0&1&0 \\ 0&0&0&0&0&0&0&1\end{smallmatrix} ,   \ \begin{smallmatrix}0&1&0&0&0&0&0&0 \\ 0&0&1&1&1&0&0&0 \\ 1&0&0&0&0&0&1&1 \\ 0&0&1&1&0&1&1&1 \\ 0&0&1&0&1&1&1&1 \\ 0&0&0&1&1&1&1&1 \\ 0&1&0&1&1&1&1&1 \\ 0&1&0&1&1&1&1&1\end{smallmatrix} ,   \ \begin{smallmatrix}0&0&1&0&0&0&0&0 \\ 1&0&0&0&0&0&1&1 \\ 0&1&0&1&1&0&0&0 \\ 0&1&0&1&0&1&1&1 \\ 0&1&0&0&1&1&1&1 \\ 0&0&0&1&1&1&1&1 \\ 0&0&1&1&1&1&1&1 \\ 0&0&1&1&1&1&1&1\end{smallmatrix} ,   \ \begin{smallmatrix}0&0&0&1&0&0&0&0 \\ 0&0&1&1&0&1&1&1 \\ 0&1&0&1&0&1&1&1 \\ 1&1&1&0&3&1&2&2 \\ 0&0&0&3&0&2&2&2 \\ 0&1&1&1&2&2&2&2 \\ 0&1&1&2&2&2&2&2 \\ 0&1&1&2&2&2&2&2\end{smallmatrix} ,   \ \begin{smallmatrix}0&0&0&0&1&0&0&0 \\ 0&0&1&0&1&1&1&1 \\ 0&1&0&0&1&1&1&1 \\ 0&0&0&3&0&2&2&2 \\ 1&1&1&0&3&1&2&2 \\ 0&1&1&2&1&2&2&2 \\ 0&1&1&2&2&2&2&2 \\ 0&1&1&2&2&2&2&2\end{smallmatrix} ,   \ \begin{smallmatrix}0&0&0&0&0&1&0&0 \\ 0&0&0&1&1&1&1&1 \\ 0&0&0&1&1&1&1&1 \\ 0&1&1&1&2&2&2&2 \\ 0&1&1&2&1&2&2&2 \\ 1&1&1&2&2&2&2&2 \\ 0&1&1&2&2&2&3&2 \\ 0&1&1&2&2&2&2&3\end{smallmatrix} ,   \ \begin{smallmatrix}0&0&0&0&0&0&1&0 \\ 0&1&0&1&1&1&1&1 \\ 0&0&1&1&1&1&1&1 \\ 0&1&1&2&2&2&2&2 \\ 0&1&1&2&2&2&2&2 \\ 0&1&1&2&2&2&3&2 \\ 1&1&1&2&2&3&1&4 \\ 0&1&1&2&2&2&4&2\end{smallmatrix} ,   \ \begin{smallmatrix}0&0&0&0&0&0&0&1 \\ 0&1&0&1&1&1&1&1 \\ 0&0&1&1&1&1&1&1 \\ 0&1&1&2&2&2&2&2 \\ 0&1&1&2&2&2&2&2 \\ 0&1&1&2&2&2&2&3 \\ 0&1&1&2&2&2&4&2 \\ 1&1&1&2&2&3&2&3\end{smallmatrix} $$

$$ \begin{smallmatrix}1&0&0&0&0&0&0&0 \\ 0&1&0&0&0&0&0&0 \\ 0&0&1&0&0&0&0&0 \\ 0&0&0&1&0&0&0&0 \\ 0&0&0&0&1&0&0&0 \\ 0&0&0&0&0&1&0&0 \\ 0&0&0&0&0&0&1&0 \\ 0&0&0&0&0&0&0&1\end{smallmatrix} ,   \ \begin{smallmatrix}0&1&0&0&0&0&0&0 \\ 0&0&1&1&1&0&0&0 \\ 1&0&0&0&0&0&1&1 \\ 0&0&1&1&0&1&1&1 \\ 0&0&1&0&1&1&1&1 \\ 0&0&0&1&1&1&1&1 \\ 0&1&0&1&1&1&1&1 \\ 0&1&0&1&1&1&1&1\end{smallmatrix} ,   \ \begin{smallmatrix}0&0&1&0&0&0&0&0 \\ 1&0&0&0&0&0&1&1 \\ 0&1&0&1&1&0&0&0 \\ 0&1&0&1&0&1&1&1 \\ 0&1&0&0&1&1&1&1 \\ 0&0&0&1&1&1&1&1 \\ 0&0&1&1&1&1&1&1 \\ 0&0&1&1&1&1&1&1\end{smallmatrix} ,   \ \begin{smallmatrix}0&0&0&1&0&0&0&0 \\ 0&0&1&1&0&1&1&1 \\ 0&1&0&1&0&1&1&1 \\ 1&1&1&0&3&1&2&2 \\ 0&0&0&3&0&2&2&2 \\ 0&1&1&1&2&2&2&2 \\ 0&1&1&2&2&2&2&2 \\ 0&1&1&2&2&2&2&2\end{smallmatrix} ,   \ \begin{smallmatrix}0&0&0&0&1&0&0&0 \\ 0&0&1&0&1&1&1&1 \\ 0&1&0&0&1&1&1&1 \\ 0&0&0&3&0&2&2&2 \\ 1&1&1&0&3&1&2&2 \\ 0&1&1&2&1&2&2&2 \\ 0&1&1&2&2&2&2&2 \\ 0&1&1&2&2&2&2&2\end{smallmatrix} ,   \ \begin{smallmatrix}0&0&0&0&0&1&0&0 \\ 0&0&0&1&1&1&1&1 \\ 0&0&0&1&1&1&1&1 \\ 0&1&1&1&2&2&2&2 \\ 0&1&1&2&1&2&2&2 \\ 1&1&1&2&2&2&2&2 \\ 0&1&1&2&2&2&3&2 \\ 0&1&1&2&2&2&2&3\end{smallmatrix} ,   \ \begin{smallmatrix}0&0&0&0&0&0&1&0 \\ 0&1&0&1&1&1&1&1 \\ 0&0&1&1&1&1&1&1 \\ 0&1&1&2&2&2&2&2 \\ 0&1&1&2&2&2&2&2 \\ 0&1&1&2&2&2&3&2 \\ 1&1&1&2&2&3&2&3 \\ 0&1&1&2&2&2&3&3\end{smallmatrix} ,   \ \begin{smallmatrix}0&0&0&0&0&0&0&1 \\ 0&1&0&1&1&1&1&1 \\ 0&0&1&1&1&1&1&1 \\ 0&1&1&2&2&2&2&2 \\ 0&1&1&2&2&2&2&2 \\ 0&1&1&2&2&2&2&3 \\ 0&1&1&2&2&2&3&3 \\ 1&1&1&2&2&3&3&2\end{smallmatrix} $$

$$ \begin{smallmatrix}1&0&0&0&0&0&0&0 \\ 0&1&0&0&0&0&0&0 \\ 0&0&1&0&0&0&0&0 \\ 0&0&0&1&0&0&0&0 \\ 0&0&0&0&1&0&0&0 \\ 0&0&0&0&0&1&0&0 \\ 0&0&0&0&0&0&1&0 \\ 0&0&0&0&0&0&0&1\end{smallmatrix} ,   \ \begin{smallmatrix}0&1&0&0&0&0&0&0 \\ 0&0&1&1&1&0&0&0 \\ 1&0&0&0&0&0&1&1 \\ 0&0&1&1&0&1&1&1 \\ 0&0&1&0&1&1&1&1 \\ 0&0&0&1&1&1&1&1 \\ 0&1&0&1&1&1&1&1 \\ 0&1&0&1&1&1&1&1\end{smallmatrix} ,   \ \begin{smallmatrix}0&0&1&0&0&0&0&0 \\ 1&0&0&0&0&0&1&1 \\ 0&1&0&1&1&0&0&0 \\ 0&1&0&1&0&1&1&1 \\ 0&1&0&0&1&1&1&1 \\ 0&0&0&1&1&1&1&1 \\ 0&0&1&1&1&1&1&1 \\ 0&0&1&1&1&1&1&1\end{smallmatrix} ,   \ \begin{smallmatrix}0&0&0&1&0&0&0&0 \\ 0&0&1&1&0&1&1&1 \\ 0&1&0&1&0&1&1&1 \\ 1&1&1&1&2&1&2&2 \\ 0&0&0&2&1&2&2&2 \\ 0&1&1&1&2&2&2&2 \\ 0&1&1&2&2&2&2&2 \\ 0&1&1&2&2&2&2&2\end{smallmatrix} ,   \ \begin{smallmatrix}0&0&0&0&1&0&0&0 \\ 0&0&1&0&1&1&1&1 \\ 0&1&0&0&1&1&1&1 \\ 0&0&0&2&1&2&2&2 \\ 1&1&1&1&2&1&2&2 \\ 0&1&1&2&1&2&2&2 \\ 0&1&1&2&2&2&2&2 \\ 0&1&1&2&2&2&2&2\end{smallmatrix} ,   \ \begin{smallmatrix}0&0&0&0&0&1&0&0 \\ 0&0&0&1&1&1&1&1 \\ 0&0&0&1&1&1&1&1 \\ 0&1&1&1&2&2&2&2 \\ 0&1&1&2&1&2&2&2 \\ 1&1&1&2&2&2&2&2 \\ 0&1&1&2&2&2&3&2 \\ 0&1&1&2&2&2&2&3\end{smallmatrix} ,   \ \begin{smallmatrix}0&0&0&0&0&0&1&0 \\ 0&1&0&1&1&1&1&1 \\ 0&0&1&1&1&1&1&1 \\ 0&1&1&2&2&2&2&2 \\ 0&1&1&2&2&2&2&2 \\ 0&1&1&2&2&2&3&2 \\ 1&1&1&2&2&3&0&5 \\ 0&1&1&2&2&2&5&1\end{smallmatrix} ,   \ \begin{smallmatrix}0&0&0&0&0&0&0&1 \\ 0&1&0&1&1&1&1&1 \\ 0&0&1&1&1&1&1&1 \\ 0&1&1&2&2&2&2&2 \\ 0&1&1&2&2&2&2&2 \\ 0&1&1&2&2&2&2&3 \\ 0&1&1&2&2&2&5&1 \\ 1&1&1&2&2&3&1&4\end{smallmatrix} $$

$$ \begin{smallmatrix}1&0&0&0&0&0&0&0 \\ 0&1&0&0&0&0&0&0 \\ 0&0&1&0&0&0&0&0 \\ 0&0&0&1&0&0&0&0 \\ 0&0&0&0&1&0&0&0 \\ 0&0&0&0&0&1&0&0 \\ 0&0&0&0&0&0&1&0 \\ 0&0&0&0&0&0&0&1\end{smallmatrix} ,   \ \begin{smallmatrix}0&1&0&0&0&0&0&0 \\ 0&0&1&1&1&0&0&0 \\ 1&0&0&0&0&0&1&1 \\ 0&0&1&1&0&1&1&1 \\ 0&0&1&0&1&1&1&1 \\ 0&0&0&1&1&1&1&1 \\ 0&1&0&1&1&1&1&1 \\ 0&1&0&1&1&1&1&1\end{smallmatrix} ,   \ \begin{smallmatrix}0&0&1&0&0&0&0&0 \\ 1&0&0&0&0&0&1&1 \\ 0&1&0&1&1&0&0&0 \\ 0&1&0&1&0&1&1&1 \\ 0&1&0&0&1&1&1&1 \\ 0&0&0&1&1&1&1&1 \\ 0&0&1&1&1&1&1&1 \\ 0&0&1&1&1&1&1&1\end{smallmatrix} ,   \ \begin{smallmatrix}0&0&0&1&0&0&0&0 \\ 0&0&1&1&0&1&1&1 \\ 0&1&0&1&0&1&1&1 \\ 1&1&1&1&2&1&2&2 \\ 0&0&0&2&1&2&2&2 \\ 0&1&1&1&2&2&2&2 \\ 0&1&1&2&2&2&2&2 \\ 0&1&1&2&2&2&2&2\end{smallmatrix} ,   \ \begin{smallmatrix}0&0&0&0&1&0&0&0 \\ 0&0&1&0&1&1&1&1 \\ 0&1&0&0&1&1&1&1 \\ 0&0&0&2&1&2&2&2 \\ 1&1&1&1&2&1&2&2 \\ 0&1&1&2&1&2&2&2 \\ 0&1&1&2&2&2&2&2 \\ 0&1&1&2&2&2&2&2\end{smallmatrix} ,   \ \begin{smallmatrix}0&0&0&0&0&1&0&0 \\ 0&0&0&1&1&1&1&1 \\ 0&0&0&1&1&1&1&1 \\ 0&1&1&1&2&2&2&2 \\ 0&1&1&2&1&2&2&2 \\ 1&1&1&2&2&2&2&2 \\ 0&1&1&2&2&2&3&2 \\ 0&1&1&2&2&2&2&3\end{smallmatrix} ,   \ \begin{smallmatrix}0&0&0&0&0&0&1&0 \\ 0&1&0&1&1&1&1&1 \\ 0&0&1&1&1&1&1&1 \\ 0&1&1&2&2&2&2&2 \\ 0&1&1&2&2&2&2&2 \\ 0&1&1&2&2&2&3&2 \\ 1&1&1&2&2&3&1&4 \\ 0&1&1&2&2&2&4&2\end{smallmatrix} ,   \ \begin{smallmatrix}0&0&0&0&0&0&0&1 \\ 0&1&0&1&1&1&1&1 \\ 0&0&1&1&1&1&1&1 \\ 0&1&1&2&2&2&2&2 \\ 0&1&1&2&2&2&2&2 \\ 0&1&1&2&2&2&2&3 \\ 0&1&1&2&2&2&4&2 \\ 1&1&1&2&2&3&2&3\end{smallmatrix} $$

$$ \begin{smallmatrix}1&0&0&0&0&0&0&0 \\ 0&1&0&0&0&0&0&0 \\ 0&0&1&0&0&0&0&0 \\ 0&0&0&1&0&0&0&0 \\ 0&0&0&0&1&0&0&0 \\ 0&0&0&0&0&1&0&0 \\ 0&0&0&0&0&0&1&0 \\ 0&0&0&0&0&0&0&1\end{smallmatrix} ,   \ \begin{smallmatrix}0&1&0&0&0&0&0&0 \\ 0&0&1&1&1&0&0&0 \\ 1&0&0&0&0&0&1&1 \\ 0&0&1&0&1&1&1&1 \\ 0&0&1&1&0&1&1&1 \\ 0&0&0&1&1&1&1&1 \\ 0&1&0&1&1&1&1&1 \\ 0&1&0&1&1&1&1&1\end{smallmatrix} ,   \ \begin{smallmatrix}0&0&1&0&0&0&0&0 \\ 1&0&0&0&0&0&1&1 \\ 0&1&0&1&1&0&0&0 \\ 0&1&0&0&1&1&1&1 \\ 0&1&0&1&0&1&1&1 \\ 0&0&0&1&1&1&1&1 \\ 0&0&1&1&1&1&1&1 \\ 0&0&1&1&1&1&1&1\end{smallmatrix} ,   \ \begin{smallmatrix}0&0&0&1&0&0&0&0 \\ 0&0&1&0&1&1&1&1 \\ 0&1&0&0&1&1&1&1 \\ 0&1&1&2&0&2&2&2 \\ 1&0&0&2&2&1&2&2 \\ 0&1&1&1&2&2&2&2 \\ 0&1&1&2&2&2&2&2 \\ 0&1&1&2&2&2&2&2\end{smallmatrix} ,   \ \begin{smallmatrix}0&0&0&0&1&0&0&0 \\ 0&0&1&1&0&1&1&1 \\ 0&1&0&1&0&1&1&1 \\ 1&0&0&2&2&1&2&2 \\ 0&1&1&0&2&2&2&2 \\ 0&1&1&2&1&2&2&2 \\ 0&1&1&2&2&2&2&2 \\ 0&1&1&2&2&2&2&2\end{smallmatrix} ,   \ \begin{smallmatrix}0&0&0&0&0&1&0&0 \\ 0&0&0&1&1&1&1&1 \\ 0&0&0&1&1&1&1&1 \\ 0&1&1&1&2&2&2&2 \\ 0&1&1&2&1&2&2&2 \\ 1&1&1&2&2&2&2&2 \\ 0&1&1&2&2&2&3&2 \\ 0&1&1&2&2&2&2&3\end{smallmatrix} ,   \ \begin{smallmatrix}0&0&0&0&0&0&1&0 \\ 0&1&0&1&1&1&1&1 \\ 0&0&1&1&1&1&1&1 \\ 0&1&1&2&2&2&2&2 \\ 0&1&1&2&2&2&2&2 \\ 0&1&1&2&2&2&3&2 \\ 1&1&1&2&2&3&0&5 \\ 0&1&1&2&2&2&5&1\end{smallmatrix} ,   \ \begin{smallmatrix}0&0&0&0&0&0&0&1 \\ 0&1&0&1&1&1&1&1 \\ 0&0&1&1&1&1&1&1 \\ 0&1&1&2&2&2&2&2 \\ 0&1&1&2&2&2&2&2 \\ 0&1&1&2&2&2&2&3 \\ 0&1&1&2&2&2&5&1 \\ 1&1&1&2&2&3&1&4\end{smallmatrix} $$

$$ \begin{smallmatrix}1&0&0&0&0&0&0&0 \\ 0&1&0&0&0&0&0&0 \\ 0&0&1&0&0&0&0&0 \\ 0&0&0&1&0&0&0&0 \\ 0&0&0&0&1&0&0&0 \\ 0&0&0&0&0&1&0&0 \\ 0&0&0&0&0&0&1&0 \\ 0&0&0&0&0&0&0&1\end{smallmatrix} ,   \ \begin{smallmatrix}0&1&0&0&0&0&0&0 \\ 0&0&1&1&1&0&0&0 \\ 1&0&0&0&0&0&1&1 \\ 0&0&1&0&1&1&1&1 \\ 0&0&1&1&0&1&1&1 \\ 0&0&0&1&1&1&1&1 \\ 0&1&0&1&1&1&1&1 \\ 0&1&0&1&1&1&1&1\end{smallmatrix} ,   \ \begin{smallmatrix}0&0&1&0&0&0&0&0 \\ 1&0&0&0&0&0&1&1 \\ 0&1&0&1&1&0&0&0 \\ 0&1&0&0&1&1&1&1 \\ 0&1&0&1&0&1&1&1 \\ 0&0&0&1&1&1&1&1 \\ 0&0&1&1&1&1&1&1 \\ 0&0&1&1&1&1&1&1\end{smallmatrix} ,   \ \begin{smallmatrix}0&0&0&1&0&0&0&0 \\ 0&0&1&0&1&1&1&1 \\ 0&1&0&0&1&1&1&1 \\ 0&1&1&2&0&2&2&2 \\ 1&0&0&2&2&1&2&2 \\ 0&1&1&1&2&2&2&2 \\ 0&1&1&2&2&2&2&2 \\ 0&1&1&2&2&2&2&2\end{smallmatrix} ,   \ \begin{smallmatrix}0&0&0&0&1&0&0&0 \\ 0&0&1&1&0&1&1&1 \\ 0&1&0&1&0&1&1&1 \\ 1&0&0&2&2&1&2&2 \\ 0&1&1&0&2&2&2&2 \\ 0&1&1&2&1&2&2&2 \\ 0&1&1&2&2&2&2&2 \\ 0&1&1&2&2&2&2&2\end{smallmatrix} ,   \ \begin{smallmatrix}0&0&0&0&0&1&0&0 \\ 0&0&0&1&1&1&1&1 \\ 0&0&0&1&1&1&1&1 \\ 0&1&1&1&2&2&2&2 \\ 0&1&1&2&1&2&2&2 \\ 1&1&1&2&2&2&2&2 \\ 0&1&1&2&2&2&3&2 \\ 0&1&1&2&2&2&2&3\end{smallmatrix} ,   \ \begin{smallmatrix}0&0&0&0&0&0&1&0 \\ 0&1&0&1&1&1&1&1 \\ 0&0&1&1&1&1&1&1 \\ 0&1&1&2&2&2&2&2 \\ 0&1&1&2&2&2&2&2 \\ 0&1&1&2&2&2&3&2 \\ 1&1&1&2&2&3&1&4 \\ 0&1&1&2&2&2&4&2\end{smallmatrix} ,   \ \begin{smallmatrix}0&0&0&0&0&0&0&1 \\ 0&1&0&1&1&1&1&1 \\ 0&0&1&1&1&1&1&1 \\ 0&1&1&2&2&2&2&2 \\ 0&1&1&2&2&2&2&2 \\ 0&1&1&2&2&2&2&3 \\ 0&1&1&2&2&2&4&2 \\ 1&1&1&2&2&3&2&3\end{smallmatrix} $$

$$ \begin{smallmatrix}1&0&0&0&0&0&0&0 \\ 0&1&0&0&0&0&0&0 \\ 0&0&1&0&0&0&0&0 \\ 0&0&0&1&0&0&0&0 \\ 0&0&0&0&1&0&0&0 \\ 0&0&0&0&0&1&0&0 \\ 0&0&0&0&0&0&1&0 \\ 0&0&0&0&0&0&0&1\end{smallmatrix} ,   \ \begin{smallmatrix}0&1&0&0&0&0&0&0 \\ 0&0&1&1&1&0&0&0 \\ 1&0&0&0&0&0&1&1 \\ 0&0&1&0&1&1&1&1 \\ 0&0&1&1&0&1&1&1 \\ 0&0&0&1&1&1&1&1 \\ 0&1&0&1&1&1&1&1 \\ 0&1&0&1&1&1&1&1\end{smallmatrix} ,   \ \begin{smallmatrix}0&0&1&0&0&0&0&0 \\ 1&0&0&0&0&0&1&1 \\ 0&1&0&1&1&0&0&0 \\ 0&1&0&0&1&1&1&1 \\ 0&1&0&1&0&1&1&1 \\ 0&0&0&1&1&1&1&1 \\ 0&0&1&1&1&1&1&1 \\ 0&0&1&1&1&1&1&1\end{smallmatrix} ,   \ \begin{smallmatrix}0&0&0&1&0&0&0&0 \\ 0&0&1&0&1&1&1&1 \\ 0&1&0&0&1&1&1&1 \\ 0&1&1&2&0&2&2&2 \\ 1&0&0&2&2&1&2&2 \\ 0&1&1&1&2&2&2&2 \\ 0&1&1&2&2&2&2&2 \\ 0&1&1&2&2&2&2&2\end{smallmatrix} ,   \ \begin{smallmatrix}0&0&0&0&1&0&0&0 \\ 0&0&1&1&0&1&1&1 \\ 0&1&0&1&0&1&1&1 \\ 1&0&0&2&2&1&2&2 \\ 0&1&1&0&2&2&2&2 \\ 0&1&1&2&1&2&2&2 \\ 0&1&1&2&2&2&2&2 \\ 0&1&1&2&2&2&2&2\end{smallmatrix} ,   \ \begin{smallmatrix}0&0&0&0&0&1&0&0 \\ 0&0&0&1&1&1&1&1 \\ 0&0&0&1&1&1&1&1 \\ 0&1&1&1&2&2&2&2 \\ 0&1&1&2&1&2&2&2 \\ 1&1&1&2&2&2&2&2 \\ 0&1&1&2&2&2&3&2 \\ 0&1&1&2&2&2&2&3\end{smallmatrix} ,   \ \begin{smallmatrix}0&0&0&0&0&0&1&0 \\ 0&1&0&1&1&1&1&1 \\ 0&0&1&1&1&1&1&1 \\ 0&1&1&2&2&2&2&2 \\ 0&1&1&2&2&2&2&2 \\ 0&1&1&2&2&2&3&2 \\ 1&1&1&2&2&3&2&3 \\ 0&1&1&2&2&2&3&3\end{smallmatrix} ,   \ \begin{smallmatrix}0&0&0&0&0&0&0&1 \\ 0&1&0&1&1&1&1&1 \\ 0&0&1&1&1&1&1&1 \\ 0&1&1&2&2&2&2&2 \\ 0&1&1&2&2&2&2&2 \\ 0&1&1&2&2&2&2&3 \\ 0&1&1&2&2&2&3&3 \\ 1&1&1&2&2&3&3&2\end{smallmatrix} $$

\noindent The one given by the finite simple group $\mathrm{PSL}(2,11)$:
$$ \begin{smallmatrix}1&0&0&0&0&0&0&0 \\ 0&1&0&0&0&0&0&0 \\ 0&0&1&0&0&0&0&0 \\ 0&0&0&1&0&0&0&0 \\ 0&0&0&0&1&0&0&0 \\ 0&0&0&0&0&1&0&0 \\ 0&0&0&0&0&0&1&0 \\ 0&0&0&0&0&0&0&1\end{smallmatrix} ,   \ \begin{smallmatrix}0&1&0&0&0&0&0&0 \\ 0&0&1&1&1&0&0&0 \\ 1&0&0&0&0&0&1&1 \\ 0&0&1&1&0&1&1&1 \\ 0&0&1&0&1&1&1&1 \\ 0&0&0&1&1&1&1&1 \\ 0&1&0&1&1&1&1&1 \\ 0&1&0&1&1&1&1&1\end{smallmatrix} ,   \ \begin{smallmatrix}0&0&1&0&0&0&0&0 \\ 1&0&0&0&0&0&1&1 \\ 0&1&0&1&1&0&0&0 \\ 0&1&0&1&0&1&1&1 \\ 0&1&0&0&1&1&1&1 \\ 0&0&0&1&1&1&1&1 \\ 0&0&1&1&1&1&1&1 \\ 0&0&1&1&1&1&1&1\end{smallmatrix} ,   \ \begin{smallmatrix}0&0&0&1&0&0&0&0 \\ 0&0&1&1&0&1&1&1 \\ 0&1&0&1&0&1&1&1 \\ 1&1&1&1&2&1&2&2 \\ 0&0&0&2&1&2&2&2 \\  0&1&1&1&2&2&2&2 \\ 0&1&1&2&2&2&2&2 \\ 0&1&1&2&2&2&2&2\end{smallmatrix} ,   \ \begin{smallmatrix}0&0&0&0&1&0&0&0 \\ 0&0&1&0&1&1&1&1 \\ 0&1&0&0&1&1&1&1 \\  0&0&0&2&1&2&2&2 \\ 1&1&1&1&2&1&2&2 \\ 0&1&1&2&1&2&2&2 \\ 0&1&1&2&2&2&2&2 \\ 0&1&1&2&2&2&2&2\end{smallmatrix} ,   \ \begin{smallmatrix}0&0&0&0&0&1&0&0 \\ 0&0&0&1&1&1&1&1 \\ 0&0&0&1&1&1&1&1 \\ 0&1&1&1&2&2&2&2 \\ 0&1&1&2&1&2&2&2 \\ 1&1&1&2&2&2&2&2 \\ 0&1&1&2&2&2&3&2 \\ 0&1&1&2&2&2&2&3\end{smallmatrix} ,   \ \begin{smallmatrix}0&0&0&0&0&0&1&0 \\ 0&1&0&1&1&1&1&1 \\ 0&0&1&1&1&1&1&1 \\ 0&1&1&2&2&2&2&2 \\ 0&1&1&2&2&2&2&2 \\ 0&1&1&2&2&2&3&2 \\ 1&1&1&2&2&3&2&3 \\ 0&1&1&2&2&2&3&3\end{smallmatrix} ,   \ \begin{smallmatrix}0&0&0&0&0&0&0&1 \\ 0&1&0&1&1&1&1&1 \\ 0&0&1&1&1&1&1&1 \\ 0&1&1&2&2&2&2&2 \\ 0&1&1&2&2&2&2&2 \\ 0&1&1&2&2&2&2&3 \\ 0&1&1&2&2&2&3&3 \\ 1&1&1&2&2&3&3&2\end{smallmatrix} $$

\noindent The one (non group-like) satisfying Schur product property (on the dual):
$$ \begin{smallmatrix} 1&0&0&0&0&0&0&0 \\ 0&1&0&0&0&0&0&0 \\ 0&0&1&0&0&0&0&0 \\ 0&0&0&1&0&0&0&0 \\ 0&0&0&0&1&0&0&0 \\ 0&0&0&0&0&1&0&0 \\ 0&0&0&0&0&0&1&0 \\ 0&0&0&0&0&0&0&1\end{smallmatrix} ,   \ \begin{smallmatrix}0&1&0&0&0&0&0&0 \\ 0&0&1&1&1&0&0&0 \\ 1&0&0&0&0&0&1&1 \\ 0&0&1&0&1&1&1&1 \\ 0&0&1&1&0&1&1&1 \\ 0&0&0&1&1&1&1&1 \\ 0&1&0&1&1&1&1&1 \\ 0&1&0&1&1&1&1&1\end{smallmatrix} ,   \ \begin{smallmatrix}0&0&1&0&0&0&0&0 \\ 1&0&0&0&0&0&1&1 \\ 0&1&0&1&1&0&0&0 \\ 0&1&0&0&1&1&1&1 \\ 0&1&0&1&0&1&1&1 \\ 0&0&0&1&1&1&1&1 \\ 0&0&1&1&1&1&1&1 \\ 0&0&1&1&1&1&1&1\end{smallmatrix} ,   \ \begin{smallmatrix}0&0&0&1&0&0&0&0 \\ 0&0&1&0&1&1&1&1 \\ 0&1&0&0&1&1&1&1 \\ 1&0&0&3&1&1&2&2 \\ 0&1&1&1&1&2&2&2 \\ 0&1&1&1&2&2&2&2 \\ 0&1&1&2&2&2&2&2 \\ 0&1&1&2&2&2&2&2\end{smallmatrix} ,   \ \begin{smallmatrix}0&0&0&0&1&0&0&0 \\ 0&0&1&1&0&1&1&1 \\ 0&1&0&1&0&1&1&1 \\ 0&1&1&1&1&2&2&2 \\ 1&0&0&1&3&1&2&2 \\ 0&1&1&2&1&2&2&2 \\ 0&1&1&2&2&2&2&2 \\ 0&1&1&2&2&2&2&2\end{smallmatrix} ,   \ \begin{smallmatrix}0&0&0&0&0&1&0&0 \\ 0&0&0&1&1&1&1&1 \\ 0&0&0&1&1&1&1&1 \\ 0&1&1&1&2&2&2&2 \\ 0&1&1&2&1&2&2&2 \\  1&1&1&2&2&2&2&2 \\ 0&1&1&2&2&2&3&2 \\ 0&1&1&2&2&2&2&3\end{smallmatrix} ,   \ \begin{smallmatrix}0&0&0&0&0&0&1&0 \\ 0&1&0&1&1&1&1&1 \\ 0&0&1&1&1&1&1&1 \\ 0&1&1&2&2&2&2&2 \\ 0&1&1&2&2&2&2&2 \\ 0&1&1&2&2&2&3&2 \\ 1&1&1&2&2&3&2&3 \\ 0&1&1&2&2&2&3&3\end{smallmatrix} ,   \ \begin{smallmatrix}0&0&0&0&0&0&0&1 \\ 0&1&0&1&1&1&1&1 \\ 0&0&1&1&1&1&1&1 \\ 0&1&1&2&2&2&2&2 \\ 0&1&1&2&2&2&2&2 \\ 0&1&1&2&2&2&2&3 \\ 0&1&1&2&2&2&3&3 \\ 1&1&1&2&2&3&3&2\end{smallmatrix} $$

\item Rank $8$ and $\FPdim$ $990$,  five of type $[[1, 1], [9, 1], [10, 1], [11, 4], [18, 1]]$:
$$ \begin{smallmatrix}1&0&0&0&0&0&0&0 \\ 0&1&0&0&0&0&0&0 \\ 0&0&1&0&0&0&0&0 \\ 0&0&0&1&0&0&0&0 \\ 0&0&0&0&1&0&0&0 \\ 0&0&0&0&0&1&0&0 \\ 0&0&0&0&0&0&1&0 \\ 0&0&0&0&0&0&0&1\end{smallmatrix} ,   \ \begin{smallmatrix}0&1&0&0&0&0&0&0 \\ 1&4&0&1&1&1&1&0 \\ 0&0&1&1&1&1&1&2 \\ 0&1&1&1&1&1&1&2 \\ 0&1&1&1&1&1&1&2 \\ 0&1&1&1&1&1&1&2 \\ 0&1&1&1&1&1&1&2 \\ 0&0&2&2&2&2&2&3\end{smallmatrix} ,   \ \begin{smallmatrix}0&0&1&0&0&0&0&0 \\ 0&0&1&1&1&1&1&2 \\ 1&1&1&1&1&1&1&2 \\ 0&1&1&2&1&1&1&2 \\ 0&1&1&1&2&1&1&2 \\ 0&1&1&1&1&2&1&2 \\ 0&1&1&1&1&1&2&2 \\ 0&2&2&2&2&2&2&3\end{smallmatrix} ,   \ \begin{smallmatrix}0&0&0&1&0&0&0&0 \\ 0&1&1&1&1&1&1&2 \\ 0&1&1&2&1&1&1&2 \\ 1&1&2&0&0&2&3&2 \\ 0&1&1&0&2&2&2&2 \\ 0&1&1&2&2&1&1&2 \\ 0&1&1&3&2&1&0&2 \\ 0&2&2&2&2&2&2&4\end{smallmatrix} ,   \ \begin{smallmatrix}0&0&0&0&1&0&0&0 \\ 0&1&1&1&1&1&1&2 \\ 0&1&1&1&2&1&1&2 \\ 0&1&1&0&2&2&2&2 \\ 1&1&2&2&0&1&2&2 \\ 0&1&1&2&1&2&1&2 \\ 0&1&1&2&2&1&1&2 \\ 0&2&2&2&2&2&2&4\end{smallmatrix} ,   \ \begin{smallmatrix}0&0&0&0&0&1&0&0 \\ 0&1&1&1&1&1&1&2 \\ 0&1&1&1&1&2&1&2 \\ 0&1&1&2&2&1&1&2 \\ 0&1&1&2&1&2&1&2 \\ 1&1&2&1&2&0&2&2 \\ 0&1&1&1&1&2&2&2 \\ 0&2&2&2&2&2&2&4\end{smallmatrix} ,   \ \begin{smallmatrix}0&0&0&0&0&0&1&0 \\ 0&1&1&1&1&1&1&2 \\ 0&1&1&1&1&1&2&2 \\ 0&1&1&3&2&1&0&2 \\ 0&1&1&2&2&1&1&2 \\ 0&1&1&1&1&2&2&2 \\ 1&1&2&0&1&2&2&2 \\ 0&2&2&2&2&2&2&4\end{smallmatrix} ,   \ \begin{smallmatrix}0&0&0&0&0&0&0&1 \\ 0&0&2&2&2&2&2&3 \\ 0&2&2&2&2&2&2&3 \\ 0&2&2&2&2&2&2&4 \\ 0&2&2&2&2&2&2&4 \\ 0&2&2&2&2&2&2&4 \\ 0&2&2&2&2&2&2&4 \\ 1&3&3&4&4&4&4&5\end{smallmatrix} $$

$$ \begin{smallmatrix}1&0&0&0&0&0&0&0 \\ 0&1&0&0&0&0&0&0 \\ 0&0&1&0&0&0&0&0 \\ 0&0&0&1&0&0&0&0 \\ 0&0&0&0&1&0&0&0 \\ 0&0&0&0&0&1&0&0 \\ 0&0&0&0&0&0&1&0 \\ 0&0&0&0&0&0&0&1\end{smallmatrix} ,   \ \begin{smallmatrix}0&1&0&0&0&0&0&0 \\ 1&4&0&1&1&1&1&0 \\ 0&0&1&1&1&1&1&2 \\ 0&1&1&1&1&1&1&2 \\ 0&1&1&1&1&1&1&2 \\ 0&1&1&1&1&1&1&2 \\ 0&1&1&1&1&1&1&2 \\ 0&0&2&2&2&2&2&3\end{smallmatrix} ,   \ \begin{smallmatrix}0&0&1&0&0&0&0&0 \\ 0&0&1&1&1&1&1&2 \\ 1&1&1&1&1&1&1&2 \\ 0&1&1&2&1&1&1&2 \\ 0&1&1&1&2&1&1&2 \\ 0&1&1&1&1&2&1&2 \\ 0&1&1&1&1&1&2&2 \\ 0&2&2&2&2&2&2&3\end{smallmatrix} ,   \ \begin{smallmatrix}0&0&0&1&0&0&0&0 \\ 0&1&1&1&1&1&1&2 \\ 0&1&1&2&1&1&1&2 \\ 1&1&2&1&1&1&2&2 \\ 0&1&1&1&2&2&1&2 \\ 0&1&1&1&2&1&2&2 \\ 0&1&1&2&1&2&1&2 \\ 0&2&2&2&2&2&2&4\end{smallmatrix} ,   \ \begin{smallmatrix}0&0&0&0&1&0&0&0 \\ 0&1&1&1&1&1&1&2 \\ 0&1&1&1&2&1&1&2 \\ 0&1&1&1&2&2&1&2 \\ 1&1&2&2&1&1&1&2 \\ 0&1&1&2&1&1&2&2 \\ 0&1&1&1&1&2&2&2 \\ 0&2&2&2&2&2&2&4\end{smallmatrix} ,   \ \begin{smallmatrix}0&0&0&0&0&1&0&0 \\ 0&1&1&1&1&1&1&2 \\ 0&1&1&1&1&2&1&2 \\ 0&1&1&1&2&1&2&2 \\ 0&1&1&2&1&1&2&2 \\ 1&1&2&1&1&2&1&2 \\ 0&1&1&2&2&1&1&2 \\ 0&2&2&2&2&2&2&4\end{smallmatrix} ,   \ \begin{smallmatrix}0&0&0&0&0&0&1&0 \\ 0&1&1&1&1&1&1&2 \\ 0&1&1&1&1&1&2&2 \\ 0&1&1&2&1&2&1&2 \\ 0&1&1&1&1&2&2&2 \\ 0&1&1&2&2&1&1&2 \\ 1&1&2&1&2&1&1&2 \\ 0&2&2&2&2&2&2&4\end{smallmatrix} ,   \ \begin{smallmatrix}0&0&0&0&0&0&0&1 \\ 0&0&2&2&2&2&2&3 \\ 0&2&2&2&2&2&2&3 \\ 0&2&2&2&2&2&2&4 \\ 0&2&2&2&2&2&2&4 \\ 0&2&2&2&2&2&2&4 \\ 0&2&2&2&2&2&2&4 \\ 1&3&3&4&4&4&4&5\end{smallmatrix} $$

$$ \begin{smallmatrix}1&0&0&0&0&0&0&0 \\ 0&1&0&0&0&0&0&0 \\ 0&0&1&0&0&0&0&0 \\ 0&0&0&1&0&0&0&0 \\ 0&0&0&0&1&0&0&0 \\ 0&0&0&0&0&1&0&0 \\ 0&0&0&0&0&0&1&0 \\ 0&0&0&0&0&0&0&1\end{smallmatrix} ,   \ \begin{smallmatrix}0&1&0&0&0&0&0&0 \\ 1&4&0&1&1&1&1&0 \\ 0&0&1&1&1&1&1&2 \\ 0&1&1&1&1&1&1&2 \\ 0&1&1&1&1&1&1&2 \\ 0&1&1&1&1&1&1&2 \\ 0&1&1&1&1&1&1&2 \\ 0&0&2&2&2&2&2&3\end{smallmatrix} ,   \ \begin{smallmatrix}0&0&1&0&0&0&0&0 \\ 0&0&1&1&1&1&1&2 \\ 1&1&1&1&1&1&1&2 \\ 0&1&1&2&1&1&1&2 \\ 0&1&1&1&2&1&1&2 \\ 0&1&1&1&1&2&1&2 \\ 0&1&1&1&1&1&2&2 \\ 0&2&2&2&2&2&2&3\end{smallmatrix} ,   \ \begin{smallmatrix}0&0&0&1&0&0&0&0 \\ 0&1&1&1&1&1&1&2 \\ 0&1&1&2&1&1&1&2 \\ 1&1&2&1&1&1&2&2 \\ 0&1&1&1&1&2&2&2 \\ 0&1&1&1&2&1&2&2 \\ 0&1&1&2&2&2&0&2 \\ 0&2&2&2&2&2&2&4\end{smallmatrix} ,   \ \begin{smallmatrix}0&0&0&0&1&0&0&0 \\ 0&1&1&1&1&1&1&2 \\ 0&1&1&1&2&1&1&2 \\ 0&1&1&1&1&2&2&2 \\ 1&1&2&1&1&2&1&2 \\ 0&1&1&2&2&0&2&2 \\ 0&1&1&2&1&2&1&2 \\ 0&2&2&2&2&2&2&4\end{smallmatrix} ,   \ \begin{smallmatrix}0&0&0&0&0&1&0&0 \\ 0&1&1&1&1&1&1&2 \\ 0&1&1&1&1&2&1&2 \\ 0&1&1&1&2&1&2&2 \\ 0&1&1&2&2&0&2&2 \\ 1&1&2&1&0&3&1&2 \\ 0&1&1&2&2&1&1&2 \\ 0&2&2&2&2&2&2&4\end{smallmatrix} ,   \ \begin{smallmatrix}0&0&0&0&0&0&1&0 \\ 0&1&1&1&1&1&1&2 \\ 0&1&1&1&1&1&2&2 \\ 0&1&1&2&2&2&0&2 \\ 0&1&1&2&1&2&1&2 \\ 0&1&1&2&2&1&1&2 \\ 1&1&2&0&1&1&3&2 \\ 0&2&2&2&2&2&2&4\end{smallmatrix} ,   \ \begin{smallmatrix}0&0&0&0&0&0&0&1 \\ 0&0&2&2&2&2&2&3 \\ 0&2&2&2&2&2&2&3 \\ 0&2&2&2&2&2&2&4 \\ 0&2&2&2&2&2&2&4 \\ 0&2&2&2&2&2&2&4 \\ 0&2&2&2&2&2&2&4 \\ 1&3&3&4&4&4&4&5\end{smallmatrix} $$

$$ \begin{smallmatrix}1&0&0&0&0&0&0&0 \\ 0&1&0&0&0&0&0&0 \\ 0&0&1&0&0&0&0&0 \\ 0&0&0&1&0&0&0&0 \\ 0&0&0&0&1&0&0&0 \\ 0&0&0&0&0&1&0&0 \\ 0&0&0&0&0&0&1&0 \\ 0&0&0&0&0&0&0&1\end{smallmatrix} ,   \ \begin{smallmatrix}0&1&0&0&0&0&0&0 \\ 1&4&0&1&1&1&1&0 \\ 0&0&1&1&1&1&1&2 \\ 0&1&1&1&1&1&1&2 \\ 0&1&1&1&1&1&1&2 \\ 0&1&1&1&1&1&1&2 \\ 0&1&1&1&1&1&1&2 \\ 0&0&2&2&2&2&2&3\end{smallmatrix} ,   \ \begin{smallmatrix}0&0&1&0&0&0&0&0 \\ 0&0&1&1&1&1&1&2 \\ 1&1&1&1&1&1&1&2 \\ 0&1&1&2&1&1&1&2 \\ 0&1&1&1&2&1&1&2 \\ 0&1&1&1&1&2&1&2 \\ 0&1&1&1&1&1&2&2 \\ 0&2&2&2&2&2&2&3\end{smallmatrix} ,   \ \begin{smallmatrix}0&0&0&1&0&0&0&0 \\ 0&1&1&1&1&1&1&2 \\ 0&1&1&2&1&1&1&2 \\ 1&1&2&1&1&1&2&2 \\ 0&1&1&1&1&2&2&2 \\ 0&1&1&1&2&1&2&2 \\ 0&1&1&2&2&2&0&2 \\ 0&2&2&2&2&2&2&4\end{smallmatrix} ,   \ \begin{smallmatrix}0&0&0&0&1&0&0&0 \\ 0&1&1&1&1&1&1&2 \\ 0&1&1&1&2&1&1&2 \\ 0&1&1&1&1&2&2&2 \\ 1&1&2&1&2&1&1&2 \\ 0&1&1&2&1&1&2&2 \\ 0&1&1&2&1&2&1&2 \\ 0&2&2&2&2&2&2&4\end{smallmatrix} ,   \ \begin{smallmatrix}0&0&0&0&0&1&0&0 \\ 0&1&1&1&1&1&1&2 \\ 0&1&1&1&1&2&1&2 \\ 0&1&1&1&2&1&2&2 \\ 0&1&1&2&1&1&2&2 \\ 1&1&2&1&1&2&1&2 \\ 0&1&1&2&2&1&1&2 \\ 0&2&2&2&2&2&2&4\end{smallmatrix} ,   \ \begin{smallmatrix}0&0&0&0&0&0&1&0 \\ 0&1&1&1&1&1&1&2 \\ 0&1&1&1&1&1&2&2 \\ 0&1&1&2&2&2&0&2 \\ 0&1&1&2&1&2&1&2 \\ 0&1&1&2&2&1&1&2 \\ 1&1&2&0&1&1&3&2 \\ 0&2&2&2&2&2&2&4\end{smallmatrix} ,   \ \begin{smallmatrix}0&0&0&0&0&0&0&1 \\ 0&0&2&2&2&2&2&3 \\ 0&2&2&2&2&2&2&3 \\ 0&2&2&2&2&2&2&4 \\ 0&2&2&2&2&2&2&4 \\ 0&2&2&2&2&2&2&4 \\ 0&2&2&2&2&2&2&4 \\ 1&3&3&4&4&4&4&5\end{smallmatrix} $$

$$ \begin{smallmatrix}1&0&0&0&0&0&0&0 \\ 0&1&0&0&0&0&0&0 \\ 0&0&1&0&0&0&0&0 \\ 0&0&0&1&0&0&0&0 \\ 0&0&0&0&1&0&0&0 \\ 0&0&0&0&0&1&0&0 \\ 0&0&0&0&0&0&1&0 \\ 0&0&0&0&0&0&0&1\end{smallmatrix} ,   \ \begin{smallmatrix}0&1&0&0&0&0&0&0 \\ 1&4&0&1&1&1&1&0 \\ 0&0&1&1&1&1&1&2 \\ 0&1&1&1&1&1&1&2 \\ 0&1&1&1&1&1&1&2 \\ 0&1&1&1&1&1&1&2 \\ 0&1&1&1&1&1&1&2 \\ 0&0&2&2&2&2&2&3\end{smallmatrix} ,   \ \begin{smallmatrix}0&0&1&0&0&0&0&0 \\ 0&0&1&1&1&1&1&2 \\ 1&1&1&1&1&1&1&2 \\ 0&1&1&2&1&1&1&2 \\ 0&1&1&1&2&1&1&2 \\ 0&1&1&1&1&2&1&2 \\ 0&1&1&1&1&1&2&2 \\ 0&2&2&2&2&2&2&3\end{smallmatrix} ,   \ \begin{smallmatrix}0&0&0&1&0&0&0&0 \\ 0&1&1&1&1&1&1&2 \\ 0&1&1&2&1&1&1&2 \\ 1&1&2&2&1&1&1&2 \\ 0&1&1&1&1&2&2&2 \\ 0&1&1&1&2&1&2&2 \\ 0&1&1&1&2&2&1&2 \\ 0&2&2&2&2&2&2&4\end{smallmatrix} ,   \ \begin{smallmatrix}0&0&0&0&1&0&0&0 \\ 0&1&1&1&1&1&1&2 \\ 0&1&1&1&2&1&1&2 \\ 0&1&1&1&1&2&2&2 \\ 1&1&2&1&2&1&1&2 \\ 0&1&1&2&1&1&2&2 \\ 0&1&1&2&1&2&1&2 \\ 0&2&2&2&2&2&2&4\end{smallmatrix} ,   \ \begin{smallmatrix}0&0&0&0&0&1&0&0 \\ 0&1&1&1&1&1&1&2 \\ 0&1&1&1&1&2&1&2 \\ 0&1&1&1&2&1&2&2 \\ 0&1&1&2&1&1&2&2 \\ 1&1&2&1&1&2&1&2 \\ 0&1&1&2&2&1&1&2 \\ 0&2&2&2&2&2&2&4\end{smallmatrix} ,   \ \begin{smallmatrix}0&0&0&0&0&0&1&0 \\ 0&1&1&1&1&1&1&2 \\ 0&1&1&1&1&1&2&2 \\ 0&1&1&1&2&2&1&2 \\ 0&1&1&2&1&2&1&2 \\ 0&1&1&2&2&1&1&2 \\ 1&1&2&1&1&1&2&2 \\ 0&2&2&2&2&2&2&4\end{smallmatrix} ,   \ \begin{smallmatrix}0&0&0&0&0&0&0&1 \\ 0&0&2&2&2&2&2&3 \\ 0&2&2&2&2&2&2&3 \\ 0&2&2&2&2&2&2&4 \\ 0&2&2&2&2&2&2&4 \\ 0&2&2&2&2&2&2&4 \\ 0&2&2&2&2&2&2&4 \\ 1&3&3&4&4&4&4&5\end{smallmatrix} $$

\item Rank $8$ and $\FPdim$ $1260$,  two of type $[[1, 1], [6, 1], [7, 2], [10, 1], [15, 1], [20, 2]]$:
$$ \begin{smallmatrix}1&0&0&0&0&0&0&0 \\ 0&1&0&0&0&0&0&0 \\ 0&0&1&0&0&0&0&0 \\ 0&0&0&1&0&0&0&0 \\ 0&0&0&0&1&0&0&0 \\ 0&0&0&0&0&1&0&0 \\ 0&0&0&0&0&0&1&0 \\ 0&0&0&0&0&0&0&1\end{smallmatrix} ,   \ \begin{smallmatrix}0&1&0&0&0&0&0&0 \\ 1&0&0&0&0&1&0&1 \\ 0&0&1&0&0&1&0&1 \\ 0&0&0&1&0&1&0&1 \\ 0&0&0&0&0&0&2&1 \\ 0&1&1&1&0&2&0&2 \\ 0&0&0&0&2&0&3&2 \\ 0&1&1&1&1&2&2&1\end{smallmatrix} ,   \ \begin{smallmatrix}0&0&1&0&0&0&0&0 \\ 0&0&1&0&0&1&0&1 \\ 1&1&0&1&0&1&0&1 \\ 0&0&1&1&0&1&0&1 \\ 0&0&0&0&1&0&2&1 \\ 0&1&1&1&0&3&0&2 \\ 0&0&0&0&2&0&4&2 \\ 0&1&1&1&1&2&2&2\end{smallmatrix} ,   \ \begin{smallmatrix}0&0&0&1&0&0&0&0 \\ 0&0&0&1&0&1&0&1 \\ 0&0&1&1&0&1&0&1 \\ 1&1&1&0&0&1&0&1 \\ 0&0&0&0&1&0&2&1 \\ 0&1&1&1&0&3&0&2 \\ 0&0&0&0&2&0&4&2 \\ 0&1&1&1&1&2&2&2\end{smallmatrix} ,   \ \begin{smallmatrix}0&0&0&0&1&0&0&0 \\ 0&0&0&0&0&0&2&1 \\ 0&0&0&0&1&0&2&1 \\ 0&0&0&0&1&0&2&1 \\ 1&0&1&1&2&3&0&1 \\ 0&0&0&0&3&0&4&2 \\ 0&2&2&2&0&4&1&4 \\ 0&1&1&1&1&2&4&3\end{smallmatrix} ,   \ \begin{smallmatrix}0&0&0&0&0&1&0&0 \\ 0&1&1&1&0&2&0&2 \\ 0&1&1&1&0&3&0&2 \\ 0&1&1&1&0&3&0&2 \\ 0&0&0&0&3&0&4&2 \\ 1&2&3&3&0&6&0&4 \\ 0&0&0&0&4&0&9&4 \\ 0&2&2&2&2&4&4&5\end{smallmatrix} ,   \ \begin{smallmatrix}0&0&0&0&0&0&1&0 \\ 0&0&0&0&2&0&3&2 \\ 0&0&0&0&2&0&4&2 \\ 0&0&0&0&2&0&4&2 \\ 0&2&2&2&0&4&1&4 \\ 0&0&0&0&4&0&9&4 \\ 1&3&4&4&1&9&2&7 \\ 0&2&2&2&4&4&7&6\end{smallmatrix} ,   \ \begin{smallmatrix}0&0&0&0&0&0&0&1 \\ 0&1&1&1&1&2&2&1 \\ 0&1&1&1&1&2&2&2 \\ 0&1&1&1&1&2&2&2 \\ 0&1&1&1&1&2&4&3 \\ 0&2&2&2&2&4&4&5 \\ 0&2&2&2&4&4&7&6 \\ 1&1&2&2&3&5&6&7\end{smallmatrix} $$

$$ \begin{smallmatrix}1&0&0&0&0&0&0&0 \\ 0&1&0&0&0&0&0&0 \\ 0&0&1&0&0&0&0&0 \\ 0&0&0&1&0&0&0&0 \\ 0&0&0&0&1&0&0&0 \\ 0&0&0&0&0&1&0&0 \\ 0&0&0&0&0&0&1&0 \\ 0&0&0&0&0&0&0&1\end{smallmatrix} ,   \ \begin{smallmatrix}0&1&0&0&0&0&0&0 \\ 1&0&0&0&0&1&0&1 \\ 0&0&1&0&0&1&0&1 \\ 0&0&0&1&0&1&0&1 \\ 0&0&0&0&0&0&2&1 \\ 0&1&1&1&0&2&0&2 \\ 0&0&0&0&2&0&2&3 \\ 0&1&1&1&1&2&3&0\end{smallmatrix} ,   \ \begin{smallmatrix}0&0&1&0&0&0&0&0 \\ 0&0&1&0&0&1&0&1 \\ 1&1&0&1&0&1&0&1 \\ 0&0&1&1&0&1&0&1 \\ 0&0&0&0&1&0&2&1 \\ 0&1&1&1&0&3&0&2 \\ 0&0&0&0&2&0&3&3 \\ 0&1&1&1&1&2&3&1\end{smallmatrix} ,   \ \begin{smallmatrix}0&0&0&1&0&0&0&0 \\ 0&0&0&1&0&1&0&1 \\ 0&0&1&1&0&1&0&1 \\ 1&1&1&0&0&1&0&1 \\ 0&0&0&0&1&0&2&1 \\ 0&1&1&1&0&3&0&2 \\ 0&0&0&0&2&0&3&3 \\ 0&1&1&1&1&2&3&1\end{smallmatrix} ,   \ \begin{smallmatrix}0&0&0&0&1&0&0&0 \\ 0&0&0&0&0&0&2&1 \\ 0&0&0&0&1&0&2&1 \\ 0&0&0&0&1&0&2&1 \\ 1&0&1&1&2&3&0&1 \\ 0&0&0&0&3&0&4&2 \\ 0&2&2&2&0&4&2&3 \\ 0&1&1&1&1&2&3&4\end{smallmatrix} ,   \ \begin{smallmatrix}0&0&0&0&0&1&0&0 \\ 0&1&1&1&0&2&0&2 \\ 0&1&1&1&0&3&0&2 \\ 0&1&1&1&0&3&0&2 \\ 0&0&0&0&3&0&4&2 \\ 1&2&3&3&0&6&0&4 \\ 0&0&0&0&4&0&7&6 \\ 0&2&2&2&2&4&6&3\end{smallmatrix} ,   \ \begin{smallmatrix}0&0&0&0&0&0&1&0 \\ 0&0&0&0&2&0&2&3 \\ 0&0&0&0&2&0&3&3 \\ 0&0&0&0&2&0&3&3 \\ 0&2&2&2&0&4&2&3 \\ 0&0&0&0&4&0&7&6 \\ 1&2&3&3&2&7&2&9 \\ 0&3&3&3&3&6&9&2\end{smallmatrix} ,   \ \begin{smallmatrix}0&0&0&0&0&0&0&1 \\ 0&1&1&1&1&2&3&0 \\ 0&1&1&1&1&2&3&1 \\ 0&1&1&1&1&2&3&1 \\ 0&1&1&1&1&2&3&4 \\ 0&2&2&2&2&4&6&3 \\ 0&3&3&3&3&6&9&2 \\ 1&0&1&1&4&3&2&13\end{smallmatrix} $$

\item Rank $8$ and $\FPdim$ $1320$,  two of type $[[1, 1], [6, 2], [10, 1], [11, 1], [15, 2], [24, 1]]$:
$$ \begin{smallmatrix}1&0&0&0&0&0&0&0 \\ 0&1&0&0&0&0&0&0 \\ 0&0&1&0&0&0&0&0 \\ 0&0&0&1&0&0&0&0 \\ 0&0&0&0&1&0&0&0 \\ 0&0&0&0&0&1&0&0 \\ 0&0&0&0&0&0&1&0 \\ 0&0&0&0&0&0&0&1\end{smallmatrix} ,   \ \begin{smallmatrix}0&1&0&0&0&0&0&0 \\ 0&0&0&1&1&0&1&0 \\ 1&0&0&0&1&0&0&1 \\ 0&0&1&0&0&1&1&1 \\ 0&1&1&0&0&1&1&1 \\ 0&0&1&1&1&1&0&2 \\ 0&0&0&1&1&2&1&1 \\ 0&1&0&1&1&1&2&3\end{smallmatrix} ,   \ \begin{smallmatrix}0&0&1&0&0&0&0&0 \\ 1&0&0&0&1&0&0&1 \\ 0&0&0&1&1&1&0&0 \\ 0&1&0&0&0&1&1&1 \\ 0&1&1&0&0&1&1&1 \\ 0&0&0&1&1&1&2&1 \\ 0&1&0&1&1&0&1&2 \\ 0&0&1&1&1&2&1&3\end{smallmatrix} ,   \ \begin{smallmatrix}0&0&0&1&0&0&0&0 \\ 0&0&1&0&0&1&1&1 \\ 0&1&0&0&0&1&1&1 \\ 1&0&0&1&1&1&1&2 \\ 0&0&0&1&2&1&1&2 \\ 0&1&1&1&1&1&2&3 \\ 0&1&1&1&1&2&1&3 \\ 0&1&1&2&2&3&3&4\end{smallmatrix} ,   \ \begin{smallmatrix}0&0&0&0&1&0&0&0 \\ 0&1&1&0&0&1&1&1 \\ 0&1&1&0&0&1&1&1 \\ 0&0&0&1&2&1&1&2 \\ 1&0&0&2&2&1&1&2 \\ 0&1&1&1&1&2&2&3 \\ 0&1&1&1&1&2&2&3 \\ 0&1&1&2&2&3&3&5\end{smallmatrix} ,   \ \begin{smallmatrix}0&0&0&0&0&1&0&0 \\ 0&0&1&1&1&1&0&2 \\ 0&0&0&1&1&1&2&1 \\ 0&1&1&1&1&1&2&3 \\ 0&1&1&1&1&2&2&3 \\ 0&2&0&2&2&2&3&4 \\ 1&1&1&1&2&2&2&5 \\ 0&1&2&3&3&5&4&6\end{smallmatrix} ,   \ \begin{smallmatrix}0&0&0&0&0&0&1&0 \\ 0&0&0&1&1&2&1&1 \\ 0&1&0&1&1&0&1&2 \\ 0&1&1&1&1&2&1&3 \\ 0&1&1&1&1&2&2&3 \\ 1&1&1&1&2&2&2&5 \\ 0&0&2&2&2&3&2&4 \\ 0&2&1&3&3&4&5&6\end{smallmatrix} ,   \ \begin{smallmatrix}0&0&0&0&0&0&0&1 \\ 0&1&0&1&1&1&2&3 \\ 0&0&1&1&1&2&1&3 \\ 0&1&1&2&2&3&3&4 \\ 0&1&1&2&2&3&3&5 \\ 0&1&2&3&3&5&4&6 \\ 0&2&1&3&3&4&5&6 \\ 1&3&3&4&5&6&6&11\end{smallmatrix} $$

$$ \begin{smallmatrix}1&0&0&0&0&0&0&0 \\ 0&1&0&0&0&0&0&0 \\ 0&0&1&0&0&0&0&0 \\ 0&0&0&1&0&0&0&0 \\ 0&0&0&0&1&0&0&0 \\ 0&0&0&0&0&1&0&0 \\ 0&0&0&0&0&0&1&0 \\ 0&0&0&0&0&0&0&1\end{smallmatrix} ,   \ \begin{smallmatrix}0&1&0&0&0&0&0&0 \\ 0&0&0&1&1&1&0&0 \\ 1&0&0&0&1&0&0&1 \\ 0&0&1&0&0&1&1&1 \\ 0&1&1&0&0&1&1&1 \\ 0&0&0&1&1&1&2&1 \\ 0&0&1&1&1&0&1&2 \\ 0&1&0&1&1&2&1&3\end{smallmatrix} ,   \ \begin{smallmatrix}0&0&1&0&0&0&0&0 \\ 1&0&0&0&1&0&0&1 \\ 0&0&0&1&1&0&1&0 \\ 0&1&0&0&0&1&1&1 \\ 0&1&1&0&0&1&1&1 \\ 0&1&0&1&1&1&0&2 \\ 0&0&0&1&1&2&1&1 \\ 0&0&1&1&1&1&2&3\end{smallmatrix} ,   \ \begin{smallmatrix}0&0&0&1&0&0&0&0 \\ 0&0&1&0&0&1&1&1 \\ 0&1&0&0&0&1&1&1 \\ 1&0&0&1&1&1&1&2 \\ 0&0&0&1&2&1&1&2 \\ 0&1&1&1&1&1&2&3 \\ 0&1&1&1&1&2&1&3 \\ 0&1&1&2&2&3&3&4\end{smallmatrix} ,   \ \begin{smallmatrix}0&0&0&0&1&0&0&0 \\ 0&1&1&0&0&1&1&1 \\ 0&1&1&0&0&1&1&1 \\ 0&0&0&1&2&1&1&2 \\ 1&0&0&2&2&1&1&2 \\ 0&1&1&1&1&2&2&3 \\ 0&1&1&1&1&2&2&3 \\ 0&1&1&2&2&3&3&5\end{smallmatrix} ,   \ \begin{smallmatrix}0&0&0&0&0&1&0&0 \\ 0&0&0&1&1&1&2&1 \\ 0&1&0&1&1&1&0&2 \\ 0&1&1&1&1&1&2&3 \\ 0&1&1&1&1&2&2&3 \\ 0&0&2&2&2&2&3&4 \\ 1&1&1&1&2&2&2&5 \\ 0&2&1&3&3&5&4&6\end{smallmatrix} ,   \ \begin{smallmatrix}0&0&0&0&0&0&1&0 \\ 0&0&1&1&1&0&1&2 \\ 0&0&0&1&1&2&1&1 \\ 0&1&1&1&1&2&1&3 \\ 0&1&1&1&1&2&2&3 \\ 1&1&1&1&2&2&2&5 \\ 0&2&0&2&2&3&2&4 \\ 0&1&2&3&3&4&5&6\end{smallmatrix} ,   \ \begin{smallmatrix}0&0&0&0&0&0&0&1 \\ 0&1&0&1&1&2&1&3 \\ 0&0&1&1&1&1&2&3 \\ 0&1&1&2&2&3&3&4 \\ 0&1&1&2&2&3&3&5 \\ 0&2&1&3&3&5&4&6 \\ 0&1&2&3&3&4&5&6 \\ 1&3&3&4&5&6&6&11\end{smallmatrix} $$
\end{itemize}

\subsection{Examples of simple integral fusion rings not of Frobenius type} \label{sub:notFrob}

There is a simple integral fusion ring of rank $6$, $\FPdim$ $143$ and type $[[1,1],[4,2],[5,1],[6,1],[7,1]]$ which is not of Frobenius type. Here are its fusion matrices:
$$\begin{smallmatrix}1&0&0&0&0&0\\0&1&0&0&0&0\\0&0&1&0&0&0\\0&0&0&1&0&0\\0&0&0&0&1&0\\0&0&0&0&0&1\end{smallmatrix}  , \  \begin{smallmatrix}0&1&0&0&0&0\\1&0&1&1&1&0\\0&1&0&1&0&1\\0&1&1&1&0&1\\0&1&0&0&1&2\\0&0&1&1&2&1\end{smallmatrix}  , \  \begin{smallmatrix}0&0&1&0&0&0\\0&1&0&1&0&1\\1&0&2&0&0&1\\0&1&0&2&1&0\\0&0&0&1&2&1\\0&1&1&0&1&2\end{smallmatrix}  , \  \begin{smallmatrix}0&0&0&1&0&0\\0&1&1&1&0&1\\0&1&0&2&1&0\\1&1&2&1&0&1\\0&0&1&0&2&2\\0&1&0&1&2&2\end{smallmatrix}  , \  \begin{smallmatrix}0&0&0&0&1&0\\0&1&0&0&1&2\\0&0&0&1&2&1\\0&0&1&0&2&2\\1&1&2&2&1&1\\0&2&1&2&1&2\end{smallmatrix}  , \  \begin{smallmatrix}0&0&0&0&0&1\\0&0&1&1&2&1\\0&1&1&0&1&2\\0&1&0&1&2&2\\0&2&1&2&1&2\\1&1&2&2&2&2\end{smallmatrix}$$

 \vspace*{0.3cm}

\noindent Note that $143 = 11 \cdot 13$, so it admits no categorification because by \cite{EGO04}, any fusion category of Frobenius-Perron dimension $pq$ (with $p$,$q$ different odd primes) is group-theoretical, whereas by \cite{ENO11}, a (weakly) group-theoretical fusion category is of Frobenius type (or alternatively, cannot be both simple and non group-like).

There are $21$ simple integral fusion rings not of Frobenius type, of rank $\le 7$ and $\FPdim \le 1500$ (with $\FPdim \neq p^aq^b, pqr$, by \cite{ENO11}), and exactly $4$ ones (below) pass the Schur product criterion:

\begin{itemize}
\item rank $6$, FPdim $924 = 2^2 \cdot 3 \cdot 7 \cdot 11$, type $[[1,1],[7,1],[8,1],[12,1],[15,1],[21,1]]$ and fusion matrices:
$$\begin{smallmatrix}1&0&0&0&0&0 \\ 0&1&0&0&0&0 \\ 0&0&1&0&0&0 \\ 0&0&0&1&0&0 \\ 0&0&0&0&1&0 \\ 0&0&0&0&0&1\end{smallmatrix} , \  \begin{smallmatrix}0&1&0&0&0&0 \\ 1&0&0&1&1&1 \\ 0&0&1&1&1&1 \\ 0&1&1&1&1&2 \\ 0&1&1&1&1&3 \\ 0&1&1&2&3&3\end{smallmatrix} , \  \begin{smallmatrix}0&0&1&0&0&0 \\ 0&0&1&1&1&1 \\ 1&1&1&1&1&1 \\ 0&1&1&2&1&2 \\ 0&1&1&1&2&3 \\ 0&1&1&2&3&4\end{smallmatrix} , \  \begin{smallmatrix}0&0&0&1&0&0 \\ 0&1&1&1&1&2 \\ 0&1&1&2&1&2 \\ 1&1&2&1&3&3 \\ 0&1&1&3&3&4 \\ 0&2&2&3&4&6\end{smallmatrix} , \  \begin{smallmatrix}0&0&0&0&1&0 \\ 0&1&1&1&1&3 \\ 0&1&1&1&2&3 \\ 0&1&1&3&3&4 \\ 1&1&2&3&4&5 \\ 0&3&3&4&5&7\end{smallmatrix} , \  \begin{smallmatrix}0&0&0&0&0&1 \\ 0&1&1&2&3&3 \\ 0&1&1&2&3&4 \\ 0&2&2&3&4&6 \\ 0&3&3&4&5&7 \\ 1&3&4&6&7&10\end{smallmatrix}$$

\item rank $6$, FPdim $1320 = 2^3 \cdot 3 \cdot 5 \cdot 11$, type $[[1,1],[9,1],[10,1],[11,1],[21,1],[24,1]]$ and fusion matrices:
$$
\begin{smallmatrix}1&0&0&0&0&0 \\ 0&1&0&0&0&0 \\ 0&0&1&0&0&0 \\ 0&0&0&1&0&0 \\ 0&0&0&0&1&0 \\ 0&0&0&0&0&1\end{smallmatrix} , \  \begin{smallmatrix}0&1&0&0&0&0 \\ 1&0&0&1&1&2 \\ 0&0&1&1&1&2 \\ 0&1&1&1&1&2 \\ 0&1&1&1&3&4 \\ 0&2&2&2&4&3\end{smallmatrix} , \  \begin{smallmatrix}0&0&1&0&0&0 \\ 0&0&1&1&1&2 \\ 1&1&0&0&2&2 \\ 0&1&0&1&2&2 \\ 0&1&2&2&3&4 \\ 0&2&2&2&4&4\end{smallmatrix} , \  \begin{smallmatrix}0&0&0&1&0&0 \\ 0&1&1&1&1&2 \\ 0&1&0&1&2&2 \\ 1&1&1&1&2&2 \\ 0&1&2&2&4&4 \\ 0&2&2&2&4&5\end{smallmatrix} , \  \begin{smallmatrix}0&0&0&0&1&0 \\ 0&1&1&1&3&4 \\ 0&1&2&2&3&4 \\ 0&1&2&2&4&4 \\ 1&3&3&4&7&8 \\ 0&4&4&4&8&9\end{smallmatrix} , \  \begin{smallmatrix}0&0&0&0&0&1 \\ 0&2&2&2&4&3 \\ 0&2&2&2&4&4 \\ 0&2&2&2&4&5 \\ 0&4&4&4&8&9 \\ 1&3&4&5&9&11\end{smallmatrix}
$$

\item  rank $7$, FPdim $560 = 2^4 \cdot 5 \cdot 7$,  type  $[[1,1],[6,1],[7,2],[10,2],[15,1]]$ and fusion matrices:
$$ \begin{smallmatrix}1&0&0&0&0&0&0\\0&1&0&0&0&0&0\\0&0&1&0&0&0&0\\0&0&0&1&0&0&0\\0&0&0&0&1&0&0\\0&0&0&0&0&1&0\\0&0&0&0&0&0&1\end{smallmatrix}  , \ \begin{smallmatrix}0&1&0&0&0&0&0\\1&0&0&0&1&1&1\\0&0&1&0&1&1&1\\0&0&0&1&1&1&1\\0&1&1&1&0&1&2\\0&1&1&1&1&0&2\\0&1&1&1&2&2&2\end{smallmatrix}, \  \begin{smallmatrix}0&0&1&0&0&0&0\\0&0&1&0&1&1&1\\1&1&0&1&1&1&1\\0&0&1&1&1&1&1\\0&1&1&1&1&1&2\\0&1&1&1&1&1&2\\0&1&1&1&2&2&3\end{smallmatrix}  , \  \begin{smallmatrix}0&0&0&1&0&0&0\\0&0&0&1&1&1&1\\0&0&1&1&1&1&1\\1&1&1&0&1&1&1\\0&1&1&1&1&1&2\\0&1&1&1&1&1&2\\0&1&1&1&2&2&3\end{smallmatrix}  ,  \  \begin{smallmatrix}0&0&0&0&1&0&0\\0&1&1&1&0&1&2\\0&1&1&1&1&1&2\\0&1&1&1&1&1&2\\0&1&1&1&2&3&2\\1&0&1&1&2&2&3\\0&2&2&2&3&2&4\end{smallmatrix}  , \   \begin{smallmatrix}0&0&0&0&0&1&0\\0&1&1&1&1&0&2\\0&1&1&1&1&1&2\\0&1&1&1&1&1&2\\1&0&1&1&2&2&3\\0&1&1&1&3&2&2\\0&2&2&2&2&3&4\end{smallmatrix}  , \   \begin{smallmatrix}0&0&0&0&0&0&1\\0&1&1&1&2&2&2\\0&1&1&1&2&2&3\\0&1&1&1&2&2&3\\0&2&2&2&3&2&4\\0&2&2&2&2&3&4\\1&2&3&3&4&4&6\end{smallmatrix}  $$

\item rank $7$, FPdim $798=2 \cdot 3 \cdot 7 \cdot 19$,  type  $[[1,1],[7,1],[8,1],[9,3],[21,1]]$ and fusion matrices:
$$ \begin{smallmatrix}1&0&0&0&0&0&0\\0&1&0&0&0&0&0\\0&0&1&0&0&0&0\\0&0&0&1&0&0&0\\0&0&0&0&1&0&0\\0&0&0&0&0&1&0\\0&0&0&0&0&0&1\end{smallmatrix}  , \  \begin{smallmatrix}0&1&0&0&0&0&0\\1&0&0&1&1&1&1\\0&0&1&1&1&1&1\\0&1&1&1&1&1&1\\0&1&1&1&1&1&1\\0&1&1&1&1&1&1\\0&1&1&1&1&1&5\end{smallmatrix}  , \   \begin{smallmatrix}0&0&1&0&0&0&0\\0&0&1&1&1&1&1\\1&1&1&1&1&1&1\\0&1&1&2&1&1&1\\0&1&1&1&2&1&1\\0&1&1&1&1&2&1\\0&1&1&1&1&1&6\end{smallmatrix}  , \   \begin{smallmatrix}0&0&0&1&0&0&0\\0&1&1&1&1&1&1\\0&1&1&2&1&1&1\\1&1&2&1&1&2&1\\0&1&1&1&2&2&1\\0&1&1&2&2&1&1\\0&1&1&1&1&1&7\end{smallmatrix}  ,  \  \begin{smallmatrix}0&0&0&0&1&0&0\\0&1&1&1&1&1&1\\0&1&1&1&2&1&1\\0&1&1&1&2&2&1\\1&1&2&2&1&1&1\\0&1&1&2&1&2&1\\0&1&1&1&1&1&7\end{smallmatrix}  ,  \ \begin{smallmatrix}0&0&0&0&0&1&0\\0&1&1&1&1&1&1\\0&1&1&1&1&2&1\\0&1&1&2&2&1&1\\0&1&1&2&1&2&1\\1&1&2&1&2&1&1\\0&1&1&1&1&1&7\end{smallmatrix}  , \  \begin{smallmatrix}0&0&0&0&0&0&1\\0&1&1&1&1&1&5\\0&1&1&1&1&1&6\\0&1&1&1&1&1&7\\0&1&1&1&1&1&7\\0&1&1&1&1&1&7\\1&5&6&7&7&7&8\end{smallmatrix}  $$

\end{itemize}

\subsection{Extra perfect fusion rings} \label{ExtraPerfect}
The computer program uses several \emph{necessary} conditions for a fusion ring $\mathfrak{A}$ to be simple (and perfect, and of Frobenius type) and categorifiable, mainly collected in \cite{nat07} and \cite{ENO11}:
\begin{itemize}
\item $\FPdim(\mathfrak{A})$ not of the form $p^aq^b$ or $pqr$ , with $p$, $q$, $r$ prime,
\item $d(x_2) \ge 3$ (in particular, $m_1=1$, i.e. $\mathfrak{A}$ is perfect).
\item $s \ge 3$ (in particular rank $r \ge 3$),
\item $n_{r+1}<n_r^2$ for all $r>1$ (otherwise it trivially cannot be simple (and perfect)),
\item (Frobenius type) $n_i$ divides $\FPdim(\mathfrak{A})$ for all $i$ (idem for the fusion subrings\footnote{There are fusion rings of rank $9$, $\FPdim$ $4620$ and type $[[1, 1], [4, 2], [5, 1], [6, 1], [7, 1], [11, 1], [66, 1]]$ (Frobenius type) with a simple fusion subring of rank $8$, $\FPdim$ $264$ and type $[[1, 1], [4, 2], [5, 1], [6, 1], [7, 1], [11, 1]]$ (not Frobenius type).}),
\item $gcd(n_2, \dots , n_s) = 1$ (consequence of Frobenius type).
\end{itemize}
Now, all these necessary conditions together are not sufficient for having a simple fusion ring, so that the computer search provided also $319$ new perfect non-simple fusion rings ($4$ of which being noncommutative):


$$\begin{array}{c|c|c|c|c|c}
\# & \text{rank} & \text{FPdim}     & \text{type} & \# \text{Schur}  & \text{note} \\ \hline
1 & 7 & 7224  & [[1,1],[3,2],[6,1],[7,1],[8,1],[84,1]] & 1 & 84>8^2  \\ \hline
16 & 8 & 360 & [[1,1],[3,2],[4,1],[5,1],[10,3]]  & 6 &  \\ \hline
26 & 8 & 660 & [[1,1],[3,2],[4,1],[5,1],[10,2],[20,1]]  & 5 &  \\ \hline
24 & 8 & 960 & [[1, 1], [3, 2], [4, 1], [5, 1], [10, 1], [20, 2]]  & 14 &  \\ \hline
47 & 8 & 1260 & [[1, 1], [3, 2], [4, 1], [5, 1], [20, 3]] & 7 &   \\ \hline
1 & 8 & 1440 & [[1, 1], [3, 2], [4, 1], [5, 1], [10, 1], [16, 1], [32, 1]]  & 1 &  \\ \hline
1 & 8 & 1680 & [[1, 1], [3, 2], [4, 1], [5, 1], [12, 1], [24, 1], [30, 1]]  & 1 &  \\ \hline
2 & 8 & 2160 & [[1,1],[3,2],[4,1],[5,1],[10,1],[20,1],[40,1]]  & 2 &  \\ \hline
1 & 8 & 3120 & [[1,1],[3,2],[4,1],[5,1],[10,1],[16,1],[52,1]]]  & 1 &  \\ \hline
60 & 8 & 3360 & [[1,1],[3,2],[4,1],[5,1],[10,1],[40,2]]  & 32 &  \\ \hline
120 & 8 & 3696 & [[1,1],[3,2],[6,1],[7,1],[8,1],[42,2]]  & 40 &  \\ \hline
2 & 8 & 3960 & [[1, 1], [5, 2], [8, 2], [9, 1], [10, 1], [60, 1]]  & 1 &  \\ \hline
7 & 9 & 360 & [[1, 1], [3, 2], [4, 1], [5, 4], [15, 1]]  & 2 &  \\ \hline
11 & 9 & 420 & [[1, 1], [3, 2], [4, 1], [5, 1], [6, 2], [12, 2]] & 1 & +4 \text{NC}
\end{array}$$


Each of them contains the Grothendieck ring of $\Rep(G)$ as a proper subring (with $G=\PSL(2,q)$, $q=5,7,9$). None of them comes from a perfect group.  Ostrik's inequality holds on all of them. In the commutative case, the Schur product property (on the dual) holds on exactly $114$ ones.

Let us mention finally that we also found extra $2242$ perfect integral fusion rings out the bounds of the above table (among them, $7$ are simple, $5$ are noncommutative, none both). In the commutative case, the Schur product property (on the dual) holds on exactly $480$ of them (none of which is simple), and Ostrik's inequality holds on all of them.

$$\begin{array}{c|c|c|c|c|c}
\# & \text{rank} & \text{FPdim}     & \text{type} & \# \text{Schur}  & \text{note} \\ \hline
1 & 7 & 28392 & [[1,1],[3,2],[6,1],[7,1],[8,1],[168,1]] & 1 & 168>8^2  \\ \hline
4 & 8 & 4620 & [[1,1],[4,2],[5,1],[6,1],[7,1],[11,1],[66,1]] & 0 &   \\ \hline
1 & 8 & 5460 & [[1,1],[3,2],[4,1],[5,1],[10,1],[20,1],[70,1]] & 1 &   \\ \hline
8 & 8 & 5460 & [[1,1],[6,1],[7,2],[10,2],[15,1],[70,1]] & 1 &   \\ \hline
1 & 8 & 6960 & [[1,1],[3,2],[4,1],[5,1],[10,1],[20,1],[80,1]] & 1 &   \\ \hline
5 & 8 & 8160 & [[1,1],[3,2],[4,1],[5,1],[10,1],[40,1],[80,1]] & 5 &   \\ \hline
250 & 8 & 14280 & [[1,1],[3,2],[6,1],[7,1],[8,1],[84,2]] & 78 & 84>8^2  \\ \hline
2 & 8 & 44310 & [[1,1],[5,3],[6,1],[7,2],[210,1]] & 1 & 210>7^2 \\ \hline
1 & 9 & 1080 & [[1, 1], [5, 2], [8, 2], [9, 1], [10, 1], [12, 1], [24, 1]] & 0 &   \\ \hline
22 & 9 & 1260 & [[1,1],[3,2],[4,1],[5,1],[10,3],[30,1]] & 10 &   \\ \hline
1 & 9 & 1320  & [[1, 1], [5, 2], [6, 2], [10, 1], [11, 1], [20, 1], [24, 1]] & 0 & \text{simple}  \\ \hline
111 & 9 & 1344 & [[1,1],[3,2],[6,1],[7,1],[8,1],[14,2],[28,1]] & 26 &   \\ \hline
4 & 9 & 1512 & [[1, 1], [6, 1], [7, 1], [8, 4], [21, 1], [27, 1]] & 0 & \text{simple}  \\ \hline
52 & 9 & 1560 & [[1,1],[3,2],[4,1],[5,1],[10,2],[20,1],[30,1]] & 9 &  \\ \hline
56 & 9 & 2160 & [[1,1],[3,2],[4,1],[5,1],[20,3],[30,1]] & 17 &   \\ \hline
69 & 9 & 2160 & [[1, 1], [5, 2], [8, 2], [9, 1], [10, 1], [30, 2]] & 10 &   \\ \hline
1086 & 9 & 2520 & [[1,1],[3,2],[6,1],[7,1],[8,1],[28,3]] & 221 &   \\ \hline
5 & 9 & 2760 & [[1,1],[3,2],[4,1],[5,1],[10,2],[30,1],[40,1]] & 2 &   \\ \hline
13 & 9 & 3696 & [[1,1],[3,2],[6,1],[7,1],[8,1],[14,2],[56,1]] & 5 &   \\ \hline
16 & 9 & 3960 & [[1,1],[3,2],[4,1],[5,1],[10,3],[60,1]] & 6 &   \\ \hline
52 & 9 & 4200 & [[1, 1], [3, 2], [4, 1], [5, 1], [24, 1], [30, 2], [42, 1]] & 4 &   \\ \hline
29 & 9 & 4260 & [[1,1],[3,2],[4,1],[5,1],[10,2],[20,1],[60,1]] & 6 &   \\ \hline
24 & 9 & 4560 & [[1,1],[3,2],[4,1],[5,1],[10,1],[20,2],[60,1]] & 14  &    \\ \hline
404 & 9 & 4872 & [[1, 1], [3, 2], [6, 1], [7, 1], [8, 1], [28, 2], [56, 1]] & 56 &   \\ \hline
2 & 10 & 720 & [[1, 1], [4, 2], [5, 2], [9, 1], [10, 3], [16, 1]]  & 0 & \text{simple} \\ \hline
1 & 10 & 1200 & [[1,1],[3,2],[4,3],[5,1],[8,1],[12,1],[30,1]] & 0 &  \\ \hline
7 & 10 & 1260 & [[1,1],[3,2],[4,1],[5,4],[15,1],[30,1]] & 2 &  \\ \hline
12 & 10 & 1320 & [[1, 1], [3, 2], [4, 1], [5, 1], [6, 2], [12, 2], [30, 1]] & 3 & +3 \text{NC}    \\ \hline
3 & 10 & 1920 & [[1,1],[3,2],[4,1],[5,1],[8,2],[16,1],[24,1],[30,1]] & 1 & +2 \text{NC}
\end{array}$$

The fusion matrices of all the ($2595$) perfect integral fusion rings mentioned in this paper, together with the computer programs (written in SageMath \cite{sage}) and checks, are available in the second author's webpage \cite{FusionAtlas}.

\end{document}